\documentclass[a4paper]{cas-sc}

\usepackage[numbers]{natbib}

\usepackage{amsmath}
\usepackage{amssymb}
\usepackage{amsfonts}
\usepackage{amsthm}
\usepackage{subcaption}
\usepackage{float}
\usepackage{placeins}
\usepackage{color}
\usepackage{leftidx}
\usepackage{tikz}
\usetikzlibrary{shapes, arrows, positioning, calc, decorations.pathmorphing}

\newtheorem{thm}{Theorem}[section]
\newtheorem{cor}[thm]{Corollary}
\newtheorem{lem}[thm]{Lemma}
\newtheorem{prop}[thm]{Proposition}
\theoremstyle{definition}
\newtheorem{defn}[thm]{Definition}
\newtheorem{example}[thm]{Example}
\theoremstyle{remark}
\newtheorem{rem}[thm]{Remark}
\numberwithin{equation}{section}

\newcommand{\Real}{\mathbb R}

\begin{document}
\let\WriteBookmarks\relax
\def\floatpagepagefraction{1}
\def\textpagefraction{.001}

\shorttitle{Space-Time Duality}
\shortauthors{C.-G. Li}

\title [mode = title]{Space-Time Duality of Diffusion via Subordination: General Theory and Relativistic Models}

\author[1]{Cheng-Gang Li}[orcid=0009-0007-8190-9871]
\ead{lichenggang@swjtu.edu.cn}

\affiliation[1]{organization={School of Mathematics, Southwest Jiaotong University},
  city={Chengdu},
  postcode={611756},
  country={China}}

\begin{abstract}
The Cattaneo--Vernotte model and the relativistic Schr\"odinger operator lead, respectively, to telegraph-type and spatially nonlocal diffusion equations.
We identify a reciprocal subordination structure linking these models through the classical heat semigroup:
spatially nonlocal diffusion is obtained by Bochner subordination of the heat evolution,
whereas the heat evolution can be recovered from the telegraph-type evolution by subordination with the associated inverse subordinator.
This motivates a space--time duality, which we extend to operator families built over a bounded \(C_0\)-semigroup;
under sectoriality and univalence assumptions on the Laplace exponent, we establish existence and uniqueness of the time-dual family. Quantitative aspects are studied for two concrete exponents:
for the power law exponent the dual families have the same self-similar scaling but different long-time profiles,
whereas for the relativistic stable exponent they share the same Gaussian long-time profile with an explicit rate and the same nonrelativistic limit. Several examples and applications are also discussed.
\end{abstract}


\maketitle

\section{Background and motivation}

It is well known that the solution of the diffusion (heat) equation
\begin{equation}\label{introduction-DE}
  \begin{aligned}
    \begin{cases}
      \frac{\partial v}{\partial t}=c^2 \frac{\partial^2 v}{\partial x^2}, \quad \quad t>0, x\in \mathbb{R}; \\
      v(x,0)=u_0(x)
    \end{cases}
  \end{aligned}
  \tag{DE}
\end{equation}
has an infinite speed of propagation (we take the one-dimensional case $x\in \mathbb{R}$ for simplicity in the remaining part of this section).
To overcome this infinite-speed feature of the diffusion equation, various relativistic diffusion models have been proposed.
One of the best-known models among them is probably the telegraph equation (damped wave equation)
\begin{equation}\label{introduction-TE}
  \begin{aligned}
    \begin{cases}
      \frac{\partial^2 u}{\partial t^2}+2a\frac{\partial u}{\partial t}=c^2 \frac{\partial^2 u}{\partial x^2}, \quad \quad t>0, x\in \mathbb{R}; \\
      u(x,0)=u_0(x), \quad \partial_t u(x,0)=0
    \end{cases}
  \end{aligned}
  \tag{TE}
\end{equation}
derived from the Cattaneo--Vernotte model
\cite{Dunkel2009PRYREP,Gaveau1984,Joseph1989, Masoliver1996, Ozisik1994}.
Note that we adopt the condition $\partial_t u(x,0)=0$ just for simplicity throughout this introduction.
The solution of \eqref{introduction-TE} has finite propagation speed. Moreover, in the diffusion-phenomenon sense, the damped-wave solution of \eqref{introduction-TE} is asymptotic to the heat evolution generated by \((2a)^{-1}c^2\Delta\) as $t\rightarrow\infty$.

On the other hand, since the (free) relativistic Schr\"odinger operator $-a+\sqrt{a^2-c^2\Delta}$ was introduced to study relativistic quantum mechanics and related problems
\cite{Applebaum2009, Carmona1990JFA,Lieb1990BAMS, Lieb2001book}, the spatially nonlocal equation
\begin{equation}\label{introduction-NLE}
  \begin{aligned}
    \begin{cases}
      \dfrac{\partial f}{\partial t}=[a-\sqrt{a^2-c^2\Delta}]f, \quad \quad t>0, x\in \mathbb{R}; \\
      f(x,0)=u_0(x)
    \end{cases}
  \end{aligned}
  \tag{NLE}
\end{equation}
is regarded by some researchers as the relativistic diffusion equation because it possesses several properties motivated by relativistic considerations
\cite{Alonso2025PRE, Baeumer2010,Muniz2015PRD},
despite the fact that its solution still has an infinite speed of propagation in the framework of diffusion processes.

  \emph{This paper develops a space--time duality between \eqref{introduction-TE} and \eqref{introduction-NLE}, in which the normal diffusion equation \eqref{introduction-DE} serves as a bridge. We also explore their generalizations and applications: the Laplacian in these equations is replaced by a general bounded $C_0$-semigroup generator $A$, and the concrete symbol by a general Laplace exponent $\varphi$, leading to a dual pair of abstract equations that are nonlocal-in-time on one side and nonlocal-in-space on the other. }

The motivation has two sources: their common origin through analytic continuation and a chain of subordination relations.

\textbf{First motivation.} In a seminal paper \cite{Gaveau1984}, the authors remarked:
``For the nonrelativistic electron there is a stochastic process that is intimately connected with its propagation, namely, the Wiener process or Brownian
motion. This is most clearly reflected in the connection between the Feynman path integral and the
Wiener integral. The two kinds of sum, like their underlying partial differential equations,
are related by an analytic continuation that can be implemented in a variety of ways: imaginary time, imaginary mass, and even imaginary $\hbar$.''

For the free Schr\"odinger equation
\[
  i\hbar\dfrac{\partial \Psi}{\partial t}=-\dfrac{\hbar^2}{2m} \Delta \Psi,
\]
a diffusion equation
\[
  \hbar\dfrac{\partial \Psi}{\partial t}=\dfrac{\hbar^2}{2m} \Delta \Psi
\]
can be obtained by using any of the three analytic continuations:
\[
  t\rightarrow it, \quad m\rightarrow im,  \quad \hbar\rightarrow -i\hbar.
\]

Consider the relativistic Schr\"odinger equation
\begin{equation}\label{introduction-RSE}
  i\hbar\dfrac{\partial \Psi}{\partial t}=[-mc^2+\sqrt{m^2c^4-\hbar^2c^2\Delta}]\Psi.
  \tag{RSE}
\end{equation}
Let $a=mc^2/\hbar$.
After the indicated analytic continuation $m\rightarrow im$ or $\hbar\rightarrow -i\hbar$, one obtains a first-order square-root equation; squaring it formally yields the telegraph equation \eqref{introduction-TE}.
However, once the analytic continuation $t\rightarrow it$ is employed,
Baeumer et al. \cite{Baeumer2010} obtained the spatially nonlocal equation \eqref{introduction-NLE}.
We will review these analytic continuations in Appendix A.

The diffusion equation \eqref{introduction-DE} and the nonlocal equation \eqref{introduction-NLE} are closely linked:
\begin{equation}\label{introduction-space-subordination}
  f(x,t)=\int_0^\infty q_\mu(s,t) v(x,s)ds, \quad t>0, x\in \mathbb{R}
\end{equation}
where $q_\mu(s,t)$ is the probability density of the subordinator with Laplace exponent $\mu(\lambda)=-a+\sqrt{\lambda+a^2}$.
This subordination relation can be interpreted directly in terms of stochastic processes \cite{Applebaum2009,Schilling2012}:
Let $B$ be Brownian motion with generator $c^2\Delta$, and let $\mathbf{D}_\mu(t)$ be a subordinator with Laplace exponent $\mu(\cdot)$ that is independent of $\mathbf{B}(t)$.
(Throughout this paper, a Brownian motion is understood to have variance rate $2$, i.e., generator
$\Delta$, unless the word {\it standard} is added; `with respect to \eqref{introduction-DE}' then
means variance rate $2c^{2}$, i.e., generator $c^{2}\Delta$.)
Then the probability density of the subordinated process $\mathbf{B}(\mathbf{D}_\mu(t))$ is the fundamental solution of \eqref{introduction-NLE}.
In other words, $a-\sqrt{a^2-c^2\Delta}$ is the generator of the subordinated process $\mathbf{B}(\mathbf{D}_\mu(t))$. See Figure 1.
The link between \eqref{introduction-TE} and \eqref{introduction-DE} is the content of the second motivation.

\tikzset{
  motivation/.style={
    >=stealth,
  },
  motivation box/.style={
    draw,
    rectangle,
    minimum width=1.20cm,
    minimum height=0.55cm,
    inner sep=2pt,
    font=\footnotesize,
  },
  motivation label/.style={
    font=\scriptsize,
    fill=white,
    inner sep=1pt,
  },
}

\begin{figure}
  \centering

  \begin{subfigure}[b]{0.45\textwidth}
    \centering
    \begin{tikzpicture}[motivation]

      \node[motivation box] (rse) at ( 0,  1.0) {RSE};
      \node[motivation box] (te)  at (-2,  0.0) {TE};
      \node[motivation box] (nle) at ( 2,  0.0) {NLE};
      \node[motivation box] (de)  at ( 0, -1.0) {DE};

      \draw[->]
        (rse.south west) --
        node[motivation label, sloped, above] {(c)}
        (te.north east);

      \draw[->]
        (rse.south east) --
        node[motivation label, sloped, above] {(d)}
        (nle.north west);

      \draw[->, dashed]
        (te.south east) --
        node[motivation label, sloped, below] {(1.4)}
        (de.north west);

      \draw[->]
        (de.north east) --
        node[motivation label, sloped, below] {(1.1)}
        (nle.south west);

    \end{tikzpicture}

    \caption{First motivation}
    \label{fig:first-motivation}
  \end{subfigure}
  \hfill
  \begin{subfigure}[b]{0.45\textwidth}
    \centering
    \begin{tikzpicture}[motivation]

      \node[motivation box] (we) at ( 0, 1.0) {WE};
      \node[motivation box] (te) at (-2, 0.0) {TE};
      \node[motivation box] (de) at ( 2, 0.0) {DE};

      \draw[->]
        (we.south west) --
        node[motivation label, sloped, above] {(1.3)}
        (te.north east);

      \draw[->]
        (we.south east) --
        node[motivation label, sloped, above] {(1.2)}
        (de.north west);

      \draw[->, dashed]
        (te.east) --
        node[motivation label, above] {(1.4)}
        (de.west);

    \end{tikzpicture}

    \caption{Second motivation}
    \label{fig:second-motivation}
  \end{subfigure}

  \caption{
    Panel~(a) summarizes the analytic-continuation links from the
    relativistic Schr\"odinger equation to the telegraph and spatially
    nonlocal equations, together with the subordination chain
    \emph{TE}$\to$\emph{DE}$\to$\emph{NLE}.
    Arrow~(c) denotes $m\to im$ or $\hbar\to-i\hbar$, and arrow~(d)
    denotes $t\to it$.
    Panel~(b) summarizes the subordination relations from the wave equation
    to the telegraph and diffusion equations and the proposed
    telegraph-to-diffusion relation.
  }
  \label{fig:motivations}
\end{figure}

\textbf{Second motivation.} Let us introduce a wave equation:
\begin{equation} \label{introduction-WE}
  \begin{aligned}
    \begin{cases}
      \frac{\partial^2 w}{\partial t^2}= c^2\frac{\partial^2 w}{\partial x^2}, \quad \quad t>0, x\in \mathbb{R}; \\
      w(x,0)=u_0(x), \quad \partial_t w(x,0)=0.
    \end{cases}
  \end{aligned}
  \tag{WE}
\end{equation}

It is well known that the following integral representation holds:
\begin{equation}\label{introduction-weierstrass}
  v(x,t)=\int_0^\infty p_\zeta(s,t) w(x,s)ds, \quad    t>0, x\in \mathbb{R}
\end{equation}
where $p_\zeta(s,t)=\frac{1}{\sqrt{\pi t}} e^{-\frac{s^2}{4t}}$.
Indeed, the formula \eqref{introduction-weierstrass} is a variant of the elementary convolution formula for the diffusion equation \eqref{introduction-DE};
this formula and its abstract version are known as the Gauss-Weierstrass (or Weierstrass) formulas;
see \cite{Arendt2010book, Romanov1947} for \eqref{introduction-weierstrass} and its abstract version.
The function $p_\zeta(s,t)$ has a probabilistic interpretation:
it is the probability density of the inverse $1/2$-stable subordinator \cite{Meerschaert2019book}
or of the absolute value of a one-dimensional Brownian motion $|\mathbf{B}(t)|$ \cite{Applebaum2009}.
This establishes a physically meaningful subordination relation from the wave process to the diffusion process.

A subordination from \eqref{introduction-WE} to \eqref{introduction-TE} is also known. We have
\[
  u(x,t)=\mathbb{E}[w(x, \mathbf{U}(t))]=\frac{1}{2}\mathbb{E}[u_0(x+c\mathbf{U}(t))+u_0(x-c\mathbf{U}(t))]
\]
where
\[
  \mathbf{U}(t)=\int_0^t (-1)^{\mathbf{N}(s)}ds,
\]
and $\{\mathbf{N}(t)\}_{(t\geq 0)}$ is a Poisson process with rate $a$; see \cite{Bobrowski2005book, Dewitt1989, Kac1974}.
This classical representation also admits the following variant:
\begin{equation}\label{introduction-Li}
  u(x,t)=\int_0^{\infty} p_\eta(s,t)w(x,s)ds, \quad t>0, x\in \mathbb{R}
\end{equation}
in which the generalized density $ p_\eta(s,t)$ has the Laplace transform \cite[Formula (20)]{Dewitt1989}
\[
  \int_0^{\infty} e^{-\lambda t}p_\eta(s,t)dt=\dfrac{\sqrt{\lambda^2+2a\lambda}}{\lambda}e^{-s\sqrt{\lambda^2+2a\lambda}}.
\]
Let $\mathbf{E}_\eta(t)$ be the inverse subordinator with Laplace exponent $\eta(\lambda)=\sqrt{\lambda^2+2a\lambda}$.
Then $p_\eta(s,t)$ can be regarded as the generalized density of $\mathbf{E}_\eta(t)$, giving a probabilistic interpretation of the displayed subordination formula;
see Example \ref{four-BF-example} and Example \ref{section-3-example-1} for the exact interpretation of $p_\eta(s,t)$ and \eqref{introduction-Li}.

Since we have the subordination from \eqref{introduction-WE} to \eqref{introduction-DE}, and \eqref{introduction-WE} to \eqref{introduction-TE}, a natural question arises:
as \eqref{introduction-TE} lies between \eqref{introduction-WE} and \eqref{introduction-DE}, is there any subordination relation
\begin{equation} \label{introduction-TEtoDE}
  v(x,t)=\int_0^{\infty} p(s,t)u(x,s)ds, \quad t>0, x\in \mathbb{R}
\end{equation}
from \eqref{introduction-TE} to \eqref{introduction-DE}?
Once the subordination relation \eqref{introduction-TEtoDE} is established, we obtain a chain of subordinations:
\[
  \text{(TE)}\mapsto \text{(DE)}\mapsto\text{(NLE)}.
\]
See Figure 1.


In the aforementioned subordination relations, the key observation of this paper is that
the function $p(s,t)$ in the subordination  formula \eqref{introduction-TEtoDE} can be regarded as the probability density of the \emph{inverse} subordinator $\mathbf{E}_\mu(t)$ with Laplace exponent $\mu(\lambda)=-a+\sqrt{\lambda+a^2}$.
To sum up, let $\mathbf{H}(t), \mathbf{B}(t)$ and $\mathbf{Y}(t)$ be the processes governed by \eqref{introduction-TE}, \eqref{introduction-DE} and \eqref{introduction-NLE}, respectively,
and let $\mathbf{D}_\mu(t)$ be independent of $\mathbf{H}(t)$ and $\mathbf{B}(t)$. Then we have
\[
 \mathbf{B}(t)\ \overset{d}=\ \mathbf{H}(\mathbf{E}_\mu(t)),
 \qquad
 \mathbf{Y}(t)\ \overset{d}=\ \mathbf{B}(\mathbf{D}_\mu(t)),
\]
where $\overset{d}=$ denotes equality in distribution, and $\mathbf{E}_\mu(t)$ is the inverse of $\mathbf{D}_\mu(t)$.
In view of the duality between $\mathbf{D}(t)$ and $\mathbf{E}(t)$:
$\mathbf{E}(t)=\inf\{z>0:\mathbf{D}(z)>t\}$ and $\mathbb{P}(\mathbf{E}(t)\leq z)=\mathbb{P}(\mathbf{D}(z)> t)$,
one may see that \eqref{introduction-TE} and  \eqref{introduction-NLE} have an intimate dual relation.

The paper is organized as follows. Section~2 reviews Bernstein functions, subordinators, inverse subordinators,
and the functional-analytic subordination tools used throughout. Section~3 proves the telegraph-to-diffusion subordination formula and explains the duality between the telegraph and relativistic nonlocal equations. Section~4 introduces generalized dual diffusion families relative to a bounded $C_0$-semigroup, gives sufficient conditions for existence and uniqueness of the time-dual family, and discusses representative examples. Section~5 studies quantitative duality. For the power law $\varphi(\lambda)=\lambda^\alpha$, the two families have the same self-similar scaling but different profiles. For the relativistic stable exponent, both converge to the same Gaussian profile at long times, and under the stated $a$--$c$ coupling both converge to the heat equation in the nonrelativistic limit. Appendices~A and~B contain the analytic-continuation derivation and auxiliary calculations, respectively.

\section{Preliminaries: Bernstein functions, subordinators and subordination principles}

This section introduces the required tools from probability and functional analysis, and fixes notation used throughout this paper.

\subsection{Bernstein function, subordinator and inverse subordinator}

First, we recall the theory of Bernstein functions and subordinators \cite{Applebaum2009,Schilling2012}.

\begin{defn}
  A function \(f:(0,\infty)\to\mathbb R\) is {\it completely monotone} if \(f\in C^\infty\) and
  \[
    (-1)^n f^{(n)}(\lambda)\geq 0
  \]
  for all \(n=0,1,2,\ldots\) and $\lambda >0$. Moreover, $\varphi(\cdot)$ is a {\it Bernstein function} if $\varphi(\cdot)\geq0$ and $\varphi'(\cdot)$ is a completely monotone function.
\end{defn}

We write $\mathcal{CMF}$ and $\mathcal{BF}$ for completely monotone functions and Bernstein functions, respectively.
It is easy to see that
$e^{-a\lambda^\alpha}(a>0, 0\leq \alpha\leq 1)$ and $\lambda^{-\alpha}(\alpha>0)$ are completely monotone, whereas
$\lambda^{\alpha}$  $(0<\alpha\leq 1)$, $1-e^{-\lambda}$ and $\ln(1+\lambda)$ are Bernstein functions.

  The following representation theorem is central to the probabilistic interpretation of $\mathcal{CMF}$ and $\mathcal{BF}$.

  \begin{prop} \label{LK bernstein function}
    (1) Let $f$ be $\mathcal{CMF}$.
    Then it is the Laplace transform of a unique (non-negative) measure $\nu$ on $[0,+\infty)$, i.e., for all $\lambda >0$,
    \[
      f(\lambda)=\mathcal{L} (\nu;\lambda)=\int_{[0,\infty)}e^{-\lambda t}\nu(dt).
    \]
    Conversely, whenever $ \mathcal{L}(\nu;\lambda)<\infty$ for every $\lambda >0$, $ \lambda \rightarrow\mathcal{L}(\nu;\lambda)$ is
    $\mathcal{CMF}$.

    (2) A function $\varphi(\cdot): (0, +\infty)\rightarrow (0, +\infty)$ is $\mathcal{BF}$ if and only if it
    admits the representation
    \[
      \varphi(\lambda)=a+b\lambda+\int_{[0,\infty)}(1-e^{-\lambda t})\sigma(dt)
    \]
    where $a,b\geq 0$ and $\sigma$ is a measure on $[0,\infty)$ satisfying
    $\int_{[0,\infty)}(1\wedge t)\sigma(dt)<\infty$. Furthermore, we have
    \[
      a=\lim_{\lambda\to0+} \varphi(\lambda), \qquad b=\lim_{\lambda\to\infty} \frac{\varphi(\lambda)}{\lambda}.
    \]
    We will call $\sigma$ the L{\'e}vy measure.

  \end{prop}

  The following are some closure properties of $\mathcal{CMF}$ and $\mathcal{BF}$; see \cite{Schilling2012}.
  \begin{prop} \label{operation properties}
    \begin{itemize}
      \item[(i)]
            If $f$ and $g$ are $\mathcal{CMF}$, then $f+g$ and $fg$ are $\mathcal{CMF}$.
      \item[(ii)]
            If $f$ and $g$ are $\mathcal{BF}$, then $f+g$ and $f\circ g$ are $\mathcal{BF}$.
      \item[(iii)]
            If $f$ is $\mathcal{CMF}$ and $g$ is $\mathcal{BF}$, then $f\circ g$ is $\mathcal{CMF}$.
      \item[(iv)]
            If $f$ is $\mathcal{BF}$, then $f(\lambda)/\lambda$ is $\mathcal{CMF}$.
    \end{itemize}
  \end{prop}

The following subclass of $\mathcal{BF}$ is useful \cite{Schilling2012}.

\begin{defn}
A $\mathcal{BF}$ $f$ is said to be a {\it complete Bernstein function} if
its L{\'e}vy measure $\mu$ has a completely monotone density $m(t)$ with respect to
Lebesgue measure:
\[
f(\lambda)=a+b\lambda+\int_{(0,\infty)}(1-e^{-\lambda t})m(t)dt.
\]
\end{defn}

We will use $\mathcal{CBF}$ to denote complete Bernstein functions.
For example, $\lambda^{\alpha}$ ($0\leq \alpha\leq 1$) and $\ln(1+\lambda)$ are $\mathcal{CBF}$, while
 $1-e^{-\lambda}$ is not $\mathcal{CBF}$.

There are several equivalent characterizations for $\mathcal{CBF}$ \cite{Schilling2012}. In particular, the complete Bernstein functions have useful holomorphic extensions.

\begin{prop}\label{charac to CBF}
Suppose that $f$ is a non-negative function on $(0,\infty)$. Then the following
conditions are equivalent:
\begin{itemize}
\item[(i)]
 $f$ is $\mathcal{CBF}$.
\item[(ii)]
 $f$ has a holomorphic extension to $H^+=\{z\in \mathbb{C}; \operatorname{Im}\,z>0 \}$ such that $\operatorname{Im}\,f(z)\geq 0$ for all $z\in H^+$ and
such that the limit $f(0+)=\lim_{(0,\infty)\ni \lambda\rightarrow0} f(\lambda)$ exists and is real.
\item[(iii)]
$f$ has a holomorphic extension to the cut complex plane $\mathbb{C}\setminus(-\infty,0]$ such that $\operatorname{Im}\,z\cdot \operatorname{Im}\,f(z)\geq 0$ for every $z\in\mathbb{C}\setminus(-\infty,0]$, and
such that the limit $f(0+)=\lim_{(0,\infty)\ni \lambda\rightarrow0} f(\lambda)$ exists and is real.
\end{itemize}
\end{prop}

  The next result, new and of independent interest, will be useful in Section 4.

  \begin{lem}\label{Section 2-Bernstein-new-result}
    Suppose that functions $f(\cdot),\varphi(\cdot): (0,+\infty)\rightarrow (0,+\infty)$ are infinitely differentiable,
    $f$ is $\mathcal{CMF}$, $\varphi'(\lambda)=f(\varphi(\lambda))$. Then $\varphi$ is $\mathcal{BF}$.

  \end{lem}
  \begin{proof}
    Since $\varphi$ is positive and $\varphi'(\lambda)=f(\varphi(\lambda))$, it suffices to show that
    $\varphi'$ is completely monotone, i.e., that
    \[
      (-1)^{m}(f\circ\varphi)^{(m)}(\lambda)\ge0,\qquad \lambda>0,\ m\ge0.
    \]
    We prove this by induction on $m$. The case $m=0$ is $f(\varphi(\lambda))>0$. Assume that
    $(-1)^{j-1}\varphi^{(j)}(\lambda)\ge0$ for $1\le j\le m$. By Fa{\'a} di Bruno's formula,
    \[
      (f\circ\varphi)^{(m)}(\lambda)
      =\sum\frac{m!}{b_1!\cdots b_m!}\,f^{(k)}(\varphi(\lambda))
      \prod_{j=1}^{m}\Bigl(\frac{\varphi^{(j)}(\lambda)}{j!}\Bigr)^{b_j},
    \]
    where the sum extends over all nonnegative integers $b_1,\dots,b_m$ with
    $b_1+2b_2+\cdots+mb_m=m$ and $k:=b_1+\cdots+b_m$.     In every summand, the complete monotonicity of $f$ gives
    $(-1)^k f^{(k)}(\varphi(\lambda))\ge0$, and the induction hypothesis gives
    $(-1)^{j-1}\varphi^{(j)}(\lambda)\ge0$ for each $j$; hence each summand equals $(-1)^m$
    times a nonnegative factor (possibly zero), because its sign factor is
    \[
      (-1)^k\prod_{j=1}^{m}(-1)^{(j-1)b_j}=(-1)^{k+\sum_{j}(j-1)b_j}=(-1)^{m},
      \qquad\text{since}\quad
      \sum_{j=1}^{m}(j-1)b_j=\sum_{j=1}^{m}jb_j-\sum_{j=1}^{m}b_j=m-k.
    \]
    Summing over all $b_1,\dots,b_m$ therefore yields
    $(-1)^m(f\circ\varphi)^{(m)}(\lambda)\ge0$, which completes the induction. Thus $\varphi'$ is
    completely monotone, and $\varphi$ is a Bernstein function.
  \end{proof}

  We next recall the probabilistic preliminaries needed below.

  \begin{defn} A {\it subordinator} $\mathbf{D}(t)$
    is a one-dimensional L{\'e}vy process taking values in $[0,\infty)$ that is nondecreasing (a.s.) and starts at the origin almost surely.
  \end{defn}

  As a special L{\'e}vy process, a subordinator has a simple form in the L{\'e}vy--Khintchine representation.

  \begin{thm}\label{LK subordinator}
    If $\mathbf{D}(t)$ is a subordinator, then the Laplace transform of the distribution is
    \[
      \mathbb{E}(e^{-\lambda \mathbf{D}(t)})=e^{-t\varphi(\lambda)},
    \]
    and the Laplace exponent $\varphi$ takes the form
    \begin{equation}\label{LK formula subordinator}
      \varphi(\lambda)=b\lambda+\int_{(0,\infty)}(1-e^{-\lambda s})\sigma(ds)
    \end{equation}
    where $b\geq 0$ and the measure $\sigma$ satisfies the same condition as in Proposition \ref{LK bernstein function}.

    Conversely, any function $\varphi:(0,\infty)\to[0,\infty)$ of the form \eqref{LK formula subordinator} is the Laplace exponent of a subordinator.
  \end{thm}

  \begin{rem}\label{relation}
    There is a one-to-one correspondence between Bernstein functions satisfying
    $\lim_{\lambda\rightarrow 0+}\varphi(\lambda)=0$ and subordinators \cite{Applebaum2009}:
    given any Bernstein function $\varphi(\cdot)$ with $\varphi(0+)=0$,
    we can associate it with a subordinator $\mathbf{D}_{\varphi}(t)$ with  Laplace exponent $\varphi(\cdot)$;
    conversely, given any subordinator $\mathbf{D}_{\varphi}(t)$, we can conclude that the Laplace exponent $\varphi$ is a
    $\mathcal{BF}$ for which $\lim_{\lambda\rightarrow 0+}\varphi(\lambda)=0$.

    We will make no distinction between Bernstein functions with $\varphi(0+)=0$ and Laplace exponents hereafter.
    Thus, a Laplace exponent is nothing but a Bernstein function $\varphi$ with $\varphi(0+)=0$.
    In the sequel we use the term {\it Laplace exponent} in probabilistic contexts (subordinators,
    inverse subordinators and their distributions) and {\it Bernstein function} in analytic contexts (functional calculi and sector conditions);
    both terms refer to the same object.

  \end{rem}

  Throughout this paper, we will use $\mathbf{D}_{\varphi}(t)$ to denote the subordinator with Laplace exponent $\varphi(\cdot)$. Define the first passage time of $\mathbf{D}_{\varphi}(t)$ as
  \begin{equation}\label{inverse subordinator definition}
    \mathbf{E}_{\varphi}(t)=\inf\{z>0: \mathbf{D}_{\varphi}(z)>t\}.
  \end{equation}
  We will also call $\mathbf{E}_{\varphi}(t)$ the {\it inverse subordinator} with  Laplace exponent $\varphi(\cdot)$. From \eqref{inverse subordinator definition}, we have $\mathbb{P}(\mathbf{E}_{\varphi}(t)> s)=\mathbb{P}(\mathbf{D}_{\varphi}(s)\leq t)$.
  Further, we always use $q_{\varphi}(s,t)$ and $p_{\varphi}(s,t)$ to denote the probability densities (Lebesgue density) of
$\mathbf{D}_{\varphi}(t)$ and $\mathbf{E}_{\varphi}(t)$ respectively, if these densities exist.

For a driftless subordinator $\mathbf{D}_{\varphi}(t)$ with Laplace exponent $\varphi$, when the Hartman--Wintner condition
\[
\lim\limits_{\lambda\to\infty}\frac{\varphi(\lambda)}{\ln(1+\lambda)}=\infty
\]
is satisfied, $\mathbf{D}_{\varphi}(t)$ has a bounded density $q_{\varphi}(s,t)$ for each $t>0$; see \cite{ChenZhenQing2020JFA}.

  Suppose $\varphi$ is a driftless Laplace exponent whose L{\'e}vy measure $\sigma$ satisfies $\sigma(0,\infty)=\varphi(\infty)=\infty$.
  Then, for almost every $t>0$ (i.e., possibly excluding a set of $t$ of Lebesgue measure zero), $\mathbf{E}_{\varphi}(t)$ admits a Lebesgue density $p_{\varphi}(s,t)$; see \cite[Lemma 2.1]{ChenZhenQing2020JFA} or \cite[Theorem 3.1]{Meerschaert2008}.

  The following result is standard.

 \begin{prop}
 Let $\mathbf{E}_{\varphi}(t)$ be the inverse subordinator with Laplace exponent $\varphi$.
Then for any $\lambda>0$ and any bounded measurable function $F:[0,\infty)\to\mathbb{R}$,
\[
\int_{0}^{\infty} e^{-\lambda t}\, \mathbb{E}\bigl[F(\mathbf{E}_{\varphi}(t))\bigr]\,dt
= \frac{\varphi(\lambda)}{\lambda} \int_{0}^{\infty} F(s)\, e^{-s\varphi(\lambda)}\,ds.
\]

Specifically, suppose $\mathbf{E}_{\varphi}(t)$ has density $p_\varphi(\cdot,t)$.
    Then the Laplace transform of $p_{\varphi}(s,t)$ in $t$ is
    \[
      \int_0^{+\infty} e^{-\lambda t} p_{\varphi}(s,t)dt=\dfrac{{\varphi}(\lambda)}{\lambda}e^{-s{\varphi}(\lambda)}.
    \]
 \end{prop}



\begin{example}\label{four-BF-example}
  The following four functions are $\mathcal{BF}$:
  \[
    \begin{aligned}
      \mu(\lambda)   & =-a+\sqrt{\lambda+a^2},             \\
      \zeta(\lambda) & =\sqrt{\lambda},                    \\
      \kappa(\lambda)  & =-a+(\lambda+a^{1/\beta})^{\beta},   \\
      \eta(\lambda)&=\sqrt{\lambda^2+2a\lambda}
    \end{aligned}
  \]
  where $\lambda>0, a\geq 0, \beta\in(0,1)$.
  For each of the Laplace exponents $\mu,\zeta,\kappa$, the Hartman--Wintner condition is satisfied, hence the corresponding subordinator has a bounded density.
  Moreover, the L{\'e}vy measures of $\mu,\zeta,\kappa$ all satisfy $\sigma(0,\infty)=\infty$. Combined with the continuity in $t$, this implies that their inverse subordinators admit Lebesgue densities for each $t>0$
\cite{Meerschaert2008}.

  Moreover, for $\eta(\lambda)=\sqrt{\lambda^2+2a\lambda} (a>0)$, the drift coefficient and total L{\'e}vy mass are
  \[
  b_\eta=\lim_{\lambda\to\infty} \frac{\eta(\lambda)}{\lambda}=1,\qquad \sigma_\eta(0,\infty)=\lim_{\lambda\to\infty}(\eta(\lambda)-\lambda)=a.
  \]
  Thus, the corresponding subordinator is a unit-drift process with finite jump activity.
  Therefore the law of the inverse subordinator $\mathbf{E}_{\eta}(t)$ has a singular endpoint component and does not admit a Lebesgue density.
  However, for notational consistency, we still use $p_\eta(s,t)$
  to denote the generalized density of $\mathbf{E}_\eta(t)$ in Remark \ref{Section-3-rem-0} and Example \ref{section-3-example-1}.
  That is, for each $t>0$, $p_\eta(s,t)$ is a generalized function with a singular $\delta$ component, satisfying
  \[
   \int_0^\infty e^{-\lambda t}p_\eta(s,t)dt=\frac{\eta(\lambda)}{\lambda}e^{-s\eta(\lambda)}.
  \]
  More precisely, integrals involving $p_\eta(s,t)$ in this paper are always understood in the sense of
  measures: with $P_\eta(ds,t)$ denoting the law of $\mathbf{E}_\eta(t)$,
  \[
    \int_0^\infty p_\eta(s,t)h(s)\,ds:=\int_{[0,\infty)}h(s)\,P_\eta(ds,t)=\mathbb{E}\bigl[h(\mathbf{E}_\eta(t))\bigr]
  \]
  for every bounded measurable function $h$; in particular, the atomic term $e^{-at}\delta(t-s)$ acts by
  evaluation at the endpoint $s=t$, i.e.,
  \[
    \int_0^\infty p_\eta(s,t)h(s)\,ds=e^{-at}h(t)+\int_0^{t}p_\eta^{(c)}(s,t)h(s)\,ds,
  \]
  where $p_\eta^{(c)}(s,t)$ denotes the continuous (Bessel-type) part of $p_\eta(s,t)$ displayed below. The
  same convention applies when $p_\eta$ appears as an inner kernel: for instance, in the composition formula
  \eqref{Section-3-example-1} the term $e^{-a\tau}\delta(\tau-s)$ contributes the integrand evaluated at
  $\tau=s$.
  Indeed, we have the explicit expression \cite[Formula (23)]{Dewitt1989}:
  \[
  p_\eta(s,t)=e^{-at}\delta(t-s)+ae^{-at}\theta(t-s)
             \left[I_0(a(t^2-s^2)^{1/2})+\frac{t}{(t^2-s^2)^{1/2}} I_1(a(t^2-s^2)^{1/2})     \right]
  \]
  where $\delta$ is the Dirac $\delta$ function, $\theta$ is the Heaviside function,
  \[
   \theta(t)=
   \begin{cases}
    1, \qquad t\geq 0;\\
    0, \qquad t< 0,
   \end{cases}
  \]
  and $I_0, I_1$ are the modified Bessel functions:
  \[
  I_0(z)=\sum_{k=0}^\infty \frac{(z/2)^{2k}}{(k!)^2}, \qquad I_1(z)=I_0'(z).
  \]

  The notation for the Laplace exponents $\mu,\zeta,\kappa,\eta$ is fixed throughout this paper.

\end{example}

  Next, we focus on the properties of
  \[
    \mu(\lambda)=-a+\sqrt{\lambda+a^2}, \quad \quad \lambda>0.
  \]
  There exists a subordinator $\mathbf{D}_{\mu}(t)$ such that
 \[
    \mathbb{E}\left(  e^{-\lambda \mathbf{D}_{\mu}(t)} \right)=\int_0^{+\infty} e^{-\lambda s} q_\mu(s,t)ds=e^{-t\mu(\lambda)}
\]
  where $q_\mu(s,t)$ is the probability density function of $\mathbf{D}_{\mu}(t)$; see \cite{Applebaum2009} for details.
  We have \cite{Alrawashdeh2017,Applebaum2009}
  \[
    q_\mu(s,t)=\dfrac{ts^{-\frac{3}{2}}}{2\sqrt{\pi}}e^{at-a^2s-\frac{t^2}{4s}}, \qquad s,t>0.
  \]
$q_\mu(s,t)$ can also be regarded as the density of the inverse Gaussian process \cite[Example 1.3.21]{Applebaum2009}:
  \[
    \mathbf{G}(t)=\inf\{z>0: \mathbf{B}(z)+\sqrt{2}az=\frac{t}{\sqrt{2}}\},
  \]
  where $\mathbf{B}(t)$ is the one-dimensional standard Brownian motion.

  Furthermore, we present a second-order partial differential equation satisfied by $q_\mu(s,t)$. This result will be used in Section 3.2.

  \begin{prop}\label{Section-2-prop-q-equation} \cite[Theorem 3.1]{Kumar2011}
    For $t,s>0$, the function $q_\mu(s,t)$  satisfies the second-order partial differential equation
    \begin{equation*}
      \frac{\partial^2 q_\mu}{\partial t^2} - 2a \frac{\partial q_\mu}{\partial t} - \frac{\partial q_\mu}{\partial s} = 0.
    \end{equation*}
  \end{prop}

  No simple closed form for $p_\mu(s,t)$ is available; useful integral and computational
  representations are available in \cite{Alrawashdeh2017,Kumar2015}.

  Finally, if $a=0$, then $\mu(\lambda)$ reduces to $\zeta(\lambda)=\sqrt{\lambda} (\lambda>0)$.
  In this case, $\mathbf{D}_\mu (t)$  reduces to the $1/2$-stable subordinator $\mathbf{D}_\zeta (t)$ and
  \begin{equation}\label{q-zeta-function}
    q_\zeta(s,t)=\dfrac{ts^{-\frac{3}{2}}}{2\sqrt{\pi}}e^{-\frac{t^2}{4s}},\quad \int_0^\infty e^{-\lambda s}q_\zeta(s,t)ds=e^{-t\sqrt{\lambda}}.
  \end{equation}
  See \cite[Example 1.3.19, Exercise 1.3.20]{Applebaum2009}. A direct computation shows that
  \begin{equation}\label{p-zeta-function}
    p_\zeta(s,t)=\dfrac{1}{\sqrt{\pi t}} e^{-\frac{s^2}{4t}}
  \end{equation}
  which is the kernel in the subordination formula \eqref{introduction-weierstrass} from the wave equation \eqref{introduction-WE} to the diffusion equation \eqref{introduction-DE}
  in Section 1.

  \subsection{Abstract evolutionary integral equations and subordination principles}

  Throughout this paper, $A$ is a densely defined linear operator on a Banach space $X$.
$B(X)$ denotes the set of all bounded linear operators on $X$.
$L^1(\mathbb{R}_+; X)$ denotes the space of all Bochner-measurable functions:
$u: \mathbb{R}_+ \rightarrow X$ such that $\|u(\cdot)\|$ is integrable.
  Methods of the vector-valued Laplace transform \cite{Arendt2010book,Engel1999,Pruss1993book} are frequently used in this paper.
  Suppose an $X$-valued function $u(t)$ is exponentially bounded, i.e., there exist constants $M>0, \omega\in \mathbb{R}$
  such that $\|u(t)\|\leqslant Me^{\omega t}$. In particular, if $\omega=0$, $u(t)$ is called bounded.
  The Laplace transform of a vector-valued function $u: \mathbb{R}_+\rightarrow X$ is denoted by
  \[
    \hat{u}(\lambda)=\int_0^\infty e^{-\lambda t}u(t)\,dt, \qquad \operatorname{Re}\lambda>\omega.
  \]
  For $f\in L^1(\mathbb{R}_+; \mathbb{R})$ and $g\in L^1(\mathbb{R}_+; X)$, the Laplace convolution is defined by
  \[
    (f\ast_L g)(t)=\int_0^t f(t-s)g(s)ds=\int_0^t f(s)g(t-s)ds, \quad t>0.
  \]
  Obviously,
  \[
    \widehat{(f\ast_L g)} (\lambda)=\hat{f}(\lambda)\hat{g}(\lambda).
  \]

  \begin{defn}\cite{Pruss1993book}
    \label{resolvent family}
    A family $\{T(t)\}_{t \geq 0} \subset B(X)$
    is called a resolvent family generated by $(a, A)$ if the
    following conditions are satisfied:

    (a) $T(t)$ is strongly continuous for $t\geq0$ and
    $T(0)=I$;

    (b) $T(t)A \subset AT(t)$ for $t\geq 0$;

    (c) for $x \in D(A)$, the resolvent equation
    \begin{equation}\label{resolvent equation}
      T(t)x = x + \int_0^t a (t-s) T(s)Ax ds
    \end{equation}
    holds for all $t\geq 0$.
  \end{defn}

  \begin{rem}

    (1) Special choices of $a(t)$ recover classical evolution equations:
    \begin{itemize}
      \item If $a(t) \equiv 1$, then $T(t)$ is a $C_0$-semigroup, and \eqref{resolvent equation} reduces to the first-order Cauchy problem $T'(t)x = A T(t)x$; see \cite{Applebaum2009,Arendt2010book, Engel1999}.
      \item If $a(t) = t$, then $T(t)$ is a cosine family, and \eqref{resolvent equation} reduces to the second-order Cauchy problem $T''(t)x = A T(t)x$; see \cite{Arendt2010book}.
      \item If $a(t) =g_{\alpha}(t)=\frac{t^{\alpha-1}}{\Gamma(\alpha)}$ with $\alpha>0$, then $T(t)$ is called an $\alpha$-times resolvent family, and \eqref{resolvent equation} can be interpreted as a differential equation of fractional order \cite{Bajlekova2001,Limiao2010}. Specifically, define the Caputo fractional derivative as
        \begin{align*}
          (D_t^\alpha f)(t)=
            \begin{cases}
             \frac{1}{\Gamma(k-\alpha)}\int_0^t (t-s)^{k-\alpha-1} f^{(k)}(s)ds,  \qquad & \text{for } k-1<\alpha<k;\\
            f^{(k)}(t), \qquad &\text{for } \alpha=k
           \end{cases}
       \end{align*}
      where $k=\lceil\alpha\rceil$ is the smallest integer no less than $\alpha>0$. Then the Volterra equation
      \[
      u(t) = x + \int_0^t g_{\alpha}(t-s) Au(s) ds
      \]
      and the fractional differential equation
      \begin{align*}
          \begin{cases}
                  D_t^\alpha u(t)=Au(t),\quad\quad t>0; \\
                  u(0) = x, \quad u^{(j)}(0)=0,\ j=1,\dots,k-1\ (\text{if}\quad k>1)
          \end{cases}
      \end{align*}
      can be transformed into each other. Moreover, their solutions can be represented as $u(t)=T(t)x$.
    \end{itemize}

    (2) Generally, if there exists $b(\cdot)\in L_{loc}^1((0,\infty))$ such that $(b\ast_L a)(t)=t^k (k=0,1,2,\cdots)$,
    we may transform \eqref{resolvent equation} into $\Phi_t T(t)x=A T(t)x$
    where $\Phi_t$ is a linear integro-differential operator acting on $T(t)x$ with respect to the time variable $t$.
    As an example, take the case $(b\ast_L a)(t)=1$: it follows from \eqref{resolvent equation} that
    \[
      \int_0^t b(t-s) (T(s)x-x)ds=(b\ast_L a\ast_L T)(t)Ax=(1\ast_L T)(t)Ax=\int_0^t T(s)Axds.
    \]
    Then
    \[
      \frac{d}{dt} \int_0^t b(t-s) (T(s)x-x)ds=T(t)Ax=AT(t)x.
    \]
    Under appropriate regularity conditions, we have
    \[
      \frac{d}{dt} \int_0^t b(t-s) (T(s)x-x)ds=\int_0^t b(t-s) T'(s)xds=AT(t)x
    \]
    where $\Phi_t T(t)x:=\int_0^t b(t-s) T'(s)xds$ is a linear integro-differential operator.
    See \cite{Kemppainen2016}.
    This formulation will be used in Section 4.

    (3) Throughout this paper, we always assume that $T(t)$ is bounded, i.e., there exists a constant $M$
    such that $\|T(t)\|\leq M$.
  \end{rem}

  Next, we give a brief introduction to two types of subordination principles corresponding to
  the specific changes in the operator $A$ (space-subordination) and the integral kernel $a(t)$ (time-subordination), respectively.
  The former can be referred to as {\it subordination in the sense of Bochner} from the viewpoint of probability theory \cite{Applebaum2009,Schilling2012}.
 The latter is often called {\it subordination in the sense of Pr\"uss} since \cite{Pruss1993book} gave a systematic account of the related methods and applied them to the study of evolutionary integral equations.

  Firstly, we introduce the Bernstein functional calculus \cite{Applebaum2009,Gomilko2015,Schilling2012}.
 Let $\varphi$ be a Laplace exponent with drift $b$ and L\'evy measure $\sigma$.
  Define the Bernstein functional calculus as
  \[
    -\varphi(-A)x=bAx-\int_{(0,\infty)}(1-e^{sA})x\sigma(ds), \quad x\in D(A).
  \]
  Then $-\varphi(-A)$ also generates a bounded $C_0$-semigroup $\{e^{-t\varphi(-A)}\}_{t\geq0}$ on $X$:
  \begin{equation}
    e^{-t\varphi(-A)}x=\int_{[0,\infty)}e^{sA}x  Q_\varphi(ds,t),  \quad t\geq0
  \end{equation}
  where $Q_\varphi(ds,t)$ is the law of the subordinator $\mathbf{D}_\varphi(t)$ with Laplace exponent $\varphi (\cdot)$.
  The semigroup $\{e^{-t\varphi(-A)}\}_{t\geq0}$ is called subordinate to the semigroup $\{e^{tA}\}_{t\geq0}$ with respect to the Laplace exponent  $\varphi (\cdot)$.

  For instance, $-\mu(-\Delta)=a-\sqrt{a^2-\Delta}$ arises from the Laplacian $\Delta$ via the Bernstein function $\mu(\lambda)=-a+\sqrt{a^2+\lambda}$;
  the semigroup $\{e^{-t\mu(-\Delta)}\}_{t\geq0}$ is subordinate to the diffusion semigroup $\{e^{t\Delta}\}_{t\geq0}$. The fractional Laplacian operator $-(-\Delta)^\alpha (0<\alpha<1)$ can also be incorporated into the framework of Bernstein functional calculus.

  Secondly, for the subordination with respect to the integral kernel,
  we introduce a fundamental principle as follows  \cite[Theorem 4.1]{Pruss1993book}.

  \begin{prop}
    Let $T_a(t), T_b(t)$ be the resolvent families generated by $(a(t),A)$ and $(b(t),A)$, respectively. If there exists a Laplace exponent $\varphi(\lambda)$
    such that $\widehat{a}(\varphi(\lambda))=\widehat{b}(\lambda)$, then
    \[
      T_b(t)x=\int_0^\infty T_a(s)xP_\varphi(ds,t),
    \]
    where the inverse law $P_\varphi(ds,t)$ satisfies $\int_0^\infty e^{-\lambda t}P_\varphi(ds,t)dt=\frac{\varphi(\lambda)}{\lambda}e^{-s\varphi(\lambda)}ds$. In particular, if $P_\varphi(s,t)$ has Lebesgue density $p_\varphi(s,t)$, then
    \[
      T_b(t)x=\int_0^\infty p_\varphi(s,t)T_a(s)x\,ds.
    \]
  \end{prop}

  Since we have used notation different from \cite{Pruss1993book}, we provide a formal
  verification of this proposition.
 Taking the Laplace transform of both sides of $T(t)x = x + \int_0^t a (t-s) T(s)Ax ds$ gives
    \[
      \widehat{T}(\lambda)x=\frac{1}{\lambda} [I-\widehat{a}(\lambda)A]^{-1}x.
    \]
    We have
    \begin{align*}
      \int_0^\infty e^{-\lambda t}\int_0^\infty T_a(s)xP_\varphi(ds,t) dt&
      =\int_0^\infty \int_0^\infty e^{-\lambda t}P_\varphi(ds,t)dt T_a(s)x\\
      &=\int_0^\infty \frac{\varphi(\lambda)}{\lambda}e^{-s\varphi(\lambda)}T_a(s)xds\\
      &=\frac{\varphi(\lambda)}{\lambda}\cdot \frac{1}{\varphi(\lambda)} [I-\widehat{a}(\varphi(\lambda))A]^{-1}x\\
      &=\frac{1}{\lambda}[I-\widehat{b}(\lambda)A]^{-1}x\\
      &=\int_0^\infty e^{-\lambda t} T_b(t)xdt.
    \end{align*}
    Then we obtain
    \[
      T_b(t)x=\int_{[0,\infty)} T_a(s)xP_\varphi(ds,t)
    \]
    by the uniqueness of the Laplace transform.

  \begin{rem}
    Although the equations in Sections 3-4 usually appear in the form of time-differential equations,
    the underlying ideas of the subordination construction primarily stem from this principle.
    When applying the above principle, the main difficulty lies in determining whether the inverse function $\varphi(\lambda)=\widehat{a}^{-1}(\widehat{b}(\lambda))$
    is a Bernstein function.

  \end{rem}

  \section{Subordination and duality of the telegraph equation}

  In this section, we first introduce the main subordination from an abstract telegraph-type equation to a diffusion-type equation.
  Through the composition of this subordination formula with other subordination representations, we provide new insights into certain established results.
  We then illustrate the \emph{duality} from the perspectives of subordination and asymptotic behaviors.
  The research in this section also serves as the prototype and motivation for the general theory of Section 4 and the quantitative theory of Section 5.

  \subsection{Abstract subordination from the telegraph equation to the diffusion equation}

  Consider abstract evolutionary equations in a Banach space $X$.
  The subordination relations among the following three abstract equations
  \begin{equation}
    \begin{aligned}
      \begin{cases}
        w''(t)=Aw(t), \quad \quad t>0; \\
        w(0)=u_0, \quad w'(0)=0,
      \end{cases}
    \end{aligned}
    \tag{AWE}
  \end{equation}

  \begin{equation}\label{abstract telegraph equation}
    \begin{aligned}
      \begin{cases}
        u''(t)+2a u'(t)=Au(t), \quad \quad t>0; \\
        u(0)=u_0, \quad u'(0)=0
      \end{cases}
    \end{aligned}
    \tag{ATE}
  \end{equation}
  and
  \begin{equation}\label{abstract diffusion equation}
    \begin{aligned}
      \begin{cases}
        v'(t) & = Av(t),\quad\quad t>0; \\
        v(0)  & = u_0
      \end{cases}
    \end{aligned}
    \tag{ADE}
  \end{equation}
  are the main objects of this section. As illustrated in Section 1, we have
  \[
    \text{(AWE)}\mapsto\text{(ADE)}: \quad  v(t)=\int_0^{+\infty} p_\zeta(s,t)w(s)ds=
    \int_0^{+\infty}\frac{1}{\sqrt{\pi t}} e^{-\frac{s^2}{4t}}w(s)ds
  \]
  and
  \[
    \text{(AWE)}\mapsto\text{(ATE)}: \quad  u(t)=\int_0^{+\infty} p_\eta(s,t)w(s)ds.
  \]

  Now we establish the subordination principle of $\text{(ATE)}\mapsto\text{(ADE)}$.

  \begin{thm}\label{Section-3-main-thm}
    Let $A$ be the generator of a bounded $C_0$-semigroup on $X$, $u_0\in D(A)$, and let $u(t)$ be the unique bounded solution of \eqref{abstract telegraph equation}.
    Then the solution of \eqref{abstract diffusion equation} can be represented as
    \begin{equation}\label{abstract subordination representation}
      v(t)=\int_0^{+\infty} p_\mu(s,t)u(s)\,ds, \quad t>0
    \end{equation}
    where $p_\mu(s,t)$ is defined  in Section 2.
  \end{thm}

  \begin{proof}
    Denote the Laplace transform of $u(t)$ by
    \[
      \hat{u}(r)=\int_0^{+\infty} e^{-r t} u(t)dt.
    \]
    Taking the Laplace transform of both sides of \eqref{abstract telegraph equation} yields
    \[
      r^2\hat{u}(r)-r u_0+2a(r\hat{u}(r)- u_0)=A\hat{u}(r),
    \]
    then
    \[
      \hat{u}(r)=(r+2a) \left(r^2+2ar-A\right)^{-1}u_0.
    \]

    Since $u(t)$ is bounded,
    \[
      \int_0^\infty\left\|e^{-\lambda t} \int_0^{+\infty} p_\mu(s,t)u(s)\,ds\right\|\,dt
      \leq\int_0^\infty e^{-(\operatorname{Re}\lambda) t}\int_0^{+\infty}Mp_\mu(s,t)\,dsdt=
      M\int_0^\infty e^{-(\operatorname{Re}\lambda) t}\,dt<\infty
    \]
    for every \(\lambda\) with \(\operatorname{Re}\lambda>0\).
    Therefore $\int_0^{+\infty} p_\mu(s,t)u(s)\,ds$ is Laplace transformable.
    Taking the Laplace transform of both sides of \eqref{abstract subordination representation} leads to
    \begin{align*}
      &\int_0^{+\infty} e^{-\lambda t} v(t)dt \\
      &=\int_0^{+\infty} e^{-\lambda t} \int_0^{+\infty} p_\mu(s,t)u(s)ds dt \\
      &=\int_0^{+\infty}\left[\int_0^{+\infty}  e^{-\lambda t} p_\mu(s,t)dt\right] u(s)ds\\
      &=\int_0^{+\infty} \frac{\mu(\lambda)}{\lambda}e^{-s\mu(\lambda)} u(s)ds\\
      &=\int_0^{+\infty} \frac{-a+\sqrt{\lambda+a^2}}{\lambda}e^{-s(-a+\sqrt{\lambda+a^2})} u(s)ds\\
      &= \frac{(-a+\sqrt{\lambda+a^2})(\sqrt{\lambda+a^2}+a) }{\lambda}  \left[(-a+\sqrt{\lambda+a^2})^2+2a(-a+\sqrt{\lambda+a^2}) -A  \right]^{-1}u_0\\
      &=(\lambda-A)^{-1}u_0
    \end{align*}
    where we use the Laplace transform $\hat{u}(r)$  with the substitution $r=-a+\sqrt{\lambda+a^2}$.
    That is,
    \[
      \hat{v}(\lambda)=(\lambda-A)^{-1}u_0.
    \]
    Therefore $v(t)$ is the solution of \eqref{abstract diffusion equation}
    by the uniqueness of vector-valued Laplace transforms and semigroup theory.
  \end{proof}

  \begin{rem}\label{Section-3-rem-0}
   (i) If $-A$ is a densely defined positive self-adjoint linear operator on a Hilbert space $X$, then
    the conditions in Theorem \ref{Section-3-main-thm} are satisfied. The Laplacian $\Delta$ is
    the simplest prototype for $A$.

   (ii)    The subordinations $\text{(AWE)}\mapsto \text{(ATE)}$ and $\text{(ATE)}\mapsto\text{(ADE)}$ can be composed.
    Recall that
    \[
      v(t)=\int_0^{+\infty} p_\zeta(s,t)w(s)ds
    \]
    and
   \[
      v(t)=\int_0^{+\infty} \left[\int_0^{+\infty} p_\mu(\tau,t) p_\eta(s,\tau) d\tau\right]w(s)ds.
   \]
   Comparing these two expressions for $v(t)$ suggests the identity
    \begin{equation}\label{Section-3-example-1}
      p_\zeta(s,t)=\int_0^{+\infty} p_\mu(\tau,t) p_\eta(s,\tau) d\tau.
    \end{equation}
    It can be verified directly via the Laplace transform.
  \end{rem}

  \begin{rem}\label{section-3-remark-1}
    Consider
    \begin{equation}\label{section-3-rem-1-2}
      \begin{cases}
        2av'(t)=Av(t), \quad \quad t>0, \\
        v(0)=u_0
      \end{cases}
    \end{equation}
    which is essentially the same as  \eqref{abstract diffusion equation}.

    Over the past few decades, the relation between  \eqref{abstract telegraph equation}  and \eqref{section-3-rem-1-2}
    has been extensively studied.
    There have been two approaches to establishing the connection between \eqref{abstract telegraph equation} and \eqref{section-3-rem-1-2},
    which are described below.

    Firstly, consider
    \begin{equation*}
      \begin{aligned}
        \begin{cases}
          \varepsilon u_\varepsilon''(t)+2a u_\varepsilon'(t)=Au_\varepsilon(t), \quad \quad t>0, \\
          u_\varepsilon(0)=u_0, \quad u_\varepsilon'(0)=0.
        \end{cases}
      \end{aligned}
    \end{equation*}
    Intuitively, it should hold that
    \begin{equation*}
      \lim\limits_{\varepsilon\rightarrow 0+} u_\varepsilon(t)=v(t).
    \end{equation*}
    This type of singular perturbation problem has been generalized to various equations  \cite{Fattorini1985book, Kisynski1963}.
    This approach is essentially a variation of Kac's scaling method. Precisely, for
    \[
      u''(t)+2b a u'(t)=c^2Au(t),
    \]
    one has $u(t)\rightarrow v(t)$ under Kac's scaling \cite{Kac1974}:
    \[
      b\rightarrow \infty, \quad c\rightarrow \infty, \quad \frac{c^2}{b}\rightarrow 1.
    \]

    Another approach studies the long-time diffusion phenomenon for the telegraph equation. Intuitively,
    the effect of the second-order time derivative term in the telegraph equation is suppressed by the damping term as $t\to\infty$. Therefore,
    \[
      \lim\limits_{t\rightarrow \infty} (u(t)-v(t))=0
    \]
    in an appropriate topology, with a quantitative convergence rate.

    Since the introduction of the stochastic model for the telegraph equation,
    its long-time diffusion behavior has been extensively studied in probability theory (see, e.g., \cite{Janssen1990}).
    Also, various abstract results with precise decay rates have been proposed.
    For example, let $-A$ be a closed self-adjoint positive semi-definite operator on a separable Hilbert space $X$ (note that this sign convention differs from the original citation, where the equation is written as $u''+2au'+Au=0$), and let $u_0\in D((-A)^{1/2})$.
    Chill and Haraux \cite{Chill2003JDE} have proved that
    there exists a constant $C\geq0$ such that for every $t\geq 1$,
    \[
      t\|u(t)-v(t)\|\leq C\|u_0\|.
    \]
    See \cite{Ikehata2003Studia, Nishihara2003, Radu2016SIAM, Sobajima2021} and the references therein
    for this issue and its generalizations from various perspectives.

    In the present paper, similar to Theorem \ref{Section-3-main-thm}, we can build a third, exact subordination approach:
    an exact subordination representation
    \begin{equation*}
      v(t)=\int_0^{+\infty} p_{*}(s,t)u(s)ds, \quad t>0
    \end{equation*}
    where $p_{*}(s,t)$ is determined by
    \[
      \int_0^{+\infty} e^{-\lambda t} p_{*}(s,t)dt=\dfrac{-a+\sqrt{a^2+2a\lambda}}{\lambda} e^{-s (-a+\sqrt{a^2+2a\lambda})}.
    \]
  \end{rem}

  \begin{example}\label{section-3-example-1}

    We review some stochastic interpretations of the telegraph equation needed below,
    and we also clarify some points from Section 1 that were only briefly mentioned.

    The solution of the telegraph equation
    \begin{equation} \label{classical telegraph}
      \begin{cases}
        \partial_t^{2}u(x,t)+2a\partial_t u(x,t)=\partial_x^2u(x,t) \quad \quad \quad t>0,x\in \Real; \\
        u(x,0)=\phi(x), \partial_t u(x,0)=0.
      \end{cases}
    \end{equation}
    can be represented as
    \[
      u(x,t)=\frac12 \mathbb{E}[\phi(x+\mathbf{U}(t))+\phi(x-\mathbf{U}(t))],
    \]
    where
    \[
      \mathbf{U}(t)=\int_0^t(-1)^{\mathbf{N}(s)}ds,\quad t\geq 0
    \]
    and $\{\mathbf{N}(t)\}_{t\geq 0}$ is a Poisson process with rate $a$. See \cite{Bobrowski2005book,Dewitt1989,Kac1974,Lichenggang2019} for more information.

    Define the telegraph process by \cite{Orsingher2004,Orsingher2009AOP}
    \begin{equation*}
      \mathbf{H}(t)=\mathbf{V}(0)\int_0^t(-1)^{\mathbf{N}(s)}ds
    \end{equation*}
    where $\mathbf{V}(0)$ is a two-valued random variable taking values $\pm1$ each with probability $1/2$, and $\mathbf{N}(t)$ is independent of $\mathbf{V}(0)$.

    Note that $\mathbf{U}(t)$ is not an inverse subordinator. However, we can connect $\mathbf{U}(t)$ with an inverse subordinator $\mathbf{E}_\eta (t)$ where $\eta(\lambda)=\sqrt{\lambda^2+2a\lambda}$.
    Let the generalized density of $\mathbf{U}(t)$ be $g(r,t)$. Then $|\mathbf{U}(t)|$ has the generalized density $p_\eta(r,t)=g(r,t)+g(-r,t) $ for $ r\geq0$.
    Moreover, we have the following Laplace transforms  \cite{Dewitt1989}
    \begin{align*}
      \int_0^\infty e^{-\lambda t}g(r,t)dt=
      \begin{cases} \frac12\bigl(\frac{\eta(\lambda)}{\lambda}+1\bigr) e^{-r\eta(\lambda)}, & r>0, \\
              \frac12\bigl(\frac{\eta(\lambda)}{\lambda}-1\bigr) e^{r\eta(\lambda)},  & r<0, \\
              \frac{\eta(\lambda)}{2\lambda},                                         & r=0.\end{cases}
    \end{align*}
    and
    \[
      \int_0^\infty e^{-\lambda t}p_\eta(r,t)dt=\frac{\eta(\lambda)}{\lambda}e^{-r\eta(\lambda)},\quad\quad r\geq 0.
    \]
    The solution of \eqref{classical telegraph} can be represented as
    \begin{align*}
      u(x,t)&=\frac12 \mathbb{E}[\phi(x+\mathbf{U}(t))+\phi(x-\mathbf{U}(t))] \\
      &=\frac12 \mathbb{E}[\phi(x+|\mathbf{U}(t)|)+\phi(x-|\mathbf{U}(t)|)] \\
      &=\int_{-\infty}^{\infty}\frac12 [\phi(x+r)+\phi(x-r)]g(r,t)dr \\
      &=\int_{0}^{\infty}\frac12 [\phi(x+r)+\phi(x-r)]p_\eta(r,t)dr.
    \end{align*}

    $|\mathbf{U}(t)|$ and the inverse subordinator $\mathbf{E}_\eta(t)$ are equal in distribution.
    Hereafter, we will use
    \[
      u(x,t)=\int_0^\infty p_\eta(s,t)w(x,s)ds
    \]
    to represent the solution of the telegraph equation  \eqref{classical telegraph}, in which
    $w(x,t)=\frac12 [\phi(x+t)+\phi(x-t)]$ is the solution of the wave equation
    \begin{equation*}
      \begin{cases}
        \partial_t^{2}w(x,t)=\partial_x^2w(x,t) \quad \quad \quad t>0,x\in \Real; \\
        w(x,0)=\phi(x), \partial_t w(x,0)=0.
      \end{cases}
    \end{equation*}

  \end{example}

  Using the subordination in Theorem \ref{Section-3-main-thm} and the stochastic interpretation for telegraph equation, we can give a new perspective on a known stochastic subordination between multi-term time differential equations.

  \begin{example}\label{section-3-example-2}
    Consider the following three equations:

    \begin{equation}\label{interpretation-Orsingher-PTRF-2004-1}
      \begin{aligned}
        \begin{cases}
          \partial_t^{2}u(x,t)+2a\partial_t u(x,t)=\partial_x^2u(x,t), \qquad t>0,x\in \Real; \\
          u(x,0)=\psi(x), \partial_t u(x,0)=0,
        \end{cases}
      \end{aligned}
    \end{equation}

    \begin{equation}\label{interpretation-Orsingher-PTRF-2004-2}
      \begin{aligned}
        \begin{cases}
          \partial_t v(x,t)=\partial_x^2v(x,t), \qquad t>0,x\in \Real; \\
          v(x,0)=\psi(x)
        \end{cases}
      \end{aligned}
    \end{equation}
    and
    \begin{equation}\label{interpretation-Orsingher-PTRF-2004-3}
      \begin{aligned}
        \begin{cases}
          \partial_t g(x,t)+2a \partial_t^{1/2}g(x,t)=\partial_x^2g(x,t), \qquad t>0,x\in \Real; \\
          g(x,0)=\psi(x)
        \end{cases}
      \end{aligned}
    \end{equation}
    where $\partial_t^{1/2} g(x,t)=\frac{1}{\Gamma(1/2)}\int_0^t (t-s)^{-1/2}\,\partial_s g(x,s)\,ds$ is the Caputo fractional derivative of order $1/2$ with respect to $t$.
    Then the fundamental solution of \eqref{interpretation-Orsingher-PTRF-2004-3} is the density of $\mathbf{H}(|\mathbf{B}(t)|)$,
    where $\mathbf{B}(t)$ is a one-dimensional  Brownian motion independent of $\mathbf{H}(t)$ \cite{Orsingher2004,Orsingher2009AOP}. Therefore, the subordination from \eqref{interpretation-Orsingher-PTRF-2004-1} to \eqref{interpretation-Orsingher-PTRF-2004-3} is established.

    The known identity above admits a refinement, which is the new content of this example: the time
    change $|\mathbf{B}(t)|$ itself decomposes, in distribution, into a composition of two independent
    inverse subordinators. Let $\mathbf{D}_1(t)$ and $\mathbf{D}_2(t)$ be independent subordinators
    with Laplace exponents $h_1(\lambda)=-a+\sqrt{a^2+\lambda}$ and $h_2(\lambda)=\lambda+ 2a\lambda^{1/2}$.
    Moreover, let $p_1(s,t)$ and $p_2(s,t)$ be the probability densities of $\mathbf{E}_1(t)$ and $\mathbf{E}_2(t)$ , respectively.
    Note that $h_2$ has a drift coefficient $1$, but its L{\'e}vy measure is infinite, so $\mathbf{E}_2(t)$ indeed admits a density, in contrast with $\mathbf{E}_\eta(t)$ in Example \ref{four-BF-example}.
    Then we have
    \[
      v(x,t)=\int_0^\infty p_1(s,t)u(x,s)ds, \quad  g(x,t)=\int_0^\infty p_2(s,t)v(x,s)ds.
    \]

    Let
    \[
      p_0(s,t)=\int_0^\infty p_2(\tau,t)p_1(s,\tau)d\tau.
    \]
    Then
    \[
      g(x,t)=\int_0^\infty p_0(s,t)u(x,s)ds
    \]
    where
    \[
      p_0(s,t)=\frac{1}{\sqrt{\pi t}} e^{-\frac{s^2}{4t}}
    \]
    is the density of $|\mathbf{B}(t)|$, or equivalently of the inverse $1/2$-stable subordinator. To sum up, we can provide a new stochastic interpretation of the subordination from \eqref{interpretation-Orsingher-PTRF-2004-1} to \eqref{interpretation-Orsingher-PTRF-2004-3}:
    \[
      \mathbf{H}(|\mathbf{B}(t)|)\ \overset{d}=\mathbf{H}(\mathbf{E}_1(\mathbf{E}_2(t)))
    \]
    in distribution. In flowchart terms,
    we can decompose the known subordination $\eqref{interpretation-Orsingher-PTRF-2004-1}\mapsto \eqref{interpretation-Orsingher-PTRF-2004-3}$ into  $\eqref{interpretation-Orsingher-PTRF-2004-1}\mapsto \eqref{interpretation-Orsingher-PTRF-2004-2}\mapsto \eqref{interpretation-Orsingher-PTRF-2004-3}$ with a stochastic interpretation. Note that the second step $\eqref{interpretation-Orsingher-PTRF-2004-2}\mapsto \eqref{interpretation-Orsingher-PTRF-2004-3}$ is the subordination principle for time-fractional diffusion equations, which is well known in the fractional calculus community \cite{Bajlekova2018}.

  \end{example}

  \subsection{Duality of the telegraph and nonlocal diffusion equations via subordination}

  In this subsection, we study the \emph{duality} between $ \frac{\partial^2 u}{\partial t^2}+2a\frac{\partial u}{\partial t}=c^2 \Delta u$
  and $\frac{\partial f}{\partial t}=(a-\sqrt{a^2-c^2\Delta})f$ from the perspective of subordination.

  Throughout this paper, we adopt the Fourier transform convention:
\[
 (\mathcal{F}_xu(\cdot,t))(\xi):=\tilde u(\xi,t):=\int_{\mathbb{R}^n} u(x,t)e^{-ix\cdot\xi}\,dx,\qquad
u(x,t)=\frac{1}{(2\pi)^n}\int_{\mathbb{R}^n}\tilde u(\xi,t)e^{ix\cdot\xi}\,d\xi.
\]

  Let $a=\frac{mc^2}{\hbar}$. Consider the following three equations:
  \begin{equation}\label{Section-4-CTE}
    \begin{aligned}
      \begin{cases}
        \frac{\partial^2 u}{\partial t^2}+2a\frac{\partial u}{\partial t}=c^2 \Delta u, \quad \quad t>0, x\in \mathbb{R}^n; \\
        u(x,0)=u_0(x), \quad u'(x,0)=0,
      \end{cases}
    \end{aligned}
  \end{equation}

  \begin{equation}\label{Section-4-CDE}
    \begin{aligned}
      \begin{cases}
        \frac{\partial v}{\partial t}=c^2 \Delta v, \quad \quad t>0, x\in \mathbb{R}^n; \\
        v(x,0)=u_0(x)
      \end{cases}
    \end{aligned}
  \end{equation}
  and
  \begin{equation}\label{Section-4-CNLE}
    \begin{aligned}
      \begin{cases}
        \frac{\partial f}{\partial t}=(a-\sqrt{a^2-c^2\Delta})f, \quad \quad t>0, x\in \mathbb{R}^n; \\
        f(x,0)=u_0(x)
      \end{cases}
    \end{aligned}
  \end{equation}
where $u_0(\cdot)\in\mathcal{S}(\mathbb{R}^n) $, the Schwartz space of rapidly decreasing functions.

  We have
  \begin{equation}\label{Section-4-subordination-expression}
    \quad f(x,t)=\int_0^\infty q_\mu(s,t) v(x,s)ds; \qquad v(x,t)=\int_0^\infty p_\mu(s,t) u(x,s)ds
  \end{equation}
  where $q_\mu(s,t)$ and $p_\mu(s,t)$ are the probability densities of the subordinator $\mathbf{D}_\mu(t)$ and the inverse subordinator $\mathbf{E}_\mu(t)$, respectively.

  Moreover, the following composition formula is valid:
  \begin{align*}
    f(x,t) & =\int_0^\infty q_\mu(s,t) v(x,s)\,ds = \int_0^\infty q_\mu(s,t) \int_0^\infty p_\mu(\tau,s) u(x,\tau)\,d\tau\,ds \\
           & =\int_0^\infty \left( \int_0^\infty q_\mu(s,t) p_\mu(\tau,s)\,ds \right) u(x,\tau)\,d\tau.
  \end{align*}

  The significance of \eqref{Section-4-subordination-expression} is that, by using the symbol $\mu(\lambda)$ and the processes $\mathbf{D}_\mu(t), \mathbf{E}_\mu(t)$ derived from the relativistic Schr\"odinger operator, we obtain subordination relations between normal diffusion \eqref{Section-4-CDE} and two relativistic diffusion equations.
  In other words, by leveraging the duality between the subordinator and the inverse subordinator,
  we establish a duality between the telegraph equation \eqref{Section-4-CTE} and the spatially nonlocal diffusion equation \eqref{Section-4-CNLE}, in which normal diffusion \eqref{Section-4-CDE} plays the role of a fulcrum.

  Using the standard harmonic extension technique \cite{Caffarelli2007CPDE,Kwasnicki2024}, we may convert a nonlocal (in space) operator into a local one.
  We write $\Delta_x$ for the Laplacian in the $x$ variables to emphasize the Laplacian with respect to $x$ in this subsection.

  \begin{prop}
    Suppose
    \begin{equation}
      \begin{cases}
        \frac{\partial^2 f}{\partial t^2}(x,t)-2a\frac{\partial f}{\partial t}(x,t)+c^2\Delta_x f(x,t)=0, & t>0,\ x\in\mathbb{R}^n; \\
        f(x,0)=f_0(x),\quad \lim\limits_{t\to\infty} f(x,t)=0.
      \end{cases}
      \tag{EE}
    \end{equation}
    where $f_0(\cdot)\in\mathcal{S}(\mathbb{R}^n) $, the Schwartz space of rapidly decreasing functions.
    Then
    \[
      \lim_{t\to 0^+} \frac{\partial f}{\partial t}(x,t) = \bigl(a-\sqrt{a^2-c^2\Delta_x}\bigr)f_0(x).
    \]
  \end{prop}
  \begin{proof}
    See \cite{Baeumer2010} for the motivation.
    We formulate the computation in a standard harmonic extension procedure. Take the Fourier transform with respect to $x$,
    \begin{equation*}
      \begin{cases}
        \frac{\partial^2}{\partial t^2}\tilde{f}(\xi,t)-2a\frac{\partial}{\partial t}\tilde{f}(\xi,t)-c^2|\xi|^2\tilde{f}(\xi,t)=0, & t>0, \ \xi\in\mathbb{R}^n; \\
        \tilde{f}(\xi,0)=\tilde{f_0}(\xi),\quad \lim_{t\to\infty}\tilde{f}(\xi,t)=0.
      \end{cases}
    \end{equation*}
    Then, selecting the non-growing branch gives
    \[
      \tilde{f}(\xi,t)=e^{t(a-\sqrt{a^2+c^2|\xi|^2})}\widetilde{f_0}(\xi).
    \]
    Therefore
    \begin{align*}
      \lim_{t\to 0^+}\frac{\partial}{\partial t}\tilde{f}(\xi,t) & = (a-\sqrt{a^2+c^2|\xi|^2})\lim_{t\to 0^+} e^{t(a-\sqrt{a^2+c^2|\xi|^2})}\widetilde{f_0}(\xi) \\
                                                                   & = (a-\sqrt{a^2+c^2|\xi|^2})\widetilde{f_0}(\xi).
    \end{align*}
    Finally,
    \[
      \lim\limits_{t\rightarrow 0^+} \frac{\partial f}{\partial t} (x,t)=\left(a-\sqrt{a^2-c^2\Delta_x}\right)f_0(x).
    \]

  \end{proof}

  When $a=0$, the above calculation reduces to the classical extension problem for the fractional Laplacian $-(-\Delta)^{1/2}$ \cite{Caffarelli2007CPDE}.
  The harmonic extension technique is an important method in PDEs and probability,
  as it transforms certain nonlocal problems into local problems in higher-dimensional spaces.

  \begin{prop}
    Let
    \[
      f(x,t) = \int_0^{+\infty} q_\mu(s,t) \, v(x,s) \, ds, (t>0), \qquad  f(x,0)=u_0(x)
    \]
    where
    \begin{equation*}
      v(x,t) = \frac{1}{(4\pi c^2 t)^{n/2}} \int_{\mathbb{R}^n} \exp\left(-\frac{|x-y|^2}{4c^2 t}\right) u_0(y) \, \mathrm{d}y
    \end{equation*}
    is the solution of \eqref{Section-4-CDE}. Then $f(x,t)$ satisfies
    \begin{equation*}
      \begin{aligned}
        \begin{cases}
          \frac{\partial^2 f}{\partial t^2}(x,t)-2a\frac{\partial f}{\partial t}(x,t)+c^2\Delta_x f (x,t)= 0,\quad\quad t>0,x\in\mathbb{R}^n; \\
          f(x,0) = u_0(x).
        \end{cases}
      \end{aligned}
    \end{equation*}

  \end{prop}

  \begin{proof}
    Since the integration domain is independent of $t$ and the integrand is sufficiently regular, differentiation and integration can be interchanged:
    \[
      \frac{\partial f}{\partial t} = \int_0^\infty (\partial_t q_\mu) v \, ds, \quad
      \frac{\partial^2 f}{\partial t^2} = \int_0^\infty (\partial_t^2 q_\mu) v \, ds, \quad
      \Delta_x f = \int_0^\infty q_\mu (\Delta_x v) \, ds.
    \]
    Therefore
    \begin{align*}
        & \frac{\partial^2 f}{\partial t^2}(x,t)-2a\frac{\partial f}{\partial t}(x,t)+c^2\Delta_x f (x,t)                                                         \\
      = & \int_0^\infty (\partial_t^2 q_\mu) v \, ds-2a\int_0^\infty (\partial_t q_\mu) v \, ds+c^2\int_0^\infty q_\mu (\Delta_x v) \, ds                         \\
      = & \int_0^\infty (\partial_t^2 q_\mu) v \, ds-2a\int_0^\infty (\partial_t q_\mu) v \, ds+\int_0^\infty q_\mu \partial_s v \, ds                            \\
      = & \int_0^\infty (\partial_t^2 q_\mu) v \, ds-2a\int_0^\infty (\partial_t q_\mu) v \, ds+\left[q_\mu(s,t)v(x,s)\right]_{s=0}^{s=\infty}-\int_0^\infty (\partial_sq_\mu)  v \, ds                          \\
      = & \int_0^\infty (\partial_t^2 q_\mu) v \, ds-2a\int_0^\infty (\partial_t q_\mu) v \, ds-\int_0^\infty (\partial_sq_\mu)  v \, ds                          \\
      = & \int_0^\infty \left( \frac{\partial^2 q_\mu}{\partial t^2} - 2a \frac{\partial q_\mu}{\partial t} - \frac{\partial q_\mu}{\partial s} \right) v \, ds=0
    \end{align*}
    where we have used Proposition \ref{Section-2-prop-q-equation} in the final step.

  \end{proof}

  \begin{rem}
    Based on the harmonic extension for the relativistic Schr\"odinger operator, we can interpret the dual relativistic diffusions \eqref{Section-4-CTE} and \eqref{Section-4-CNLE}
    in another form. We outline the idea as follows.
    Consider the hyperbolic equation \eqref{Section-4-CTE}, the parabolic equation \eqref{Section-4-CDE} and the elliptic equation
    \begin{equation*}
      \begin{aligned}
        \begin{cases}
          \frac{\partial^2 f}{\partial t^2}(x,t)-2a\frac{\partial f}{\partial t}(x,t)+c^2\Delta_x f (x,t)= 0,\quad\quad t>0,x\in\mathbb{R}^n; \\
          f(x,0) = u_0(x).
        \end{cases}
      \end{aligned}
    \end{equation*}
    Once the appropriate decay restrictions on $f(x,t)$ as $t\rightarrow\infty$ are imposed, we have
    \begin{align*}
      v(x,t) & = \int_0^\infty p_\mu(s,t) u(x,s)\,ds, \\
      f(x,t) & = \int_0^\infty q_\mu(s,t) v(x,s)\,ds.
    \end{align*}

    In the case $a=0$, some authors have already explored these properties, which are collectively called transmutation theory,
    and even extended them to more abstract and nonhomogeneous forms \cite{Carroll1979book, Gzyl2002book}.

  \end{rem}

  \subsection{Duality in asymptotic behaviors}

  We collect several standard asymptotic facts to clarify, in part, the interpretation of relativity and duality of \eqref{Section-4-CTE} and \eqref{Section-4-CNLE}  through examining their asymptotic behaviors
  as $c\rightarrow \infty, t\rightarrow\infty$.
  The related properties have been extensively studied in the mathematical and physical literature.
 To be precise, the $c\rightarrow \infty$ and $t\rightarrow \infty$ limits of \eqref{Section-4-CTE} have been discussed in Remark \ref{section-3-remark-1} and the references therein; one can refer to \cite{Alonso2025PRE,Baeumer2010, Koponen1995,Rosinski2007SPA} for the
  $t\rightarrow \infty$ limit of \eqref{Section-4-CNLE}. It should be noted that several of the
  four statements below are contained, explicitly or implicitly, in the extensive literature
  cited above; nevertheless, to the best of our knowledge, the quantitative $L^\infty$
  formulations with explicit rates and first-order corrections, stated in parallel for the
  dual pair, are not explicitly available in the literature in this form. Since our purpose is
  the essential mathematical and physical interpretation of the relativity and duality of
   \eqref{Section-4-CTE} and \eqref{Section-4-CNLE}, we consider their solutions with initial
  data $u_0\in\mathcal{S}(\mathbb{R}^n)$ and do not pursue optimal function spaces in PDEs.
  For this reason, and for the reader's convenience, we collect the elementary but detailed
  computations for the two telegraph limits in Appendix B; the corresponding two limits of the
  nonlocal diffusion equation are the boundary instances $\beta=1/2$ of the general theorems in
  Section 5, as indicated in each case below.

Let
\begin{equation}\label{fundamental-solution-K}
  K_D(x,t)=\frac{1}{(4\pi D t)^{n/2}}\exp\left(-\frac{|x|^2}{4D t}\right), \qquad x\in\mathbb{R}^n, t>0
\end{equation}
be the standard heat kernel with diffusion coefficient $D := \frac{\hbar}{2m}$. Moreover, we adopt the semigroup
convention
\[
K_D(\cdot,0)\ast u_0:=u_0.
\]

  \begin{itemize}
    \item  Case 1: $c\rightarrow \infty$.

  From a physical point of view, when we construct a relativistic model, a question naturally arises:
  How can the classical diffusion equation be obtained in the $c\rightarrow\infty$ limit?
  As shown below, both equations converge to the classical diffusion equation as $c\rightarrow\infty$.

    Let $u_0\in\mathcal{S}(\mathbb{R}^n)$, and let $u_c(x,t)$ and $f_c(x,t)$ be the solutions of the
    Cauchy problems
    \[
     \frac{D}{c^2} \frac{\partial^2 u}{\partial t^2}+\frac{\partial u}{\partial t}=D\Delta u,
     \qquad u(x,0)=u_0(x),\quad \partial_t u(x,0)=0,
    \]
   and
    \[
     \frac{\partial f}{\partial t}=\left[\frac{mc^2}{\hbar}-\sqrt{\frac{m^2c^4}{\hbar^2}-c^2\Delta}\right]f,
     \qquad f(x,0)=u_0(x),
    \]
   respectively. Then there exists a constant $C=C(D,u_0)$ such that, for all sufficiently
   large $c$,
    \begin{equation}\label{Section-4-c-infinity-convergence}
      \bigl\|u_c(\cdot,t)-K_D(\cdot,t)*u_0\bigr\|_{L^\infty(\mathbb{R}^n)}\le\frac{C}{c^{2}},
      \qquad
      \bigl\|f_c(\cdot,t)-K_D(\cdot,t)*u_0\bigr\|_{L^\infty(\mathbb{R}^n)}\le\frac{C}{c^{2}},
      \qquad t\ge0;
    \end{equation}
   in particular, as $c\to\infty$ both $u_c$ and $f_c$ converge to $K_D(\cdot,t)*u_0$ uniformly in
   $(x,t)\in\mathbb{R}^n\times[0,\infty)$, with the same rate $O(c^{-2})$. The detailed verification
   for the telegraph equation is given in Appendix B; the estimate for $f_c$ is the boundary case
   $\beta=1/2$ of Theorem \ref{Section-6-nonrelativistic} (iii), under the Kac scaling $a=c^{2}/(2D)$.
   \item  Case 2: $t\rightarrow \infty$.

   At asymptotically long times, high-frequency modes are fully damped. Both systems governed by \eqref{Section-4-CTE} and \eqref{Section-4-CNLE} lose memory of their initial relativistic structure---be it nonlocal jumps or the finite wave speed. Consequently, relativistic corrections become negligible, the central limit theorem governs the spread, and the density profile universally approaches the ordinary
   Gaussian heat kernel of classical Brownian diffusion.
   More precisely, let $u_0\in\mathcal{S}(\mathbb{R}^n)$. Then as $t\to\infty$,
   \[
      \|u(\cdot,t)-K_D(\cdot,t)*u_0\|_{L^\infty(\mathbb{R}^n)}\le Ct^{-(n+2)/2}, \quad
      \|f(\cdot,t)-K_D(\cdot,t)*u_0\|_{L^\infty(\mathbb{R}^n)}\le Ct^{-(n+2)/2},\quad t\to\infty
   \]
where $K_D(x,t)$ is defined in \eqref{fundamental-solution-K} and $D=\hbar/(2m)$. The telegraph
   estimate is verified in Appendix B; the estimate for $f$ is the boundary case $\beta=1/2$ of
   Theorem \ref{Section-6-long-time-Gaussian}, with $\gamma=\frac{1}{2a}=D$.


 \end{itemize}

  For further comparisons between \eqref{Section-4-CTE} and \eqref{Section-4-CNLE} from other physical perspectives, e.g., the asymptotic behaviors as $m\rightarrow 0,  t\rightarrow 0$ and other regimes, see \cite{Alonso2025PRE,Baeumer2010,Muniz2015PRD}.

  \section{Generalized dual diffusion: general subordination}

  Now we extend the concept of dual relativistic diffusion in two ways:
  (i) Replace  $\mu(\lambda)=-a+\sqrt{a^2+\lambda}$ with a general Laplace exponent $\varphi(\cdot)$;
  (ii) Replace the diffusion semigroup $e^{t\Delta}$ with a general bounded $C_0$-semigroup.

  In this section, we suppose \(\varphi: (0,\infty)\to (0,\infty)\) is a Laplace exponent satisfying \(\lim_{\lambda\to\infty}\varphi(\lambda)=\infty\).
  Therefore $\varphi'(\lambda)>0$ for any $\lambda>0$, which implies that the inverse \(\varphi^{-1}\) is well-defined on \((0,\infty)\).
  Moreover, we denote the inverse function of $\varphi$ simply by $\lambda(\cdot)$, i.e.,
  $\lambda(\varphi(s))=s$ for all $s>0$; the free variable of $\lambda(\cdot)$ will be named
  according to local convenience.


 \subsection{Generalized dual diffusion families of operators}

  \begin{defn}\label{Section-5-definition-1}
    Let $T(t)$ be a bounded $C_0$-semigroup  generated by a densely defined linear operator $A$ on a Banach space $X$, and let $\varphi(\cdot)$ be a Laplace exponent.
    Let $Q_\varphi(s,t)$ and $P_\varphi(s,t)$ be the distributions of the subordinator and the inverse subordinator with Laplace exponent $\varphi(\cdot)$,
    respectively. Moreover, let
    \begin{equation}\label{Section-5-definition-2}
      \begin{aligned}
        \begin{cases}
          S_\varphi(t)x=\int_0^\infty T(s)x Q_\varphi(ds,t), \quad t>0,x\in X; \\
          S_\varphi(0)=I
        \end{cases}
      \end{aligned}
    \end{equation}
    and suppose there exists a strongly continuous operator family $R_\varphi(t)$ satisfying
    \begin{equation}\label{Section-5-definition-3}
      \begin{aligned}
        \begin{cases}
          T(t)x=\int_0^\infty R_\varphi(s)x P_\varphi(ds,t), \quad t>0, x\in X; \\
          R_\varphi(0)=I.
        \end{cases}
      \end{aligned}
    \end{equation}
    Then $R_\varphi(t)$ and $S_\varphi(t)$ are called
      {\it generalized dual diffusion families of operators} relative to the semigroup $T(t)$ and the Laplace exponent $\varphi(\cdot)$.
    Moreover, $R_\varphi(t)$ is called the time-dual of $S_\varphi(t)$, while  $S_\varphi(t)$ is called the space-dual of $R_\varphi(t)$.

  \end{defn}



  \begin{rem}

     (1)  Since many subordinators have no positive exponential moments, we suppose $T(t)$ is bounded in
    Definition \ref{Section-5-definition-1} to ensure that \eqref{Section-5-definition-2} is well-defined. The boundedness of $T(t)$ is a commonly used condition in the Bernstein functional calculus \cite{Gomilko2015,Schilling2012}.

     (2)  The $R_\varphi$ in \eqref{Section-5-definition-3} need not exist.
  \end{rem}

  Next, we attempt to explore the existence of $R_\varphi(t)$. We first derive the formula formally and then justify it. Assume for the moment that a family \(R_\varphi\) satisfying \eqref{Section-5-definition-3} exists and is sufficiently integrable. Taking the Laplace transform in \(t\) on both sides of \eqref{Section-5-definition-3}, we formally obtain
\[
  (\lambda - A)^{-1}x = \int_0^\infty e^{-\lambda t} \left(\int_0^\infty R_\varphi(s)x \, P_\varphi(\mathrm ds,t)\right) \mathrm dt.
\]
Using a formal interchange of integrals and the fundamental identity for the inverse subordinator,
\[
  \int_0^\infty e^{-\lambda t} P_\varphi(\mathrm ds,t)\,\mathrm dt = \frac{\varphi(\lambda)}{\lambda} e^{-\varphi(\lambda) s}\,\mathrm ds,
\]
we obtain
\[
  (\lambda  - A)^{-1}x = \int_0^\infty R_\varphi(s)x \, \frac{\varphi(\lambda)}{\lambda} e^{-\varphi(\lambda) s}\,\mathrm ds.
\]
Then we have
\[
  \int_0^\infty e^{-\varphi(\lambda) s} R_\varphi(s)x \, \mathrm ds=\frac{\lambda}{\varphi(\lambda)}(\lambda  - A)^{-1}x.
\]
Now we use the inverse function notation $\lambda=\lambda(\varphi)$ and regard $\varphi$ as the independent variable,
\[
  \int_0^\infty e^{-\varphi s} R_\varphi(s)x \, \mathrm ds=\frac{\lambda(\varphi)}{\varphi}(\lambda(\varphi) - A)^{-1}x.
\]

  To ensure the existence of the operator family $R_\varphi(t)$, we need to prove that the vector-valued function
  \[
    G(\varphi):=\frac{\lambda(\varphi)}{\varphi}(\lambda(\varphi) - A)^{-1}
  \]
  is the Laplace transform of a function, which requires certain restrictions to be imposed on $A$ and $\lambda(\varphi)$ (or $\varphi(\lambda)$). We recall some preliminary materials.

  \begin{defn}
      An operator $B$ on a Banach space $X$ is {\it sectorial} of angle $\omega\in[0,\pi)$ ($B\in \text{Sect}(\omega)$, in short) if the following conditions hold:

     (i) $B$ is densely defined and closed.

     (ii) $\sigma(B)$ is contained in the closure of the sector
      \[
       \Sigma_\omega :=\{z\in\mathbb{C}: z\neq 0\, \text{and} \,|\arg z|<\omega\}
      \]
      for $\omega>0$ or $\Sigma_0:=(0,\infty)$.

     (iii) For every $\omega'\in(\omega,\pi)$, $\sup\{\|z(z-B)^{-1}\|: z\in \mathbb{C}\setminus \overline{\Sigma_{\omega'}} \}<\infty$.

  \end{defn}

See \cite{Batty2017Advance,Gomilko2015,Haase2006book,Limiao2010} for the definition and functional calculus of sectorial operators. It is well known that $B\in \text{Sect}(\omega)$ for some $\omega\in [0,\pi/2)$ if and only if $-B$ is the generator of a sectorially bounded holomorphic $C_0$ semigroup on $X$ of angle $\pi/2-\omega$
\cite{Batty2017Advance,Haase2006book}.

The following characterization of the complex Laplace inversion transform is cited from \cite[Theorem 0.1]{Pruss1993book},
see also \cite[Theorem 2.6.1]{Arendt2010book}.

  \begin{lem} \label{complex-inversion-Laplace-transform}
      Let $F:(0,\infty)\to X$ and $\theta_0 \in (0, \pi/2]$. Then the following are equivalent:

      (i) There is a function $u(z)$ holomorphic on the sector $\Sigma_{\theta_0}$ and bounded on each sector
      $\Sigma_{\theta}\,(0<\theta<\theta_0)$ such that $F(\lambda)=\hat{u}(\lambda)$ for each $\lambda>0$.

      (ii) $F(\lambda)$ admits a holomorphic extension to the sector $\Sigma_{\pi/2+\theta_0}$,
      and $\lambda F(\lambda)$ is bounded on each sector $\Sigma_{\pi/2+\theta}\, (0<\theta<\theta_0)$.

  \end{lem}

The following lemma, which follows from the open mapping theorem and the inverse function theorem, is a standard result in complex analysis.
\begin{lem} \label{holomorphic-inverse-function}
Let $\Omega\subset \mathbb{C}$ be a domain (in particular, $\Omega=\mathbb{C}\setminus(-\infty,0])$. Let
$\varphi: \Omega\to \mathbb{C}$ be a holomorphic and univalent function. Then the inverse function
\[
\varphi^{-1}: D:=\varphi(\Omega) \to \Omega
\]
exists and is holomorphic on the domain $D$.
\end{lem}

The following theorem establishes the existence and uniqueness of $R_\varphi(t)$.

\begin{thm} \label{existence-R(t)-theorem-1}
 Suppose the following three conditions are satisfied:

 (i)   $-A\in \text{Sect}(\omega_1)$, $\omega_1\in [0,\pi/2)$, and \(T(t)\) is the bounded \(C_0\)-semigroup generated by $A$;

 (ii)  $\varphi(\cdot)\in \mathcal{CBF}$, and the holomorphic extension to $\mathbb{C}\setminus(-\infty,0]$ is
       univalent;

 (iii) There exist $\omega$  and  $\omega_2$  such that $\omega_1<\omega_2<\omega<\frac{\pi}{2}$ and
       $\Sigma_{\pi-\omega} \subset \varphi(\Sigma_{\pi-\omega_2})$.

Then there exists a unique strongly continuous and bounded operator family \(R_\varphi:[0,\infty)\to B(X)\) satisfying the relation \eqref{Section-5-definition-3}.
\end{thm}

\begin{proof}
  \emph{Step 1: sectorial resolvent bounds and the inverse function.} Since $-A\in
  \text{Sect}(\omega_1)$, $\omega_1\in [0,\pi/2)$, we have:

(1) $(\lambda-A)^{-1}$ is holomorphic on the sector $\Sigma_{\pi-\omega_1}$;

(2) for every $\omega_1'\in (\omega_1,\pi/2)$, there exists $M(A,\omega_1')<\infty$ such that
\[
\|\lambda(\lambda-A)^{-1}\|\leq M(A,\omega_1'), \qquad \lambda\in\Sigma_{\pi-\omega_1'}.
\]

By condition (ii) and Lemma \ref{holomorphic-inverse-function}, $\lambda=\lambda(\varphi)$, the inverse function
of $\varphi=\varphi(\lambda)$, is holomorphic on
\[
 D:=\varphi(\mathbb{C}\setminus(-\infty,0]).
\]

  \emph{Step 2: holomorphy and boundedness of $G$ on $\Sigma_{\pi-\omega}$.} By condition (iii), for every $\varphi\in\Sigma_{\pi-\omega}$, we have $\lambda(\varphi)\in \Sigma_{\pi-\omega_2}\subset\Sigma_{\pi-\omega_1}$. Hence $(\lambda(\varphi)-A)^{-1}$ is holomorphic in $\varphi$ on the sector $\Sigma_{\pi-\omega}$.
Moreover, $\lambda(\varphi)/\varphi$ is holomorphic in $\varphi$ on the sector $\Sigma_{\pi-\omega}$. Therefore, $G$
is holomorphic on the sector $\Sigma_{\pi-\omega}$.

Furthermore, for every $\omega'\in(\omega,\pi/2)$, we have
\begin{align*}
\sup\{\|\varphi G(\varphi)\|: \varphi\in \Sigma_{\pi-\omega'} \}
     &=\sup\{\|\lambda(\varphi) (\lambda(\varphi) - A)^{-1}\|: \varphi\in \Sigma_{\pi-\omega'} \}\\
&\leq\sup\{\|\lambda(\lambda- A)^{-1}\|: \lambda\in \Sigma_{\pi-\omega_2} \}\\
&\leq M(A,\omega_2)<\infty.
\end{align*}

  \emph{Step 3: construction of the candidate family.} Applying Lemma \ref{complex-inversion-Laplace-transform} with $\theta_0=\pi/2-\omega$, there is a function $R_\varphi(z)$ holomorphic on the sector $\Sigma_{\theta_0}$ and bounded on each sector
$\Sigma_{\theta}\,(0<\theta<\theta_0)$ such that $G(\lambda)=\widehat{R_\varphi}(\lambda)$ for each $\lambda>0$.
Hence $R_\varphi(t) (t>0)$ is a strongly continuous operator family satisfying $\|R_\varphi(t)\|\leq M$, where
$M$ depends on $A$ and $\varphi$. Fix \(x\in X\) and \(t>0\). The integral \(\int_0^\infty R_\varphi(s)x \, P_\varphi(\mathrm ds,t)\) is well-defined as a Bochner integral because
\[
\int_0^\infty \|R_\varphi(s)x\| \, P_\varphi(\mathrm ds,t)
\le M\|x\| \int_0^\infty P_\varphi(\mathrm ds,t) = M\|x\| < \infty.
\]
  \emph{Step 4: identification with $T(t)$ via the Laplace transform.} Now consider the Laplace transform in \(t\) of this candidate:
\[
\int_0^\infty e^{-\lambda t}
\left( \int_0^\infty R_\varphi(s)x \, P_\varphi(\mathrm ds,t) \right) \mathrm dt.
\]

Since
\[
\int_0^\infty\int_0^\infty \left\| e^{-\lambda t} R_\varphi(s)x \right\| \, P_\varphi(\mathrm ds,t)\,\mathrm dt
\le M\|x\| \int_0^\infty e^{-\lambda t}\,\mathrm dt
= \frac{M\|x\|}{\lambda} < \infty,
\]
we apply Fubini's theorem to obtain
\[
\begin{aligned}
\int_0^\infty e^{-\lambda t}\left( \int_0^\infty R_\varphi(s)x \, P_\varphi(\mathrm ds,t) \right) \mathrm dt
&= \int_0^\infty R_\varphi(s)x \, \frac{\varphi(\lambda)}{\lambda} e^{-\varphi(\lambda) s}\,\mathrm ds \\
&= \frac{\varphi(\lambda)}{\lambda} G(\varphi(\lambda)) x \qquad \\
&= \frac{\varphi(\lambda)}{\lambda} \cdot \frac{\lambda}{\varphi(\lambda)} (\lambda I - A)^{-1}x \\
&= (\lambda I - A)^{-1}x.
\end{aligned}
\]
But \((\lambda I - A)^{-1}x\) is exactly the Laplace transform of \(T(t)x\). By the uniqueness of the
vector-valued Laplace transform, the identity
\[
\int_0^\infty R_\varphi(s)x \, P_\varphi(\mathrm ds,t) = T(t)x
\]
holds for almost every $t>0$.
To obtain equality for every $t$, note that $\varphi(\infty)=\infty$ implies that the associated subordinator has positive drift or infinite activity and hence is strictly increasing; its inverse has continuous paths. Since $s\mapsto R_\varphi(s)x$ is bounded and continuous, dominated convergence shows that $t\mapsto\mathbb E[R_\varphi(E_\varphi(t))x]$ is continuous. It therefore agrees with the continuous function $T(t)x$ for every $t>0$. The uniqueness of $R_\varphi(t) (t>0)$ is derived from the uniqueness of the Laplace transform.

  \emph{Step 5: strong continuity at $t=0$.} It remains to show that
\begin{equation}\label{Section-5-strong-continuity}
  \lim_{t\to 0+} R_\varphi(t)x=x, \qquad x\in X.
\end{equation}
By Lemma \ref{complex-inversion-Laplace-transform}, the function $t\mapsto R_\varphi(t)x$ is holomorphic on the
sector $\Sigma_{\theta_0}$ and bounded on each subsector $\Sigma_{\theta}$ ($0<\theta<\theta_0$), i.e., we are in
the situation of \cite[Theorem 2.6.1]{Arendt2010book}. Since
\[
  \int_0^\infty e^{-\varphi t}R_\varphi(t)x\,dt=G(\varphi)x=\frac{\lambda(\varphi)}{\varphi}(\lambda(\varphi) - A)^{-1}x,
\]
we have
\begin{align*}
  \lim_{\varphi\to +\infty} \varphi G(\varphi)x
  &=\lim_{\varphi\to +\infty} \lambda(\varphi)(\lambda(\varphi) - A)^{-1}x\\
  &=\lim_{\lambda\to +\infty} \lambda(\lambda- A)^{-1}x\\
  &=x,
\end{align*}
where we used that $\lambda(\varphi)\to+\infty$ as $\varphi\to+\infty$ and the standard resolvent limit for
generators of bounded $C_0$ semigroups. The Tauberian theorem for sectorially holomorphic functions
\cite[Theorem 2.6.4 (a)]{Arendt2010book} then yields \eqref{Section-5-strong-continuity}.
Define $R_\varphi(0):=I$, so that the initial condition in \eqref{Section-5-definition-3} is satisfied.

Therefore, $R_\varphi(t)$ $(t\geq0)$ is the desired operator family.

\end{proof}

 \begin{rem}

The next subsection gives examples illustrating the hypotheses of Theorem \ref{existence-R(t)-theorem-1}. Nevertheless, the theorem does not cover every case of interest.
 \end{rem}

  \begin{rem}\label{Section-5-rem-1}

    Let $x\in D(A)$. Then  $v(t)=T(t)x$ is the unique strong solution of
    \begin{equation}\label{section-5-v}
      \begin{aligned}
        \begin{cases}
          v'(t)=Av(t), \quad t>0, \\
          v(0)=x.
        \end{cases}
      \end{aligned}
    \end{equation}
    Furthermore,  $-\varphi(-A)$ is well-defined as in Subsection 2.2.
    Therefore, $f(t)=S_\varphi(t)x$   is the unique strong solution of
    \begin{equation}\label{section-5-f}
      \begin{aligned}
        \begin{cases}
          f'(t)=-\varphi(-A)f(t), \quad t>0, \\
          f(0)=x.
        \end{cases}
      \end{aligned}
    \end{equation}
    Given $\varphi(\cdot)$ and $A$, according to \eqref{Section-5-definition-3}, we may seek an equation
    \begin{equation}\label{section-5-u}
      \begin{aligned}
        \begin{cases}
          \Phi_t u(t)=Au(t),\quad\quad t>0, \\
          u(0)=x,\quad u'(0)=0,\quad u''(0)=0,\ \dotsc
        \end{cases}
      \end{aligned}
    \end{equation}
   where $u(t)=R_\varphi(t)x$.
    In general, the operator $\Phi_t$ is a linear integro-differential operator acting on $u(t)$ with respect to the time variable $t$.
    Therefore, \eqref{section-5-u} is a nonlocal-in-time equation in general.
    We may regard \eqref{section-5-f} and \eqref{section-5-u} as {\it generalized dual diffusion equations} with respect to the Laplace exponent $\varphi(\cdot)$.
    Moreover, \eqref{section-5-u} is the time-dual of \eqref{section-5-f}, while \eqref{section-5-f} is the space-dual of \eqref{section-5-u}.
    If $\varphi(\lambda)=\mu(\lambda)=-a+\sqrt{a^2+\lambda}$, $A=c^2\Delta$, then \eqref{Section-5-definition-2} and \eqref{Section-5-definition-3}
    reduce to \eqref{Section-4-subordination-expression},
    while \eqref{section-5-v}, \eqref{section-5-f} and \eqref{section-5-u} reduce to
    \eqref{Section-4-CDE}, \eqref{Section-4-CNLE} and \eqref{Section-4-CTE}, respectively.

  \end{rem}

  Let $\varphi$ be an invertible Laplace exponent with drift $b$ and L\'evy measure $\sigma$ satisfying $\sigma(0,\infty)=\infty$.
  Set $x=v(0)=u(0)$.
  According to Definition \ref{Section-5-definition-1} and Remark \ref{Section-5-rem-1},
  we describe the derivation of the time operator $\Phi_t$ in the nonlocal-in-time equation \eqref{section-5-u} for a given $\varphi$.

  Since
  \[
    \int_0^\infty e^{-\varphi t} R_\varphi(t)x \, \mathrm dt=\frac{\lambda(\varphi)}{\varphi}(\lambda(\varphi) - A)^{-1}x
  \]
  and $u(t)=R_\varphi(t)x,u'(0)=u''(0)=\cdots=0$, we have
  \[
    (\lambda(\varphi)-A)\int_0^\infty  e^{-t \varphi} u(t)dt=\frac{\lambda(\varphi)}{\varphi}x.
  \]
  Therefore,
  \begin{equation}\label{Section-5-procedure-1}
    \lambda(\varphi)\int_0^\infty  e^{-t \varphi} u(t)dt-\frac{\lambda(\varphi)}{\varphi}x=A\int_0^\infty  e^{-t \varphi} u(t)dt.
  \end{equation}


  Taking Laplace transforms in $t$ on both sides of the equation $\Phi_t u(t)=Au(t)$ gives
  \begin{align}\label{Section-5-procedure-2}
    \int_0^\infty  e^{-\varphi t }\Phi_t u(t)dt & = \int_0^\infty  e^{-\varphi t }Au(t)dt \notag                                                     \\
     & = A\int_0^\infty  e^{-\varphi t }u(t)dt \notag                                                                                                \\
     & = \lambda(\varphi)\int_0^\infty  e^{-\varphi t } u(t)dt-\frac{\lambda(\varphi)}{\varphi}x, \quad (by \, \eqref{Section-5-procedure-1}).
  \end{align}

  Now, we have several choices depending on $\varphi(\lambda)$. For instance,

  (1)  If there exists a locally integrable function $h_1 (t)$ such that
  \[
    \int_0^\infty  e^{-\varphi t} h_1 (t) dt=\frac{\lambda(\varphi)}{\varphi^2},
  \]
 rewrite the equation as
  \begin{align*}
    \int_0^\infty  e^{-\varphi t }\Phi_t u(t)dt&=\frac{\lambda(\varphi)}{\varphi^2}\left[\varphi^2\int_0^\infty  e^{-\varphi t } u(t)dt-\varphi x\right]\\
    &=\frac{\lambda(\varphi)}{\varphi^2}\cdot \int_0^\infty  e^{-\varphi t } u''(t)dt.
  \end{align*}
  Then
  \[
    \int_0^\infty  e^{-\varphi t}\Phi_t u(t)dt=\int_0^\infty  e^{-\varphi t} h_1 (t) dt \cdot  \int_0^\infty  e^{-\varphi t} u''(t)dt=\int_0^\infty e^{-\varphi t} (h_1\ast u'')(t)dt.
  \]
  Therefore
  \[
    \Phi_t u(t)=(h_1\ast u'')(t).
  \]

  (2)  If there exists a locally integrable function $h_2 (t)$ such that
\[
  \int_0^\infty  e^{-\varphi t} h_2 (t) dt=\frac{\lambda(\varphi)}{\varphi^3},
\]
rewrite the equation as
\begin{align*}
  \int_0^\infty  e^{-\varphi t }\Phi_t u(t)dt&=\frac{\lambda(\varphi)}{\varphi^3}\left[\varphi^3\int_0^\infty  e^{-\varphi t } u(t)dt-\varphi^2 x\right]\\
  &=\frac{\lambda(\varphi)}{\varphi^3}\cdot \int_0^\infty  e^{-\varphi t } u^{(3)}(t)dt.
\end{align*}
Then
\[
  \int_0^\infty  e^{-\varphi t}\Phi_t u(t)dt=\int_0^\infty  e^{-\varphi t} h_2 (t) dt \cdot  \int_0^\infty  e^{-\varphi t} u^{(3)}(t)dt
  =\int_0^\infty e^{-\varphi t} (h_2\ast u^{(3)})(t)dt.
\]
Therefore
\[
  \Phi_t u(t)=(h_2\ast u^{(3)})(t).
\]

\begin{rem}[Relation to inverse-subordinator time changes of semigroups]\label{Section-5-rem-toaldo}
  The construction of $R_\varphi(t)$ and $\Phi_t$ is closely related to, but should be distinguished from, the
  inverse-subordinator time changes of semigroups developed in
  \cite{ChenZhenQing2020JFA,MeerschaertToaldo2019,Toaldo2015}. Given a Bernstein function
  $\nu$ whose L\'evy measure is infinite, let $\mathbf{E}_\nu(t)$ be the corresponding inverse subordinator.
  For a bounded $C_0$-semigroup $T(t)$ generated by $A$, those works consider the time-changed family
  \begin{equation}\label{Section-5-toaldo}
    w(t)=\mathbb{E}\bigl[T(E_\nu(t))x\bigr]=\int_0^\infty T(s)x\,P_\nu(ds,t)
  \end{equation}
  and show that $w$ solves the time-fractional equation $\mathcal{D}_t^\nu w=Aw$, $w(0)=x$, where
  $\mathcal{D}_t^\nu$ is the convolution-type derivative built from the tail of the L\'evy measure of $\nu$
  \cite{Toaldo2015}; its Laplace symbol is $\nu(\lambda)$ itself, so that
  $\hat w(\lambda)=\frac{\nu(\lambda)}{\lambda}(\nu(\lambda)-A)^{-1}x$. For $\nu(\lambda)=\lambda^\alpha$ this
  is the Caputo derivative of order $\alpha\in(0,1)$, and \eqref{Section-5-toaldo} is subdiffusive. The
  viewpoint of the present paper is reversed: the semigroup is not the input but the output---it is
  \emph{recovered} from the time-nonlocal evolution by the inverse subordinator,
  $T(t)x=\mathbb{E}[R_\varphi(E_\varphi(t))x]$, cf. \eqref{Section-5-definition-3}, and the equation
  $\Phi_t u=Au$ has the Laplace symbol $\lambda(\varphi)$, the \emph{inverse function} of $\varphi$, so that
  $\hat u(\varphi)=\frac{\lambda(\varphi)}{\varphi}(\lambda(\varphi)-A)^{-1}x$. Both constructions rely on the
  same fundamental identity $\int_0^\infty e^{-\lambda t}p_\psi(s,t)\,dt=\frac{\psi(\lambda)}{\lambda}e^{-s\psi(\lambda)}$,
  and the two resolvent formulas coincide in form, with $\psi=\nu$ in \eqref{Section-5-toaldo} and
  $\psi=\lambda(\varphi)$ here. Two differences are essential. First, $\lambda(\varphi)$ is typically
  \emph{not} a Bernstein function (it grows superlinearly), so $\Phi_t$ generally falls outside the
  convolution-derivative framework of \cite{Toaldo2015}: for $\varphi(\lambda)=\lambda^\alpha$, their
  construction yields the Caputo derivative of order $\alpha\in(0,1)$, whereas our $\Phi_t$ is the Caputo
  derivative of order $1/\alpha\in(1,2)$ (Example \ref{general-example-2}); the question of whether a
  specific $\lambda(\varphi)$ is nevertheless Bernstein is exactly the problem posed in Remark
  \ref{Section-5-rem-open-problem}. Second, the long-time behaviors are opposite: \eqref{Section-5-toaldo} is
  subdiffusive, while the time-dual evolution $R_\varphi(t)$ is typically superdiffusive or wave-like (e.g.,
  finite propagation for $\varphi=\mu$; cf. also Theorem \ref{Section-6-long-time-anomalous}, where the two
  sides share the same anomalous exponent with complementary profiles).
\end{rem}

\subsection{Examples}

     All subordinators and inverse subordinators in this subsection have densities. Therefore, the subordination relations \eqref{Section-5-definition-2} and
    \eqref{Section-5-definition-3} can be written in simpler forms:
     \[
     S_\varphi(t)x=\int_0^\infty q_\varphi(s,t)T(s)x\,ds, \qquad T(t)x=\int_0^\infty p_\varphi(s,t)R_\varphi(s)x\,ds
     \]
    where $q_\varphi(s,t)$ and $p_\varphi(s,t)$ are the densities of
    subordinators $\mathbf{D}_\varphi(t)$ and inverse subordinators $\mathbf{E}_\varphi(t)$, respectively.

   We first consider a special example. Although it does not satisfy the conditions of Theorem \ref{existence-R(t)-theorem-1}, it still fulfills the requirements of the Definition \ref{Section-5-definition-1} and illustrates the motivation behind the concepts introduced therein.

 \begin{example}\label{general-example-1}
   Let $T(t)=e^{t\Delta}$ be the semigroup generated by $A=\Delta$ on $L^2(\mathbb{R}^n)$,  and let $\zeta(\lambda)=\lambda^{1/2}$ be a Laplace exponent. Moreover, let
   $q_\zeta(s,t)$ and $p_\zeta(s,t)$ be the densities of the $1/2$-stable subordinator and the inverse $1/2$-stable subordinator, respectively (see their explicit
   expressions in \eqref{q-zeta-function} and \eqref{p-zeta-function}).
   Then
   \[
     S_\zeta(t)g=e^{-t\sqrt{-\Delta}}g=\int_0^\infty q_\zeta(s,t) T(s)g\,ds, \qquad t>0, g\in L^2(\mathbb{R}^n)
   \]
   and
   \[
     T(t)g=\int_0^\infty p_\zeta(s,t) R_\zeta(s)g \,ds, \qquad t>0, g\in L^2(\mathbb{R}^n)
   \]
   where $ R_\zeta(t)=\text{C} (t)$ is the cosine function generated by $A$. The cosine function and the abstract Weierstrass formula (subordination from the cosine function to a $C_0$-
   semigroup) have been widely studied \cite[Theorem 3.14.17, Formula 3.102]{Arendt2010book}.

   Furthermore, let $f(x,t)=S_\zeta(t)g(x)$ and $u(x,t)=R_\zeta(t)g(x)$. Then we have
    \begin{equation*}
      \begin{aligned}
        \begin{cases}
          \frac{\partial f}{\partial t}(x,t)=-(-\Delta)^{1/2}f(x,t),\quad\quad t>0, x\in \mathbb{R}^n; \\
          f(x,0)=g(x),
        \end{cases}
      \end{aligned}
    \end{equation*}
  and
    \begin{equation*}
      \begin{aligned}
        \begin{cases}
          \Phi_t u(x,t)=\frac{\partial^2 u}{\partial t^2}(x,t)=\Delta u(x,t),\quad\quad t>0,x\in \mathbb{R}^n; \\
           u(x,0)=g(x),\quad u'(x,0)=0.
        \end{cases}
      \end{aligned}
    \end{equation*}
  \end{example}

We now turn to several prototypical examples that motivate this paper. These examples do satisfy all the conditions of Theorem \ref{existence-R(t)-theorem-1}. One may also compute them directly.

\begin{example}\label{general-example-2}

  Let $\varphi(\lambda)=\lambda^{\alpha} (0<\alpha<1)$, and suppose that $-A$ is a sectorial operator on a Banach space $X$.
  Then $\varphi(\cdot)\in \mathcal{CBF}$ and its holomorphic extension to $\mathbb{C}\setminus(-\infty,0]$ is univalent.
  Since the image of $\varphi(\cdot)$  on $\mathbb{C}\setminus (-\infty,0]$  is the sector
 \[
   D_\alpha=\{w\in\mathbb C:\ |\arg w|<\alpha\pi\},
 \]
 the condition (iii) in Theorem \ref{existence-R(t)-theorem-1} is satisfied provided that the sectorial angle $\omega_1$ of $-A$ is small enough (see item (2) below for a precise range).

  Let $T(t)$ be the bounded $C_0$ semigroup generated by $A$. We have

  (1) The $C_0$ semigroup generated by $-(-A)^\alpha$ is given by \cite{Yosida1960}
  \[
    S_\varphi(t)x=\int_0^\infty q_\varphi (s,t) T(s)xds, \qquad t>0, x\in X
  \]
  where
  \[
    q_\varphi(s,t)=\frac{1}{\pi}\int_0^\infty e^{sr\cos\theta-tr^\alpha\cos\alpha\theta}
                        \sin(sr\sin\theta-tr^\alpha\sin\alpha\theta+\theta)dr, \qquad \theta\in[\pi/2,\pi]
  \]
 and
  \[
    \widehat{q_\varphi} (\lambda,t)=\int_0^\infty e^{-\lambda s} q_\varphi (s,t)ds =e^{-t \lambda^\alpha}.
  \]

  (2) If $1/2<\alpha<1$ and the sectorial angle of $-A$ satisfies $\omega_1<\pi\bigl(1-\frac{1}{2\alpha}\bigr)$, we have
  \[
    T(t)x=\int_0^\infty p_\varphi(s,t) R_\varphi(s)xds, \qquad t>0, x\in X
  \]
  where
  \[
    \widehat{p_\varphi} (s,\lambda)=\int_0^\infty e^{-\lambda t} p_\varphi (s,t)dt =\lambda^{\alpha-1}e^{-s\lambda^\alpha},
    \widehat{p_\varphi} (\lambda,t)=\int_0^\infty e^{-\lambda s} p_\varphi (s,t)ds =E_\alpha(-\lambda t^\alpha)
  \]
  and $R_\varphi(t)$ is a \((1/\alpha)\)-times fractional resolvent family \cite{Limiao2010}.
  See \cite[Theorem 3.1 and Corollary 3.3]{Limiao2010} for the integral representation of $q_\varphi (s,t)$, $p_\varphi (s,t)$ and more related results.
    Moreover,
  \[
    \Phi_t u(t)=D_t^{1/\alpha}u(t)=\int_0^t\frac{(t-s)^{1-\frac{1}{\alpha}}}{\Gamma(2-\frac{1}{\alpha})} u''(s)ds
  \]
  where $D_t^{1/\alpha}$ denotes the Caputo time-fractional derivative of order $1/\alpha$.
  The quantitative long-time asymptotic duality of this example will be studied in Theorem
  \ref{Section-6-long-time-anomalous} (Section 5.1).

\end{example}

  If $\alpha=1/2$, one can compute the above example directly as in Example  \ref{general-example-1}.

\begin{example}\label{general-example-3}
  Consider
  \begin{equation}\label{Section-5-relativistic-stable-1}
    \dfrac{\partial f}{\partial t}(x,t)=[a-(a^{1/\beta}-\Delta_x)^{\beta}]f(x,t),  \quad \beta\in(0,1)
  \end{equation}
  with the relativistic stable operator $a-(a^{1/\beta}-\Delta_x)^{\beta}$; see \cite{Alrawashdeh2017,Applebaum2009,ChenZhenQing2012AOP,Kumar2015}.

  Let $\mathbf{D}_\kappa(t)$ be the subordinator with Laplace exponent
  \[
    \kappa(\lambda)=-a+(\lambda+a^{1/\beta})^{\beta}, \quad\quad \beta\in(0,1).
  \]
 We have
  \[
    \mathbb{E}(e^{-\lambda \mathbf{D}_\kappa(t)})=\int_0^\infty e^{-\lambda s} q_\kappa(s,t)ds=e^{-t[-a+(\lambda+a^{1/\beta})^{\beta}]}.
  \]
  Then the relativistic stable process governed by \eqref{Section-5-relativistic-stable-1} can be regarded as $\mathbf{B}(\mathbf{D}_\kappa(t))$
  where $\mathbf{B}(t)$ and $\mathbf{D}_\kappa(t)$ are independent.

  Let $\mathbf{E}_\kappa(t)$ be the inverse of $\mathbf{D}_\kappa(t)$, $p_\kappa(s,t)$ be the probability density of $\mathbf{E}_\kappa(t)$, and
  \[
    \int_0^\infty e^{-\lambda t}p_\kappa(s,t)dt=\frac{\kappa(\lambda)}{\lambda}e^{-s \kappa(\lambda)}.
  \]

 We now verify condition (iii) of Theorem \ref{existence-R(t)-theorem-1} (here $\omega_1=0$, since
 $-\Delta\in\mathrm{Sect}(0)$), and show that it holds exactly for $\beta\in(1/2,1)$. The image of
 $\mathbb{C}\setminus(-\infty,-a^{1/\beta}]$ under $\kappa$ is exactly the open sector
 \[
   D=\{z\in\mathbb C:\ |\arg (z+a)|<\beta\pi\}
 \]
 ($\lambda\mapsto a^{1/\beta}+\lambda$, $w\mapsto w^{\beta}$ and $\zeta\mapsto\zeta-a$ being
 bijective at each step); its opening $2\beta\pi$ exceeds $\pi$ exactly when
 $\beta>1/2$. For the sector containment required by (iii), write the inverse function as
 \[
   \lambda(z):=(z+a)^{1/\beta}-a^{1/\beta}
   =\frac{z}{\beta}\int_0^1 (tz+a)^{\frac{1-\beta}{\beta}}\,dt,
 \]
 and note that $\Sigma_{\pi-\omega}\subset D$ for $\omega>\pi(1-\beta)$ (since
 $|\arg(z+a)|\le|\arg z|<\pi-\omega<\beta\pi$), so that $\kappa(\lambda(z))=z$ on
 $\Sigma_{\pi-\omega}$. Since a shift by a positive real number decreases the argument, $|\arg(tz+a)|<|\arg z|$ for
 $t\in[0,1]$; hence for $z\in\Sigma_{\pi-\omega}$ the integrand lies in the convex cone
 $\{|\arg\zeta|\le\frac{1-\beta}{\beta}(\pi-\omega)\}$, and therefore
 \[
   |\arg\lambda(z)|<\frac{\pi-\omega}{\beta},\qquad z\in\Sigma_{\pi-\omega}.
 \]
 If $\beta\in(1/2,1)$, choose $\omega\in(\pi(1-\beta),\pi/2)$ and set
 $\omega_2:=\pi-\frac{\pi-\omega}{\beta}$; then $0<\omega_2<\omega$ and
 $\lambda(\Sigma_{\pi-\omega})\subset\Sigma_{\pi-\omega_2}$, i.e.,
 $\Sigma_{\pi-\omega}\subset\kappa(\Sigma_{\pi-\omega_2})$, so condition (iii) holds.
 Conversely, if $\beta\le1/2$, then $\lambda(re^{i\theta})=r^{1/\beta}e^{i\theta/\beta}(1+o(1))$
 as $r\to\infty$, and for $\theta$ close to $\beta\pi$ the point $\lambda(re^{i\theta})$
 approaches the negative real axis; hence $\lambda(\Sigma_{\pi-\omega})$ is contained in no
 sector $\Sigma_{\pi-\omega_2}$ with $\omega_2>0$, and condition (iii) fails.
 The borderline case $\beta=1/2$ is not covered by Theorem \ref{existence-R(t)-theorem-1}; it is
 exactly the telegraph equation, which has been treated in Theorem \ref{Section-3-main-thm}.

  We seek a linear operator $\Phi_t$ satisfying \eqref{section-5-u} as follows.
  First, it follows from $\kappa(\lambda)=-a+(\lambda+a^{1/\beta})^{\beta}$ that $\lambda=(\kappa+a)^{1/\beta}-a^{1/\beta}$.
  Then, due to \eqref{Section-5-procedure-2}, we have
  \begin{equation}\label{20260417}
    \widehat{\Phi_t u}(x,\kappa)=\left[(\kappa+a)^{1/\beta}-a^{1/\beta}\right]\hat{u}(x,\kappa)-\frac{(\kappa+a)^{1/\beta}-a^{1/\beta}}{\kappa}u(x,0)=\Delta \hat{u}(x,\kappa).
  \end{equation}

  Next, we will consider the special case $\beta=1/2$. The equation \eqref{20260417} is reduced to
  \[
    [\kappa^2 \hat{u}(x,\kappa)-\kappa u(x,0)] +2a [\kappa \hat{u}(x,\kappa)-u(x,0)]=\Delta \hat{u}(x,\kappa),
  \]
  then, by inverting the Laplace transform, we obtain the telegraph equation
  \[
    \partial_t^2u(x,t)+2a\partial_tu(x,t)=\Delta_x u(x,t).
  \]

  Note that $\beta$ should be restricted to $[1/2,1)$ to obtain a well-posed Cauchy problem
  $\Phi_t u(x,t)=\Delta_x u(x,t)$, based on the theory of the abstract higher-order Cauchy problem \cite{Fattorini1983book}.
  When $\beta\in (1/2,1)$, $\Phi_t $ is an integro-differential operator.
  The quantitative long-time asymptotic duality and the nonrelativistic limit of this example
  will be studied in Section 5 (Theorems \ref{Section-6-long-time-Gaussian} and
  \ref{Section-6-nonrelativistic}).

\end{example}


Finally, we briefly discuss several formal examples that fall outside the scope of Definition \ref{Section-5-definition-1}. While they do not satisfy the defining conditions, their analysis sheds valuable indirect light on the duality relation under investigation and on Bernstein function theory.

The index $\alpha$  (or $\beta$) in Examples \ref{general-example-2} and \ref{general-example-3}
should be restricted to $[1/2,1)$ to obtain a well-posed Cauchy problem
  $\Phi_t u(t)=A u(t)$, in which the ``power'' of the integro-differential operator $\Phi_t$ is less than or equal to $2$.
The following examples are dedicated to exploring whether general higher-order Cauchy problems possess the structure of generalized dual diffusion. In these examples, we disregard the well-posedness of the equations and perform purely formal calculations.

\begin{example}\label{general-example-3-beta-small}
  As stated in Example \ref{general-example-3}, the index $\beta$ therein should be restricted to $[1/2,1)$.

  We have to resort to a weak notion of well-posedness for the general $\Phi_t u(x,t)=\Delta_x u(x, t)$ if $\beta\in (0,1/2)$.
  For instance, let $\beta=1/3$; \eqref{20260417} reduces to
  \[
    [\kappa^3 \hat{u}(x,\kappa)-\kappa^2 u(x,0)] +3a[\kappa^2 \hat{u}(x,\kappa)-\kappa u(x,0)] +3a^2 [\kappa \hat{u}(x,\kappa)-u(x,0)]=\Delta \hat{u}(x,\kappa).
  \]
  Then we have formally
  \begin{equation}\label{20260602}
    \Phi_t u(x,t)=\partial_t^3u(x,t)+3a\partial_t^2u(x,t)+3a^2\partial_tu(x,t)=\Delta_x u(x,t).
  \end{equation}
  In physics, researchers have already used a different third-order equation to simulate relativistic diffusion \cite{Kampen1970}.

\end{example}

\begin{example}\label{example-two-terms-high-order}
  Consider
  \begin{equation}\label{general two term equation}
    \begin{aligned}
      \begin{cases}
        \Phi_t u(t)=au^{(m+1)}(t)+b  u^{(m)}(t)=Au(t),\quad\quad t>0, \\
        u(0) = u_0, \quad u'(0)=u''(0)=\cdots =u^{(m)}(0)=0
      \end{cases}
    \end{aligned}
  \end{equation}
  where $a,b>0, m=1,2,\ldots$.

  As in the preceding computation, let
  \[
    \lambda=\lambda(\varphi)= a\varphi^{m+1}+b\varphi^{m}, \quad \varphi>0.
  \]
  This equation has a unique positive solution
  \[
    \varphi=\varphi(\lambda), \quad \lambda>0.
  \]
  For $\varphi>0$, we have
  \begin{equation*}
    \varphi'(\lambda)= \frac{1}{a(m+1)\varphi^m+bm\varphi^{m-1}}=\frac{1}{\varphi^{m-1}}\cdot\frac{1}{a(m+1)\varphi+bm}=(f_1\circ\varphi)(\lambda)(f_2\circ\varphi)(\lambda).
  \end{equation*}
  Both $f_1$ and $f_2$ are completely monotone; hence their product is completely monotone.
  Therefore, $\varphi(\lambda)$ is $\mathcal{BF}$ due to  Lemma \ref{Section 2-Bernstein-new-result}.

  Let $p_\varphi(s,t)$ be the probability density of an inverse subordinator with Laplace exponent $\varphi(\cdot)$.
  Then
  \[
    v(t)=\int_0^\infty p_\varphi(s,t)u(s)ds
  \]
  where $v(t)$ satisfies
  \begin{equation*}
    \begin{aligned}
      \begin{cases}
        v'(t)=Av(t), \quad t>0 \\
        v(0)=u_0.
      \end{cases}
    \end{aligned}
  \end{equation*}

\end{example}

\begin{example}
  Consider
  \begin{equation}\label{general three order equation}
    \begin{aligned}
      \begin{cases}
        \Phi_tu(t)=u'''(t)+a u''(t)+b u'(t)=Au(t),\quad\quad t>0; \\
        u(0) = u_0, \quad u'(0)=u''(0)=0
      \end{cases}
    \end{aligned}
  \end{equation}
  where $a,b\geq0$. Moreover
  \begin{equation*}
    \begin{aligned}
      \begin{cases}
        v'(t)=Av(t),\quad\quad t>0, \\
        v(0) = u_0.
      \end{cases}
    \end{aligned}
  \end{equation*}

  As in  Example \ref{example-two-terms-high-order}, let
  \[
    \lambda=\lambda(\varphi)=\varphi^3+a \varphi^2+b \varphi.
  \]
  This equation has a unique positive solution
  \[
    \varphi=\varphi(\lambda), \quad \lambda>0.
  \]
  The key point is to examine whether $\varphi(\cdot)$ is $\mathcal{BF}$.

  If $a^2= 3b$ (see \eqref{20260602}), then $\lambda=(\varphi+a/3)^3-(a/3)^3$, hence $\varphi(\lambda)=(\lambda+(a/3)^3)^{1/3}-a/3$ is evidently $\mathcal{BF}$; if $b=0$ (see \eqref{general two term equation}), the same conclusion follows from Example \ref{example-two-terms-high-order}. Therefore, we can build
  \[
    v(t)=\int_0^\infty p_\varphi(s,t)u(s)ds
  \]
  where $p_\varphi(s,t)$ is the probability density of an inverse subordinator with Laplace exponent $\varphi(\cdot)$.

  If $a=0,b=1$, we have
  \begin{equation*}
    \varphi'''(\lambda)=\frac{6(15\varphi^2(\lambda)-1)}{(3\varphi^2(\lambda)+1)^5},
  \end{equation*}
  hence $\varphi(\lambda)$ is not $\mathcal{BF}$. Therefore,
  \eqref{general three order equation} does not have the structure of generalized dual diffusion in this case.

\end{example}

\begin{rem}\label{Section-5-rem-open-problem}
 Suppose
  \[
    \lambda = \lambda(\varphi) = \varphi^n + a_1 \varphi^{n-1} + \cdots + a_{n-1}\varphi, \qquad \lambda>0,\ \varphi>0,
  \]
  where $n$ is a positive integer and all the coefficients $a_j \ge 0$. Then the inverse function $\varphi = \varphi(\lambda)$ exists. One may consider the following problem: what conditions must the coefficients $a_j$ ($1\le j\le n-1$) satisfy for $\varphi = \varphi(\lambda)$ to be a Bernstein function?

  Note that the first two Bernstein conditions are automatically satisfied: since $\lambda(\varphi)$ is increasing
  and convex on $(0,\infty)$,
  \[
    \varphi'(\lambda)=\frac{1}{\lambda'(\varphi)}>0,\qquad
    \varphi''(\lambda)=-\frac{\lambda''(\varphi)}{\lambda'(\varphi)^{3}}\leq0,
  \]
  so the problem genuinely starts at the third derivative.

  For $n=3$, this problem was posed in the first version of the present paper, posted on arXiv on June 2, 2026
  (arXiv:2606.04270v1), and was solved one month later by Salazar \cite{Salazar2026arxiv} (arXiv:2607.04467,
  July 2026): the inverse $\varphi$ of $\lambda=\varphi^{3}+a\varphi^{2}+b\varphi$ ($a,b\geq0$) is a Bernstein
  function if and only if $a^{2}\geq 3b$. We remark that the sufficiency part also follows directly from
  Lemma \ref{Section 2-Bernstein-new-result}: if $a^{2}\geq 3b$,
  then the factorization constants $r_{1,2}=\bigl(a\mp\sqrt{a^{2}-3b}\,\bigr)/3$ in $3x^{2}+2ax+b=3(x+r_1)(x+r_2)$ are nonnegative.
  Hence
  \[
    \varphi'(\lambda)=\frac{1}{3\bigl(\varphi(\lambda)+r_{1}\bigr)\bigl(\varphi(\lambda)+r_{2}\bigr)}
  \]
  is a completely monotone function of $\varphi(\lambda)$, and Lemma \ref{Section 2-Bernstein-new-result} applies.

  To the best of our knowledge, the case $n\geq4$ remains open.
\end{rem}

\section{Quantitative duality: asymptotic equivalence of the dual families}

The subordination identities of Section 4 give a structural relation between $S_\varphi$ and $R_\varphi$.
Quantitatively, two distinct phenomena occur. For the power exponent $\varphi(\lambda)=\lambda^\alpha$,
the two families are exactly self-similar with the same spreading exponent but different profiles.
For the relativistic stable exponent $\kappa$, the two families have the same Gaussian long-time profile;
under a suitable coupling of $a$ and $c$, they also have the same classical heat limit as $c\to\infty$.
The case $\beta=1/2$ recovers the computations of Section 3.3.

Throughout this section, let $A=\Delta$ on $\mathbb{R}^n$, $u_0\in\mathcal{S}(\mathbb{R}^n)$, and let
\[
  K_D(x,t)=\frac{1}{(4\pi D t)^{n/2}}\exp\Bigl(-\frac{|x|^2}{4Dt}\Bigr), \qquad x\in\mathbb{R}^n, t>0
\]
be the heat kernel with diffusion coefficient $D>0$. We adopt the semigroup convention
\[
K_D(\cdot,0)\ast u_0:=u_0.
\]
Recall that $S_\varphi(t)$ is always well-defined by
\eqref{Section-5-definition-2}, whereas the time-dual family $R_\varphi(t)$ need not exist in general;
whenever conclusions involving $R_\varphi(t)$ are stated in this section, we assume that $R_\varphi(t)$
exists (e.g., under the conditions of Theorem \ref{existence-R(t)-theorem-1}, or in the concrete cases
covered by Theorem \ref{Section-3-main-thm} and Examples \ref{general-example-1}--\ref{general-example-3}).

\subsection{Long-time asymptotics}

The following theorem collects the long-time asymptotics of the dual pair $S_\varphi(t)$ and $R_\varphi(t)$
in the prototypical Example \ref{general-example-2}. We emphasize that all the results stated in this theorem
are already known; they are collected here in a form that highlights the duality between $S_\varphi$ and
$R_\varphi$, as a preparation for the new quantitative results in the sequel (Theorems
\ref{Section-6-long-time-Gaussian} and \ref{Section-6-nonrelativistic}).

\begin{thm}\label{Section-6-long-time-anomalous}
  Let $\varphi(\lambda)=\lambda^{\alpha}$ with $1/2<\alpha<1$, and let $A=\Delta$ on $\mathbb{R}^n$. Then
  $f(x,t)=S_\varphi(t)u_0(x)$ and $u(x,t)=R_\varphi(t)u_0(x)$ solve, respectively, the space-fractional heat equation
  \begin{equation}\label{Section-6-fractional-heat}
    \partial_t f=-(-\Delta)^{\alpha}f, \qquad f(x,0)=u_0(x),
  \end{equation}
  and the time-fractional diffusion-wave equation of order $1/\alpha\in(1,2)$
  \begin{equation}\label{Section-6-diffusion-wave}
    D_t^{1/\alpha}u=\Delta u, \qquad u(x,0)=u_0(x),\quad \partial_t u(x,0)=0,
  \end{equation}
  and their Fourier symbols are
  \[
    \tilde f(\xi,t)=e^{-t|\xi|^{2\alpha}}\tilde u_0(\xi),
    \qquad
    \tilde u(\xi,t)=E_{1/\alpha}\bigl(-|\xi|^{2}t^{1/\alpha}\bigr)\tilde u_0(\xi),
  \]
  where $E_{1/\alpha}$ denotes the Mittag-Leffler function. Both fundamental solutions are self-similar with the
  same anomalous exponent:
  \begin{equation}\label{Section-6-selfsimilar}
    S_\varphi(t)\delta_0=t^{-\frac{n}{2\alpha}}\Phi_{\mathrm{st}}\bigl(|x|t^{-\frac{1}{2\alpha}}\bigr),
    \qquad
    R_\varphi(t)\delta_0=t^{-\frac{n}{2\alpha}}\Phi_{\mathrm{MW}}\bigl(|x|t^{-\frac{1}{2\alpha}}\bigr),
  \end{equation}
  where $\Phi_{\mathrm{st}}$ is the density of the symmetric $2\alpha$-stable law and $\Phi_{\mathrm{MW}}$ is an
  M-Wright type profile. Explicitly,
  \[
    \Phi_{\mathrm{st}}(x)=\frac{1}{(2\pi)^{n}}\int_{\mathbb{R}^{n}}e^{-|\xi|^{2\alpha}}e^{ix\cdot\xi}\,d\xi,
    \qquad x\in\mathbb{R}^{n},
  \]
  and, for $n=1$,
  \[
    \Phi_{\mathrm{MW}}(z)=\frac12\,M_{\frac{1}{2\alpha}}(z),
    \qquad
    M_{\nu}(z):=\sum_{k=0}^{\infty}\frac{(-z)^{k}}{k!\,\Gamma(-\nu k+1-\nu)}.
  \]
  Their precise estimates are as follows.

  (i) {\em (Symmetric stable side.)} The fundamental solution of \eqref{Section-6-fractional-heat} is the
  density of the symmetric $2\alpha$-stable law, which is strictly positive and satisfies the sharp two-sided
  estimate
  \begin{equation}\label{Section-6-stable-two-sided}
    C^{-1}\Bigl(t^{-\frac{n}{2\alpha}}\wedge \frac{t}{|x|^{n+2\alpha}}\Bigr)
    \le S_\varphi(t)\delta_0(x)
    \le C\Bigl(t^{-\frac{n}{2\alpha}}\wedge \frac{t}{|x|^{n+2\alpha}}\Bigr),
    \qquad t>0,\ x\in\mathbb{R}^n,
  \end{equation}
  for some constant $C=C(n,\alpha)>1$. In particular, it has a heavy algebraic tail
  $S_\varphi(t)\delta_0(x)\asymp t\,|x|^{-(n+2\alpha)}$ as $|x|/t^{1/(2\alpha)}\to\infty$. Consequently
  \begin{equation}\label{Section-6-stable-sup}
    \|S_\varphi(t)\|_{L^1\to L^\infty}=t^{-\frac{n}{2\alpha}}\Phi_{\mathrm{st}}(0),
    \qquad
    \Phi_{\mathrm{st}}(0)=\frac{\Gamma\bigl(\frac{n}{2\alpha}\bigr)}
    {2\alpha\,2^{\,n-1}\pi^{n/2}\,\Gamma\bigl(\frac{n}{2}\bigr)},
  \end{equation}
  and, for $1\le p\le q\le\infty$,
  \begin{equation}\label{Section-6-stable-Lp}
    \|S_\varphi(t)\|_{L^p\to L^q}\asymp t^{-\frac{n}{2\alpha}\bigl(\frac1p-\frac1q\bigr)}.
  \end{equation}

  (ii) {\em (Diffusion-wave side, $n=1$.)} The fundamental solution of \eqref{Section-6-diffusion-wave} in
  $\mathbb{R}$ is the non-negative probability density
  \begin{equation}\label{Section-6-mwright-formula}
    R_\varphi(t)\delta_0
    =\frac{1}{2}\,t^{-\frac{1}{2\alpha}}\,M_{\frac{1}{2\alpha}}\bigl(|x|t^{-\frac{1}{2\alpha}}\bigr),
  \end{equation}
  where $M_\nu$ denotes the M-Wright (Mainardi) function \cite{Mainardi1996}. It has unbounded support and
  has super-Gaussian stretched-exponential decay as $|x|\to\infty$, and
  \begin{equation}\label{Section-6-mwright-sup}
    \|R_\varphi(t)\|_{L^1\to L^\infty}
    =\frac{1}{2}\,t^{-\frac{1}{2\alpha}}\,\bigl\|M_{\frac{1}{2\alpha}}\bigr\|_{L^\infty(0,\infty)}
    =\frac{1}{2}\,t^{-\frac{1}{2\alpha}}\,M_{\frac{1}{2\alpha}}(z^{*}),
    \qquad z^{*}=z^{*}\Bigl(\frac{1}{2\alpha}\Bigr)>0,
  \end{equation}
  where $z^{*}(\nu)$ denotes a maximum point of $M_\nu$ on $[0,\infty)$ (its existence and strict
  positivity for $\nu=\frac{1}{2\alpha}\in(1/2,1)$ are justified in the proof).
  Moreover, for $1\le p\le q\le\infty$,
  \begin{equation}\label{Section-6-mwright-Lp}
    \|R_\varphi(t)\|_{L^p\to L^q}\asymp t^{-\frac{1}{2\alpha}\bigl(\frac1p-\frac1q\bigr)}.
  \end{equation}

\end{thm}

\begin{proof}
  The equation \eqref{Section-6-fractional-heat} and the Fourier symbol of $f$ follow from the Bernstein
  functional calculus of Subsection 2.2, since $-\varphi(-\Delta)=-(-\Delta)^{\alpha}$ has symbol
  $-|\xi|^{2\alpha}$. The self-similar form in \eqref{Section-6-selfsimilar} and the value of
  $\Phi_{\mathrm{st}}(0)$ follow from the change of variables $\eta=t^{1/(2\alpha)}\xi$ in the inversion
  integral:
  \[
    \Phi_{\mathrm{st}}(0)=\frac{1}{(2\pi)^{n}}\int_{\mathbb{R}^{n}}e^{-|\eta|^{2\alpha}}d\eta
    =\frac{\omega_{n-1}}{(2\pi)^{n}}\int_0^\infty e^{-r^{2\alpha}}r^{n-1}dr
    =\frac{\Gamma\bigl(\frac{n}{2\alpha}\bigr)}{2\alpha\,2^{\,n-1}\pi^{n/2}\Gamma\bigl(\frac{n}{2}\bigr)},
  \]
  where $\omega_{n-1}=2\pi^{n/2}/\Gamma(n/2)$ and the substitution $s=r^{2\alpha}$ was used in the last step.
  The two-sided estimate \eqref{Section-6-stable-two-sided} is the sharp (whole-space) form of the classical bounds
  for symmetric stable densities \cite{BlumenthalGetoor1960}, recorded as \cite[eq.~(1.1)]{ChenKimSong2010}; the
  generator $-(-\Delta)^{\alpha/2}$ used there corresponds to our $2\alpha$.

  For $u$, the symbol relation $\lambda(\varphi)=\varphi^{1/\alpha}$ gives
  \[
    \int_0^\infty e^{-\varphi t}\tilde u(\xi,t)\,dt
    =\frac{\varphi^{1/\alpha-1}}{\varphi^{1/\alpha}+|\xi|^{2}}\,\tilde u_0(\xi),
  \]
  and the standard Laplace identity
  \[
    \int_0^\infty e^{-st}E_\beta(-\omega t^{\beta})\,dt=\frac{s^{\beta-1}}{s^{\beta}+\omega},
    \qquad s>0,\ \omega\ge 0,
  \]
  (see \cite[eq. (1.80)]{Podlubny1999book}) with $\beta=1/\alpha$ yields
  $\tilde u(\xi,t)=E_{1/\alpha}(-|\xi|^{2}t^{1/\alpha})\tilde u_0(\xi)$, which is exactly the Fourier solution
  of \eqref{Section-6-diffusion-wave}; the initial conditions $u(0)=u_0$, $u'(0)=0$ match those of
  \eqref{section-5-u} and Example \ref{general-example-2}. The self-similar form of the fundamental solution
  follows from the same change of variables as above. For $n=1$, the fundamental solution of
  \eqref{Section-6-diffusion-wave} is classical and is given by \eqref{Section-6-mwright-formula}, where
  $M_\nu$ denotes the M-Wright function \cite{Mainardi1996}.
  For \eqref{Section-6-mwright-sup}, note that $M_\nu$ is entire and decays super-Gaussian at
  infinity, so its maximum on $[0,\infty)$ is necessarily attained at some point
  $z^{*}=z^{*}(\nu)$; moreover, the defining power series gives $M_\nu(0)=1/\Gamma(1-\nu)$ and
  $M_\nu'(0)=-1/\Gamma(1-2\nu)$, and for $\nu=1/(2\alpha)\in(1/2,1)$ one has $1-2\nu\in(-1,0)$, on
  which $\Gamma$ is negative, so $M_\nu'(0)>0$ and necessarily $z^{*}>0$ with
  $\|M_\nu\|_{L^\infty(0,\infty)}=M_\nu(z^{*})>M_\nu(0)$; here $\frac{1}{2}M_\nu(|x|)$ is a bounded probability
  density with super-Gaussian decay at infinity.

  It remains to justify the $L^{p}\to L^{q}$ estimates \eqref{Section-6-stable-Lp} and
  \eqref{Section-6-mwright-Lp}. We treat the two sides simultaneously. Both kernels in question,
  $k_t(x)=S_\varphi(t)\delta_0(x)$ (with any $n\ge1$) and $k_t(x)=R_\varphi(t)\delta_0(x)$ (with
  $n=1$), have the self-similar form
  \[
    k_t(x)=t^{-n\gamma}k_1\bigl(t^{-\gamma}x\bigr),
    \qquad
    \gamma:=\frac{1}{2\alpha},
  \]
  where the profile $k_1$ satisfies
  \[
    k_1\in L^{1}(\mathbb{R}^{n})\cap L^{\infty}(\mathbb{R}^{n}),
    \qquad
    k_1>0,
    \qquad
    \|k_1\|_{1}=1:
  \]
  for the stable side this is part (i), and for the M-Wright side it follows from the properties
  of $\frac12M_{\nu}$ recorded in part (ii).

  \emph{Step 1: a scaling identity.} Write $(\tau_\lambda f)(x):=f(\lambda x)$ for $\lambda>0$.
  Direct substitution gives the dilation rules
  \[
    \|\tau_\lambda f\|_{p}=\lambda^{-n/p}\|f\|_{p}\qquad(1\le p<\infty),
    \qquad
    \|\tau_\lambda f\|_{\infty}=\|f\|_{\infty}.
  \]
  Moreover, convolution with the dilated kernel is conjugate to convolution with the profile:
  substituting $y=t^{\gamma}z$,
  \[
    (k_t*f)(x)
    =\int_{\mathbb{R}^n}t^{-n\gamma}k_1\bigl(t^{-\gamma}(x-y)\bigr)f(y)\,dy
    =\int_{\mathbb{R}^n}k_1\bigl(t^{-\gamma}x-z\bigr)\,f\bigl(t^{\gamma}z\bigr)\,dz
    =\tau_{t^{-\gamma}}\bigl(k_1*(\tau_{t^{\gamma}}f)\bigr)(x).
  \]
  Taking $L^q$ norms and applying the dilation rules (with $t^{n\gamma/q}=1$ if $q=\infty$),
  \[
    \|k_t*f\|_{q}
    =\bigl\|\tau_{t^{-\gamma}}\bigl(k_1*(\tau_{t^{\gamma}}f)\bigr)\bigr\|_{q}
    =t^{n\gamma/q}\,\|k_1*(\tau_{t^{\gamma}}f)\|_{q},
    \qquad
    \|\tau_{t^{\gamma}}f\|_{p}=t^{-n\gamma/p}\|f\|_{p}.
  \]
  Since $\tau_{t^{\gamma}}$ maps the unit ball of $L^p$ onto the ball of radius $t^{-n\gamma/p}$,
  taking suprema over $\|f\|_{p}\le1$ and using the homogeneity $\|k_1*(cg)\|_{q}=c\,\|k_1*g\|_{q}$
  gives
  \[
    \sup_{\|f\|_{p}\le1}\|k_t*f\|_{q}
    =t^{n\gamma/q}\sup_{\|g\|_{p}\le t^{-n\gamma/p}}\|k_1*g\|_{q}
    =t^{n\gamma/q}\,t^{-n\gamma/p}\sup_{\|h\|_{p}\le1}\|k_1*h\|_{q}.
  \]
  That is, with equality,
  \begin{equation}\label{Section-6-scaling-identity}
    \|k_t*\,\cdot\,\|_{L^p\to L^q}=t^{-n\gamma(\frac1p-\frac1q)}\|k_1*\,\cdot\,\|_{L^p\to L^q},
    \qquad 1\le p\le q\le\infty,\ t>0.
  \end{equation}

  \emph{Step 2: the constant is finite and positive.} Since $k_1\in L^{1}\cap L^{\infty}$,
  log-convexity of the $L^r$ norms gives
  $\|k_1\|_{r}\le\|k_1\|_{1}^{1/r}\|k_1\|_{\infty}^{1-1/r}<\infty$ for every $1\le r\le\infty$.
  Young's convolution inequality \cite[Theorem 4.2]{Lieb2001book},
  $\|k_1*f\|_{q}\le\|k_1\|_{r}\|f\|_{p}$ with $\frac1r=1-\frac1p+\frac1q$
  (note $r\in[1,\infty]$ since $p\le q$), therefore yields
  $\|k_1*\,\cdot\,\|_{L^p\to L^q}\le\|k_1\|_{r}<\infty$.
  The constant is also positive: for any nonzero $f\ge0$, since $k_1>0$ everywhere,
  \[
    (k_1*f)(x)=\int_{\mathbb{R}^n}k_1(x-y)f(y)\,dy>0
    \qquad\text{for every }x,
  \]
  hence $\|k_1*f\|_{q}>0$ and
  $\|k_1*\,\cdot\,\|_{L^p\to L^q}\ge\|k_1*f\|_{q}/\|f\|_{p}>0$. Thus the constant in
  \eqref{Section-6-scaling-identity} is a finite positive number depending only on
  $n,\alpha,p,q$, which proves \eqref{Section-6-stable-Lp} and, taking $n=1$,
  \eqref{Section-6-mwright-Lp}.

  \emph{Step 3: the endpoint $p=1$, $q=\infty$.} We claim that
  $\|k_1*\,\cdot\,\|_{L^1\to L^\infty}=\|k_1\|_{\infty}$. The upper bound
  $\|k_1*f\|_{\infty}\le\|k_1\|_{\infty}\|f\|_{1}$ is immediate. For the lower bound, note that
  $k_1$ is continuous (the stable density, respectively the M-Wright profile) and attains its
  maximum at some point $x_{*}$: for the stable density this follows from continuity and decay
  at infinity, and for the M-Wright profile the maximum point $x_{*}=z^{*}$ was exhibited above.
  Testing against an approximate identity $f_j\ge0$, $\|f_j\|_{1}=1$, concentrating at the
  origin, continuity gives
  \[
    (k_1*f_j)(x_{*})
    =\int_{\mathbb{R}^n}k_1(x_{*}-y)f_j(y)\,dy
    \longrightarrow k_1(x_{*})=\|k_1\|_{\infty},
  \]
  whence $\|k_1*\,\cdot\,\|_{L^1\to L^\infty}\ge\|k_1\|_{\infty}$. Therefore
  \eqref{Section-6-scaling-identity} with $p=1$, $q=\infty$ reads
  $\|S_\varphi(t)\|_{L^1\to L^\infty}=t^{-n\gamma}\|\Phi_{\mathrm{st}}\|_{\infty}$ and
  $\|R_\varphi(t)\|_{L^1\to L^\infty}=t^{-\gamma}\bigl\|\frac12M_{\nu}\bigr\|_{\infty}$.

  \emph{Step 4: the two maxima.} For the stable profile, the Fourier representation
  \[
    \Phi_{\mathrm{st}}(x)=\frac{1}{(2\pi)^{n}}\int_{\mathbb{R}^n}e^{-|\xi|^{2\alpha}}e^{ix\cdot\xi}\,d\xi
  \]
  has a nonnegative symbol, so
  $|\Phi_{\mathrm{st}}(x)|\le(2\pi)^{-n}\int_{\mathbb{R}^n}e^{-|\xi|^{2\alpha}}\,d\xi
  =\Phi_{\mathrm{st}}(0)$ for every $x$: the maximum is attained at the origin and
  $\|\Phi_{\mathrm{st}}\|_{\infty}=\Phi_{\mathrm{st}}(0)$, whose value was computed in
  \eqref{Section-6-stable-sup}. For the M-Wright profile,
  $\bigl\|\frac12M_{\nu}\bigr\|_{L^\infty(0,\infty)}=\frac12M_{\nu}(z^{*})$ by the very definition
  of $z^{*}$ in \eqref{Section-6-mwright-sup}. This completes the proof of
  \eqref{Section-6-stable-sup} and \eqref{Section-6-mwright-sup}.
\end{proof}

\begin{rem}
  (1) Theorem \ref{Section-6-long-time-anomalous} shows that the long-time behavior of the dual pair is governed
  by the same anomalous exponent: both profiles in \eqref{Section-6-selfsimilar} spread like $t^{1/(2\alpha)}$,
  although their universal shapes differ. As $\alpha\to1^{-}$, both profiles degenerate to the same Gaussian heat
  kernel, in accordance with Theorem \ref{Section-6-long-time-Gaussian}. We note that part (i), as well as the
  space-dual half of \eqref{Section-6-selfsimilar}, holds verbatim for every $\alpha\in(0,1)$, since only the
  symmetric stable subordination is involved; the restriction $1/2<\alpha<1$ is imposed solely by the time-dual
  side (ii), where the order $1/\alpha$ is required to lie in $(1,2)$.

  (2) Both sides have infinite propagation speed: neither profile has compact support. The contrast between them
  is therefore one of decay and sign, not of speed: the space-dual side \eqref{Section-6-fractional-heat} is
  strictly positive with a heavy algebraic tail $t\,|x|^{-(n+2\alpha)}$ in every dimension, while the time-dual
  side \eqref{Section-6-diffusion-wave} is non-negative with super-Gaussian decay for $n=1$ and sign-changing
  with a singularity for $n\ge2$ (see (3) below). In the limiting case $\alpha=1/2$ the pair reduces to the
  Cauchy--Poisson semigroup $e^{-t(-\Delta)^{1/2}}$ (space side) versus the classical wave equation
  $D_t^{2}u=\Delta u$ (time side), cf. Example \ref{general-example-1}.

  (3) For completeness we briefly discuss the case $n\ge2$ of the diffusion-wave side, which seems to have
  attracted limited attention in the literature; we only summarize the known facts and leave a detailed
  study to future work. The fundamental solution of \eqref{Section-6-diffusion-wave} is then no longer
  non-negative: for $n\ge2$ the symbol $E_{1/\alpha}(-|\xi|^{2})$ is not integrable and is eventually
  negative, since $|E_{1/\alpha}(-s)|\sim s^{-1}/\Gamma(1-1/\alpha)$ with $\Gamma(1-1/\alpha)<0$
  \cite[Theorem 1.4]{Podlubny1999book}; correspondingly, the self-similar profile is sign-changing and
  singular at the origin (logarithmic for $n=2$, of order $|x|^{2-n}$ for $n\ge3$), with a negative
  singular part balanced by the total mass $1$; in particular $\|R_\varphi(t)\|_{L^1\to L^\infty}=\infty$.
\end{rem}

We next study the long-time asymptotics of the relativistic stable dual pair $S_\kappa(t)$ and
$R_\kappa(t)$ in Example \ref{general-example-3}. Throughout,
$\gamma:=\kappa'(0)=\beta a^{1-1/\beta}$ and $\delta:=-\kappa''(0)/2=\frac{\beta(1-\beta)}{2}a^{1-2/\beta}$.
The Gaussian limit has a well-known probabilistic counterpart, which is also one motivation for
this study: Rosi\'nski's functional central limit theorem
\cite[Theorem 3.1 (ii)]{Rosinski2007SPA} shows that the tempered stable L\'evy process converges, in
the long-time scaling limit, to Brownian motion with diffusion coefficient $\gamma=\kappa'(0)$.
Theorem \ref{Section-6-long-time-Gaussian} below is the density-level counterpart of this
convergence, with an explicit rate and for the dual pair.
The time-dual side requires a contour analysis of the resolvent symbol
$V(\xi,s)=\frac{\lambda(s)/s}{\lambda(s)+|\xi|^{2}}$,
$\lambda(s)=(s+a)^{1/\beta}-a^{1/\beta}$, and the analysis is carried out once and for all in
dimensionless form: with $h:=|\xi|^{2}$, the scaling $z=s/a$, $y=h/a^{1/\beta}$ gives
\begin{equation}\label{Section-6-scaling-multiplier}
  V(\xi,s)=\frac1a\,M_\beta\Bigl(\frac{h}{a^{1/\beta}},\frac{s}{a}\Bigr),
  \qquad
  \frac{1}{2\pi i}\int_{\sigma-i\infty}^{\sigma+i\infty}e^{st}V(\xi,s)\,ds
  =m_\beta\Bigl(\frac{h}{a^{1/\beta}},at\Bigr),
\end{equation}
where $M_\beta$ and $m_\beta$ are as in the following lemma. This identity is the key
simplification: all contour constants are proved once, as functions of $\beta$ and a fixed
dimensionless cutoff only; the same bounds will then apply verbatim in the proofs of Theorems
\ref{Section-6-long-time-Gaussian} and \ref{Section-6-nonrelativistic} below.

\begin{lem}[Dimensionless Multiplier Lemma]\label{Section-6-DML}
  Let $\beta\in(1/2,1)$. Define
  $L_\beta(z):=(1+z)^{1/\beta}-1$ (principal branch) on $\mathbb{C}\setminus(-\infty,-1]$. For $y\ge0$, set
  \[
    M_\beta(y,z):=\frac{L_\beta(z)}{z\,\bigl(L_\beta(z)+y\bigr)}
    \qquad
    W_\beta(y,z):=M_\beta(y,z)-\frac1z
    =-\frac{y}{z\,\bigl(L_\beta(z)+y\bigr)}.
  \]
  For $y>0$ the factor $z$ is cancelled and $M_\beta$ is analytic at $z=0$, since
  $L_\beta(z)=z/\beta+O(z^{2})$; for $y=0$ the quotient reduces to $M_\beta(0,z)=1/z$ (as a
  meromorphic identity), so that $W_\beta(0,\cdot)\equiv0$. Denote by
  \[
    m_\beta(y,\tau):=\frac{1}{2\pi i}\int_{\sigma-i\infty}^{\sigma+i\infty}e^{z\tau}M_\beta(y,z)\,dz
    =1+\frac{1}{2\pi i}\int_{\sigma-i\infty}^{\sigma+i\infty}e^{z\tau}W_\beta(y,z)\,dz
    \qquad(\sigma>0)
  \]
  its inverse Laplace transform in $z$ (the second integral converges absolutely by
  \eqref{Section-6-dml-modulus} and \eqref{Section-6-dml-sector} below; for $y=0$ it gives
  $m_\beta(0,\tau)\equiv1$). Fix a cutoff
  $y_0\in\bigl(0,1-2^{-1/\beta}\bigr)$ and write
  \[
    r_\beta(y):=(1-y)^{\beta}-1,
    \qquad
    A_\beta(y):=-\frac{\beta y(1-y)^{\beta-1}}{(1-y)^{\beta}-1},\quad A_\beta(0):=1.
  \]
  Then, with constants $C,c>0$ depending only on $\beta$ and $y_0$:

  (i) \emph{(diffusive region)} for $0\le y\le y_0$ and $\tau>0$,
  \begin{equation}\label{Section-6-dml-diffusive}
    m_\beta(y,\tau)=A_\beta(y)\,e^{r_\beta(y)\tau}+\rho_\beta(y,\tau),
    \qquad
    |\rho_\beta(y,\tau)|\le C\,y\,e^{-\tau/2};
  \end{equation}

  (ii) \emph{(exterior region)} for $y\ge y_0$ and $\tau>0$,
  \begin{equation}\label{Section-6-dml-exterior}
    \bigl|m_\beta(y,\tau)\bigr|\le C\,e^{-c\tau}\bigl(1+\log(1+y)\bigr),
    \qquad c:=-\frac{r_\beta(y_0)}{2}>0;
  \end{equation}

  (iii) \emph{(crude bound)} for all $y\ge0$ and $\tau>0$,
  \begin{equation}\label{Section-6-dml-crude}
    \bigl|m_\beta(y,\tau)\bigr|\le 1+C\,\tau y.
  \end{equation}

  Moreover, as $y\to0+$,
  \begin{equation}\label{Section-6-dml-expansions}
    r_\beta(y)=-\beta y-\frac{\beta(1-\beta)}{2}\,y^{2}+O(y^{3}),
    \qquad
    A_\beta(y)=1+\frac{1-\beta}{2}\,y+O(y^{2}).
  \end{equation}
\end{lem}

\begin{proof}
  \emph{Singularity inventory.} The function $L_\beta$ is holomorphic on
  $\mathbb{C}\setminus(-\infty,-1]$ with the single branch point $z=-1$. The poles of
  $M_\beta(y,\cdot)$ are the solutions of $(1+z)^{1/\beta}=1-y$; writing
  $(1+z)^{1/\beta}=|1+z|^{1/\beta}e^{i\theta/\beta}$ with $\theta:=\arg(1+z)\in(-\pi,\pi)$, for
  $0<y<1$ the right-hand side is positive; $\theta/\beta=0$ yields the real pole
  \[
    z_*(y):=r_\beta(y)\in(-1,0),
  \]
  and $\theta/\beta=\pm2\pi$ would require $|\theta|=2\pi\beta>\pi$, so there is no further
  solution; for $y\ge1$ the right-hand side is non-positive; $\theta/\beta=\pm\pi$ yields the two
  conjugate poles $z=-1+(y-1)^{\beta}e^{\pm i\pi\beta}$ (coalescing at the branch point $z=-1$
  when $y=1$), with $\operatorname{Re}z=-1+(y-1)^{\beta}\cos(\pi\beta)\le-1$, since
  $\cos(\pi\beta)<0$ for $\beta\in(1/2,1)$. At the real pole,
  $L_\beta'(z)=\frac1\beta(1+z)^{1/\beta-1}$ gives
  \begin{equation}\label{Section-6-dml-residue}
    \operatorname*{Res}_{z=z_*(y)}M_\beta(y,\cdot)
    =\frac{L_\beta(z_*(y))}{z_*(y)\,L_\beta'(z_*(y))}
    =\frac{-y}{r_\beta(y)\,\frac1\beta(1-y)^{1-\beta}}
    =A_\beta(y).
  \end{equation}

  \emph{Symbol bounds.} We record three estimates. First, for $\operatorname{Re}z>0$ the
  representation $L_\beta(z)=\frac{z}{\beta}\int_0^1(1+sz)^{1/\beta-1}\,ds$ has integrand factors
  in the convex cone $\bigl\{w:|\arg w|<\frac{(1-\beta)\pi}{2\beta}\bigr\}$ of opening $<\pi$, so
  the integral inherits the same argument bound and has modulus at least
  $\cos\frac{(1-\beta)\pi}{2\beta}$; hence, with $c_1:=\frac1\beta\cos\frac{(1-\beta)\pi}{2\beta}>0$,
  \begin{equation}\label{Section-6-dml-modulus}
    |L_\beta(z)|\ge c_1|z|,
    \qquad
    \bigl|\arg L_\beta(z)\bigr|<\frac{\pi}{2}+\frac{(1-\beta)\pi}{2\beta}=\frac{\pi}{2\beta}
    \qquad (\operatorname{Re}z>0).
  \end{equation}
  Since $\frac{\pi}{2\beta}<\pi$, the number $L_\beta(z)$ keeps the fixed positive angle
  $\pi-\frac{\pi}{2\beta}$ away from the negative real half-axis, while $y>0$ is real; therefore
  \begin{equation}\label{Section-6-dml-sector}
    \bigl|L_\beta(z)+y\bigr|\ge c_\beta\bigl(|L_\beta(z)|+y\bigr)
    \qquad (\operatorname{Re}z>0,\ y>0).
  \end{equation}
  Second, $|L_\beta(z)|\ge|1+z|^{1/\beta}-1\ge c|z|^{1/\beta}$ for $|z|$ large, say
  \begin{equation}\label{Section-6-dml-large}
    |L_\beta(z)|\ge c_2|z|^{1/\beta}
    \qquad (|z|\ge R_1),
  \end{equation}
  and conversely $|L_\beta(z)|\le|1+z|^{1/\beta}+1\le C_4|z|^{1/\beta}$ for $|z|\ge1$; we may
  (and do) take $R_1\ge1$.
  Third, for $\operatorname{Re}z>-1$ one has
  $\arg(1+z)/\beta\in\bigl(-\frac{\pi}{2\beta},\frac{\pi}{2\beta}\bigr)$, and the latter interval
  contains no zero of $\sin$ except $0$ (attained only when $\operatorname{Im}z=0$), since
  $\frac{\pi}{2\beta}<\pi$; hence
  \begin{equation}\label{Section-6-dml-imag}
    \operatorname{Im}L_\beta(z)=|1+z|^{1/\beta}\sin\bigl(\arg(1+z)/\beta\bigr)\neq0
    \qquad(\operatorname{Re}z>-1,\ \operatorname{Im}z\neq0),
  \end{equation}
  and $|\operatorname{Im}L_\beta(z)|\ge\sin\bigl(\frac{\pi}{2\beta}\bigr)|1+z|^{1/\beta}$
  whenever $|\arg(1+z)|\ge\frac{\beta\pi}{2}$, with $\sin\bigl(\frac{\pi}{2\beta}\bigr)>0$.

  \emph{Contour deformation} (see Figure \ref{Section-6-fig-contours}). All shifts below use a
  rectangle with vertical sides
  $\operatorname{Re}z=\sigma$ and $\operatorname{Re}z=\sigma_0\in(-1,0)$ and horizontal
  connectors $\operatorname{Im}z=\pm R$: on the connectors $|e^{z\tau}|\le e^{\sigma\tau}$ and,
  for $R$ so large that $|L_\beta(z)|\ge2y$ (using \eqref{Section-6-dml-large}),
  $|W_\beta(y,z)|\le\frac{2y}{|z|\,|L_\beta(z)|}\le\frac{C}{R^{1+1/\beta}}\to0$
  uniformly in $\sigma_0\le\operatorname{Re}z\le\sigma$; the connector integrals vanish as
  $R\to\infty$, and the shifts are legitimate.

\begin{figure}
\centering
\begin{tikzpicture}[scale=2.5,>=stealth,every node/.style={font=\footnotesize}]

  \draw[->,gray!80!black] (-2.7,0) -- (1.4,0);
  \node[anchor=north west] at (1.4,-0.02) {$\operatorname{Re}z$};
  \draw[->,gray!80!black] (0,-1.85) -- (0,2.3);
  \node[anchor=south east] at (-0.03,2.3) {$\operatorname{Im}z$};

  \draw[red!80!black,very thick,decorate,decoration={zigzag,segment length=2.2mm,amplitude=0.8mm}]
    (-2.62,0) -- (-1.7,0);
  \fill[red!80!black] (-1.7,0) circle (0.03);
  \node[red!80!black,anchor=south] at (-2.15,0.05) {branch cut $(-\infty,-1]$};
  \node[red!80!black,anchor=north] at (-1.7,-0.05) {branch point $z=-1$};

  \fill[purple!75!black] (-1.8,0.35) circle (0.026);
  \fill[purple!75!black] (-1.8,-0.35) circle (0.026);
  \node[purple!75!black,anchor=north,align=center,font=\scriptsize] at (-1.8,-0.45)
    {complex poles ($y\ge1$):\\ $\operatorname{Re}z\le-1$};

  \fill[purple!75!black] (-0.7,0) circle (0.03);
  \node[purple!75!black,anchor=north,font=\scriptsize] at (-0.62,-0.05) {$z_*(y_0)=-2c$};
  \fill[purple!75!black] (-0.15,0) circle (0.03);
  \fill[purple!75!black] (0,0) circle (0.036);
  \node[purple!75!black,anchor=north west,font=\scriptsize] at (0.04,-0.06)
    {$z=0$,\ \ $\operatorname*{Res}=-1$};
  \node[purple!75!black,anchor=south,align=center,font=\scriptsize] at (0.35,0.32)
    {$z_*(y)=r_\beta(y)$,\\ $\operatorname*{Res}=A_\beta(y)$};
  \draw[gray!70,thin] (0.3,0.27) -- (-0.12,0.03);

  \draw[blue!70!black,thick,->] (0.95,-1.55) -- (0.95,1.9);
  \node[blue!70!black,anchor=south] at (0.95,1.9) {$\sigma$};
  \draw[teal!75!black,thick,->] (-1.0,-1.55) -- (-1.0,1.9);
  \node[teal!75!black,anchor=south] at (-1.0,1.9) {$-\frac12$};
  \draw[orange!90!black,thick,->] (-0.35,-1.55) -- (-0.35,1.9);
  \node[orange!90!black,anchor=south] at (-0.35,1.9) {$-c$};

  \draw[teal!75!black,very thick,->] (0.95,1.25) -- (-0.96,1.25);
  \node[teal!75!black,anchor=south,align=center,font=\scriptsize] at (-0.1,1.28)
    {shift left (pole formula):\\ cross $z=0$ and $z_*(y)$};
  \draw[orange!90!black,very thick,->] (0.95,-1.2) -- (-0.32,-1.2);
  \node[orange!90!black,anchor=north,align=center,font=\scriptsize] at (0.2,-1.23)
    {shift left (exterior):\\ cross only $z=0$};

  \draw[gray,dashed] (-1.0,1.7) -- (0.95,1.7);
  \draw[gray,dashed] (-1.0,-1.7) -- (0.95,-1.7);
  \node[gray!80!black,anchor=north,align=center,font=\scriptsize] at (-0.55,-1.73)
    {connectors $\operatorname{Im}z=\pm R\to\infty$\\ (vanish)};

\end{tikzpicture}
\caption{\textbf{Contour deformations in the proof of Lemma \ref{Section-6-DML} (schematic).}
The original Bromwich contour $\operatorname{Re}z=\sigma$ ($\sigma>0$, blue) is shifted to the
left in two regimes. In the diffusive region $0<y\le y_0$, the new contour
$\operatorname{Re}z=-\frac12$ (teal) crosses the pole of $W_\beta$ at $z=0$
($\operatorname*{Res}=-1$) and the real pole $z_*(y)=r_\beta(y)$
($\operatorname*{Res}=A_\beta(y)$, cf. \eqref{Section-6-dml-residue}), which ranges in
$\bigl(z_*(y_0),0\bigr)$ with $z_*(y_0)=-2c>-\frac12$. In the exterior region $y\ge y_0$, the
new contour $\operatorname{Re}z=-c$ (orange), $c:=-r_\beta(y_0)/2$, crosses only $z=0$: the
real pole then satisfies $z_*(y)\le z_*(y_0)=-2c$, while the complex poles for $y\ge1$ and the
branch cut $(-\infty,-1]$ (red zigzag, branch point $z=-1$) all satisfy
$\operatorname{Re}z\le-1<-c$. The dashed horizontal connectors $\operatorname{Im}z=\pm R$
vanish as $R\to\infty$.}
\label{Section-6-fig-contours}
\end{figure}

  \emph{The pole formula \eqref{Section-6-dml-diffusive}.} For $y=0$ the formula holds
  trivially, since $M_\beta(0,z)=1/z$ and $W_\beta\equiv0$, giving $m_\beta(0,\tau)=1
  =A_\beta(0)\,e^{r_\beta(0)\tau}$ and $\rho_\beta(0,\tau)=0$. Assume now $0<y\le y_0$. Shifting
  the contour of $m_\beta(y,\tau)-1$ to $\operatorname{Re}z=-\frac12$ crosses the pole of
  $W_\beta$ at $z=0$ (residue $-1$, cancelling the term $1$) and the real pole $z_*(y)$
  with residue \eqref{Section-6-dml-residue}; indeed $z_*(y)\ge z_*(y_0)>-\frac12$, since
  $r_\beta$ is decreasing and $y_0<1-2^{-1/\beta}$ means
  $r_\beta(y_0)=(1-y_0)^{\beta}-1>-\frac12$. No other singularity is crossed: the branch cut
  $(-\infty,-1]$ lies strictly to the left, and for $y<1$ there are no complex poles. Hence
  \eqref{Section-6-dml-diffusive} holds with
  $\rho_\beta(y,\tau)=\frac{1}{2\pi i}\int_{-1/2-i\infty}^{-1/2+i\infty}e^{z\tau}
  W_\beta(y,z)\,dz$. For the bound we estimate $|L_\beta(z)+y|$ on the new contour,
  uniformly in $0\le y\le y_0$. At the real point,
  $L_\beta\bigl(-\frac12\bigr)+y=y-\bigl(1-2^{-1/\beta}\bigr)
  \le y_0-\bigl(1-2^{-1/\beta}\bigr)=:-\eta_0<0$;
  for $\operatorname{Im}z\neq0$, $\operatorname{Im}L_\beta(z)\neq0$ by
  \eqref{Section-6-dml-imag}. By continuity in $(\rho,y)$ on the compact set
  $\bigl\{\bigl(-\tfrac12+i\rho,y\bigr):\,|\rho|\le R,\ 0\le y\le y_0\bigr\}$,
  where $R:=\max\bigl(R_1,\frac12\tan\frac{\beta\pi}{2}\bigr)$, there is a uniform lower bound
  $|L_\beta(z)+y|\ge m_0>0$; on
  this near part $|W_\beta(y,z)|\le\frac{y}{(1/2)m_0}$. On the far part
  $|\operatorname{Im}z|\ge R$ one has $|\arg(1+z)|\ge\frac{\beta\pi}{2}$, hence
  $|L_\beta(z)+y|\ge|\operatorname{Im}L_\beta(z)|
  \ge\sin\bigl(\frac{\pi}{2\beta}\bigr)|1+z|^{1/\beta}\ge c|z|^{1/\beta}$ and
  $|W_\beta(y,z)|\le\frac{Cy}{|z|^{1+1/\beta}}$, whose integral over
  $|\operatorname{Im}z|\ge R$ is $O(y)$. Altogether
  \[
    |\rho_\beta(y,\tau)|
    \le\frac{e^{-\tau/2}}{2\pi}\int_{\mathbb{R}}
    \bigl|W_\beta\bigl(y,-\tfrac12+i\rho\bigr)\bigr|\,d\rho
    \le C\,y\,e^{-\tau/2},
  \]
  with $C$ depending only on $\beta$ and $y_0$.

  \emph{The exterior bound \eqref{Section-6-dml-exterior}.} Let $y\ge y_0$ and
  $c:=-r_\beta(y_0)/2>0$. All singularities of $M_\beta(y,\cdot)$ lie strictly to the left of
  $\operatorname{Re}z=-c$: the real pole satisfies $z_*(y)\le z_*(y_0)=-2c<-c$ (for $y<1$), and
  the complex poles (for $y\ge1$) and the branch point lie at $\operatorname{Re}z\le-1<-c$.
  Shifting the contour of $m_\beta(y,\tau)-1$ to $\operatorname{Re}z=-c$ therefore crosses only
  the pole of $W_\beta$ at $z=0$ (residue $-1$), so that
  \[
    m_\beta(y,\tau)=\frac{1}{2\pi i}\int_{-c-i\infty}^{-c+i\infty}e^{z\tau}W_\beta(y,z)\,dz.
  \]
  We estimate $I(y):=\int_{\mathbb{R}}\bigl|W_\beta(y,-c+i\rho)\bigr|\,d\rho
  =\int_{\mathbb{R}}\frac{y\,d\rho}{|z|\,|L_\beta(z)+y|}$, with
  $R_0:=\max\bigl(R_1,(1-c)\tan\frac{\beta\pi}{2}\bigr)$ and
  $\Lambda_0:=\max_{|\rho|\le R_0}|L_\beta(-c+i\rho)|$. On $\{|\rho|\le R_0\}$: if
  $y\le2\Lambda_0$, then $|L_\beta(z)+y|\ge m_0'>0$ by continuity in $(\rho,y)$ on the compact
  set $\{(-c+i\rho,y):|\rho|\le R_0,\ y_0\le y\le2\Lambda_0\}$ (nonvanishing: at $\rho=0$,
  $L_\beta(-c)+y>y-y_0\ge0$, since
  $-L_\beta(-c)=1-\bigl(\frac{1+(1-y_0)^{\beta}}{2}\bigr)^{1/\beta}<y_0$; and
  \eqref{Section-6-dml-imag} applies), hence
  $|L_\beta(z)+y|\ge\frac{m_0'}{1+2\Lambda_0}(1+y)$; if $y\ge2\Lambda_0$, then
  $|L_\beta(z)+y|\ge y-\Lambda_0\ge\frac{y}{2}$. In both cases the integrand is
  $O(1)$, and this part contributes $O(1)$. On $\{|\rho|>R_0\}$ two cases are
  distinguished. If $|L_\beta(z)|\le y/2$, then $|L_\beta(z)+y|\ge y/2$ and the integrand is at
  most $\frac{2}{|z|}$; moreover \eqref{Section-6-dml-large} gives
  $|z|\le\bigl(\frac{y}{2c_2}\bigr)^{\beta}$ on this part (which is thus empty unless
  $y>2c_2R_0^{1/\beta}$), so it contributes at most
  $\int_{R_0}^{(y/2c_2)^{\beta}}\frac{2\,d\rho}{|\rho|}
  =2\log\frac{(y/2c_2)^{\beta}}{R_0}=O\bigl(\log(1+y)\bigr)$. If $|L_\beta(z)|>y/2$, then
  $|z|^{1/\beta}\ge\frac{y}{2C_4}$ (using $|L_\beta(z)|\le C_4|z|^{1/\beta}$ for $|z|\ge1$) and
  $|L_\beta(z)+y|\ge|\operatorname{Im}L_\beta(z)|\ge c''|z|^{1/\beta}$ by
  \eqref{Section-6-dml-imag}; the integrand is at most $\frac{Cy}{|z|^{1+1/\beta}}$, and, since
  $|z|\ge\bigl(\frac{y}{2C_4}\bigr)^{\beta}$ on this part and the whole-line integral of
  $|z|^{-1-1/\beta}$ is finite with tail $O(T^{-1/\beta})$ beyond radius $T$, its integral is
  $O\Bigl(y\min\Bigl\{1,\Bigl(\frac{y}{2C_4}\Bigr)^{-1}\Bigr\}\Bigr)=O(1)$. Altogether
  $I(y)\le C\bigl(1+\log(1+y)\bigr)$, and
  \[
    \bigl|m_\beta(y,\tau)\bigr|\le\frac{e^{-c\tau}}{2\pi}\,I(y)
    \le C\,e^{-c\tau}\bigl(1+\log(1+y)\bigr).
  \]

  \emph{The crude bound \eqref{Section-6-dml-crude}.} The case $y=0$ is trivial, since
  $m_\beta(0,\tau)\equiv1$. For $y>0$, invert along
  $\operatorname{Re}z=\frac1\tau$ without any shift: by \eqref{Section-6-dml-modulus} and
  \eqref{Section-6-dml-sector}, with $c_3:=c_\beta\min\{c_1,1\}$,
  \[
    \bigl|W_\beta(y,z)\bigr|
    \le\frac{y}{c_3\,|z|\,\bigl(|z|+y\bigr)}
    \qquad\Bigl(\operatorname{Re}z=\tfrac1\tau\Bigr),
  \]
  and, substituting $\rho=u/\tau$ and writing $k:=y\tau$ (so that
  $|z|=\frac1\tau\sqrt{1+u^{2}}$),
  \[
    \int_{\mathbb{R}}\frac{d\rho}{|z|\,\bigl(|z|+y\bigr)}
    =\tau\int_{\mathbb{R}}\frac{du}{\sqrt{1+u^{2}}\,\bigl(\sqrt{1+u^{2}}+k\bigr)}
    \le\tau\int_{\mathbb{R}}\frac{du}{1+u^{2}}
    =\pi\tau,
  \]
  whence
  \[
    \bigl|m_\beta(y,\tau)-1\bigr|
    \le\frac{e}{2\pi}\,y\cdot\frac{\pi\tau}{c_3}
    =\frac{e}{2c_3}\,\tau y.
  \]

  Finally, the expansions \eqref{Section-6-dml-expansions} follow from Taylor's theorem:
  $r_\beta(y)=-\beta y+\frac{\beta(\beta-1)}{2}y^{2}+O(y^{3})$, and, writing
  $A_\beta(y)=\frac{-\beta y}{r_\beta(y)}(1-y)^{\beta-1}$ with
  $\frac{-\beta y}{r_\beta(y)}=\bigl(1+\frac{1-\beta}{2}y+O(y^{2})\bigr)^{-1}$ and
  $(1-y)^{\beta-1}=1+(1-\beta)y+O(y^{2})$, one obtains
  $A_\beta(y)=1+\frac{1-\beta}{2}y+O(y^{2})$.
\end{proof}

\begin{thm}\label{Section-6-long-time-Gaussian}
  Let $\kappa(\lambda)=-a+(a^{1/\beta}+\lambda)^{\beta}$ with $a>0$ and
  $\beta\in[1/2,1)$ (Example \ref{general-example-3}), and set
  \[
    \gamma:=\kappa'(0)=\beta a^{1-1/\beta},
    \qquad
    \delta:=-\frac{\kappa''(0)}{2}=\frac{\beta(1-\beta)}{2}\,a^{1-2/\beta}.
  \]
  Then, for every $u_0\in\mathcal{S}(\mathbb{R}^n)$, as $t\to\infty$,
  \begin{equation}\label{Section-6-gaussian-S}
    \bigl\|S_\kappa(t)u_0-K_\gamma(\cdot,t)*u_0\bigr\|_{L^\infty(\mathbb{R}^n)}=O\bigl(t^{-(n+2)/2}\bigr),
  \end{equation}
  \begin{equation}\label{Section-6-gaussian-R}
    \bigl\|R_\kappa(t)u_0-K_\gamma(\cdot,t)*u_0\bigr\|_{L^\infty(\mathbb{R}^n)}=O\bigl(t^{-(n+2)/2}\bigr),
  \end{equation}
  and consequently
  \begin{equation}\label{Section-6-gaussian-difference}
    \bigl\|S_\kappa(t)u_0-R_\kappa(t)u_0\bigr\|_{L^\infty(\mathbb{R}^n)}=O\bigl(t^{-(n+2)/2}\bigr).
  \end{equation}
  Hence the two dual families share the same Gaussian long-time profile with diffusion coefficient
  $\gamma=\kappa'(0)$, and their mutual distance decays one order faster than the profile itself.
\end{thm}

\begin{proof}
  For $\beta=1/2$ the time-dual conclusion \eqref{Section-6-gaussian-R}, with
  $\gamma=\frac{1}{2a}$ and $\delta=\frac{1}{8a^{3}}$, is exactly the long-time asymptotics of the
  telegraph equation in Section 3.3 (Case 2); in fact the telegraph
  multiplier $m_c$ of Appendix B is the boundary instance
  $m_{1/2}\bigl(c^{2}q/a^{2},at\bigr)$ of the scaling identity
  \eqref{Section-6-scaling-multiplier}. Step 1 below
  applies verbatim to $\beta=1/2$; we therefore assume $\beta\in(1/2,1)$ only in Step 2.

  Since the only singularity of
  $\kappa$ is the branch point $-a^{1/\beta}$, the function $\kappa$ is analytic at the origin
  and
  \begin{equation}\label{Section-6-kappa-expansion}
    \kappa(h)=\gamma h-\delta h^{2}+O(h^{3}), \qquad h\to0+.
  \end{equation}

  \textbf{Step 1: the space-dual side.} For $f(x,t)=S_\kappa(t)u_0(x)$, the Fourier transform in $x$ is
  $\tilde f(\xi,t)=e^{-t\kappa(|\xi|^{2})}\tilde u_0(\xi)$. For the $L^\infty$ estimate we use
  $\|g\|_{L^\infty(\mathbb{R}^n)}\le(2\pi)^{-n}\|\tilde g\|_{L^1(\mathbb{R}^n)}$ and the elementary
  identity
  \begin{equation}\label{Section-6-gaussian-moments}
    \int_{\mathbb{R}^n}e^{-\gamma t|\xi|^{2}}|\xi|^{2k}\,d\xi=C(n,\gamma,k)\,t^{-(n+2k)/2}, \qquad k=0,1,2,\dots.
  \end{equation}
  By \eqref{Section-6-kappa-expansion} there are constants $C_0,h_0>0$ such that
  $|\kappa(h)-\gamma h|\le(\delta+C_0h)h^{2}$ for $0\le h\le h_0$. Choose
  $C:=\delta+C_0h_0$ and $\varepsilon:=\min\bigl(h_0,\frac{\gamma}{2C}\bigr)$; then, for
  $0\le h\le\varepsilon$,
  \[
    |\kappa(h)-\gamma h|\le Ch^{2}, \qquad
    \kappa(h)\ge\gamma h-Ch^{2}\ge\frac{\gamma}{2}h.
  \]
  By the mean value theorem for the exponential,
  \[
    \bigl|e^{-t\kappa(h)}-e^{-\gamma th}\bigr|
    \le t\,|\kappa(h)-\gamma h|\,e^{-t\min\{\kappa(h),\gamma h\}}
    \le Cth^{2}e^{-\frac{\gamma}{2}th}, \qquad 0\le h\le\varepsilon.
  \]
  Therefore, with $h=|\xi|^{2}$,
  \[
    \int_{|\xi|^{2}\le\varepsilon}\bigl|\tilde f(\xi,t)-e^{-\gamma t|\xi|^{2}}\tilde u_0(\xi)\bigr|d\xi
    \le Ct\|\tilde u_0\|_\infty\int_{\mathbb{R}^n}|\xi|^{4}e^{-\frac{\gamma}{2}t|\xi|^{2}}d\xi
    =O\bigl(t^{-(n+2)/2}\bigr)
  \]
  by \eqref{Section-6-gaussian-moments} with $k=2$. On the fixed complement $|\xi|^{2}>\varepsilon$, the
  monotonicity of the Bernstein function $\kappa$ gives $\kappa(|\xi|^{2})\ge\kappa(\varepsilon)>0$, hence
  \[
    \int_{|\xi|^{2}>\varepsilon}\bigl(|\tilde f(\xi,t)|+e^{-\gamma t|\xi|^{2}}|\tilde u_0(\xi)|\bigr)d\xi
    \le\bigl(e^{-t\kappa(\varepsilon)}+e^{-\gamma\varepsilon t}\bigr)\|\tilde u_0\|_{1}
    =O(e^{-ct}).
  \]
  Combining the two parts proves \eqref{Section-6-gaussian-S}.

  \textbf{Step 2: the time-dual side.} For $u(t)=R_\kappa(t)u_0$, recall from Section 4 that
  \begin{equation}\label{Section-6-u-laplace}
    \int_0^\infty e^{-s t}\tilde u(\xi,t)\,dt
    =V(\xi,s)\,\tilde u_0(\xi),
    \qquad
    V(\xi,s):=\frac{\lambda(s)/s}{\lambda(s)+|\xi|^{2}},
  \end{equation}
  where $\lambda(s)=(s+a)^{1/\beta}-a^{1/\beta}$ is the inverse function of $\kappa$. Write
  $h:=|\xi|^{2}$ and $h_c:=a^{1/\beta}\bigl(1-2^{-1/\beta}\bigr)$, and apply Lemma
  \ref{Section-6-DML} through the scaling identity \eqref{Section-6-scaling-multiplier} with
  $y=h/a^{1/\beta}$, $\tau=at$ and the cutoff $y_0:=\frac{1-2^{-1/\beta}}{2}$, so that
  $h\le h_c/2$ corresponds to $y\le y_0$. The real pole $z_*(y)=r_\beta(y)$ of
  $M_\beta(y,\cdot)$ corresponds to the pole
  \begin{equation}\label{Section-6-pole}
    s_*(h):=a\,r_\beta\bigl(h/a^{1/\beta}\bigr)=(a^{1/\beta}-h)^{\beta}-a=\kappa(-h)
    =-\gamma h-\delta h^{2}+O(h^{3})\qquad(h\to0+)
  \end{equation}
  of $V(\xi,\cdot)$, with residue
  \[
    \operatorname*{Res}_{s=s_*(h)}V(\xi,\cdot)
    =A_\beta\Bigl(\frac{h}{a^{1/\beta}}\Bigr)
    =-\frac{h\,\kappa'(-h)}{\kappa(-h)}
    =1+\frac{\delta}{\gamma}h+O(h^{2})\qquad(h\to0+)
  \]
  (for the last equality use \eqref{Section-6-dml-expansions} and
  $\frac{1-\beta}{2}a^{-1/\beta}=\frac{\delta}{\gamma}$). Hence
  \eqref{Section-6-dml-diffusive} gives, for $h\le h_c/2$,
  \begin{equation}\label{Section-6-u-pole}
    \tilde u(\xi,t)=\tilde u_0(\xi)\,e^{s_*(h)t}\Bigl(1+\frac{\delta}{\gamma}h+O(h^{2})\Bigr)
    +\tilde u_0(\xi)\,\mathcal{R}(\xi,t),
    \qquad |\mathcal{R}(\xi,t)|\le Ce^{-at/2},
  \end{equation}
  with $C$ independent of $\xi$ and of $h\le h_c/2$, and \eqref{Section-6-dml-exterior} gives,
  for $h>h_c/2$,
  \begin{equation}\label{Section-6-u-exterior}
    |\tilde u(\xi,t)|\le Ce^{-c_0t}\bigl(1+\log^{+}h\bigr)|\tilde u_0(\xi)|,
    \qquad c_0:=-\frac{a}{2}\,r_\beta(y_0)>0,
  \end{equation}
  since $1+\log\bigl(1+h/a^{1/\beta}\bigr)\le C\bigl(1+\log^{+}h\bigr)$ for $h>h_c/2$,
  with $\log^{+}h:=\max\{\log h,0\}$.
  The pole factor in \eqref{Section-6-u-pole} admits a uniform first-order expansion:
  with $g(h):=s_*(h)+\gamma h=-\delta h^{2}+O(h^{3})$ uniformly on $[0,h_c/2]$ (Taylor's theorem,
  $s_*$ being analytic on the compact interval) and $g(h)\le0$ on $[0,a^{1/\beta})$ by the
  concavity of $\kappa$ ($g(0)=g'(0)=0$, $g''<0$), one has
  \begin{equation}\label{Section-6-pole-factor-bound}
    \Bigl|e^{s_*(h)t}-e^{-\gamma ht}\bigl(1-\delta th^{2}\bigr)\Bigr|
    \le C\bigl(h^{3}t+h^{4}t^{2}\bigr)e^{-\gamma ht},
    \qquad 0\le h\le h_c/2,\ t>0.
  \end{equation}
  Indeed, with $w:=g(h)t\le0$, Taylor's remainder gives $|e^{w}-1-w|\le w^{2}/2$, and
  \[
    e^{s_*(h)t}-e^{-\gamma ht}\bigl(1-\delta th^{2}\bigr)
    =e^{-\gamma ht}\Bigl[\bigl(e^{g(h)t}-1-g(h)t\bigr)+\bigl(g(h)t+\delta th^{2}\bigr)\Bigr],
  \]
  where the first bracket is at most $Ch^{4}t^{2}$ and the second is $O(th^{3})$ by
  \eqref{Section-6-pole}. Therefore \eqref{Section-6-u-pole} takes the explicit first-order form
  \begin{equation}\label{Section-6-u-expansion}
    \tilde u(\xi,t)=\tilde u_0(\xi)\,e^{-\gamma t|\xi|^{2}}
    \Bigl(1+\frac{\delta}{\gamma}|\xi|^{2}-\delta t|\xi|^{4}
    +O\bigl(|\xi|^{4}+t|\xi|^{6}+t^{2}|\xi|^{8}\bigr)\Bigr)
    +\tilde u_0(\xi)\,\mathcal{R}(\xi,t).
  \end{equation}

  Finally,
  \begin{align*}
    \bigl\|u(t)-K_\gamma(\cdot,t)*u_0\bigr\|_{L^\infty(\mathbb{R}^n)}
    \le(2\pi)^{-n}\Bigl[&\int_{h\le h_c/2}\bigl|\tilde u(\xi,t)-e^{-\gamma t|\xi|^{2}}\tilde u_0(\xi)\bigr|d\xi\\
    &+\int_{h>h_c/2}\bigl(|\tilde u(\xi,t)|+e^{-\gamma t|\xi|^{2}}|\tilde u_0(\xi)|\bigr)d\xi\Bigr].
  \end{align*}
  In the first integral, \eqref{Section-6-u-expansion} gives
  \[
    \bigl|\tilde u(\xi,t)-e^{-\gamma t|\xi|^{2}}\tilde u_0(\xi)\bigr|
    \le\Bigl[e^{-\gamma t|\xi|^{2}}\Bigl(\frac{\delta}{\gamma}|\xi|^{2}+\delta t|\xi|^{4}
    +C\bigl(|\xi|^{4}+t|\xi|^{6}+t^{2}|\xi|^{8}\bigr)\Bigr)+Ce^{-at/2}\Bigr]|\tilde u_0(\xi)|,
  \]
  whose integral is $O(t^{-(n+2)/2})$ by \eqref{Section-6-gaussian-moments} with $k=1,2,3,4$. In the
  second integral, the first term is $O(e^{-c_0t})$ by \eqref{Section-6-u-exterior} and the rapid
  decay of $\tilde u_0$, and the second is $O(e^{-\gamma h_c t/4})$. This proves
  \eqref{Section-6-gaussian-R}, and \eqref{Section-6-gaussian-difference} follows from
  \eqref{Section-6-gaussian-S} and \eqref{Section-6-gaussian-R} by the triangle inequality.
\end{proof}

\begin{cor}[$L^{q}$ version of Theorem \ref{Section-6-long-time-Gaussian}]\label{Section-6-cor-Lq}
  Under the assumptions of Theorem \ref{Section-6-long-time-Gaussian}, for every
  $u_0\in\mathcal{S}(\mathbb{R}^n)$ and every $q\in[2,\infty]$, as $t\to\infty$,
  \begin{equation}\label{Section-6-gaussian-Lq}
    \bigl\|S_\kappa(t)u_0-K_\gamma(\cdot,t)*u_0\bigr\|_{L^{q}(\mathbb{R}^n)}
    =O\Bigl(t^{-\frac{n}{2}\bigl(1-\frac1q\bigr)-1}\Bigr),
    \qquad
    \bigl\|R_\kappa(t)u_0-K_\gamma(\cdot,t)*u_0\bigr\|_{L^{q}(\mathbb{R}^n)}
    =O\Bigl(t^{-\frac{n}{2}\bigl(1-\frac1q\bigr)-1}\Bigr),
  \end{equation}
  i.e., in every $L^{q}$ with $q\ge2$ the difference decays one order faster than the profile
  itself. In particular, the kernel difference $k_t:=S_\kappa(t)\delta_0-K_\gamma(\cdot,t)$
  satisfies $\|k_t\|_{L^{r}(\mathbb{R}^n)}=O\bigl(t^{-\frac{n}{2}(1-\frac1r)-1}\bigr)$ for every
  $r\in[2,\infty]$, and hence, by Young's convolution inequality, for $1\le p\le q\le\infty$
  with $\frac1p-\frac1q\ge\frac12$,
  \begin{equation}\label{Section-6-gaussian-LpLq}
    \bigl\|S_\kappa(t)-K_\gamma(\cdot,t)*\,\cdot\,\bigr\|_{L^{p}\to L^{q}}
    \le\|k_t\|_{L^{r}}
    =O\Bigl(t^{-\frac{n}{2}\bigl(\frac1p-\frac1q\bigr)-1}\Bigr),
    \qquad \tfrac1r=1-\tfrac1p+\tfrac1q.
  \end{equation}
\end{cor}

\begin{proof}
  Write
  $w(t):=S_\kappa(t)u_0-K_\gamma(\cdot,t)*u_0$, so that
  $\tilde w(\xi,t)=\bigl(e^{-t\kappa(|\xi|^{2})}-e^{-\gamma t|\xi|^{2}}\bigr)\tilde u_0(\xi)$,
  and let $q':=q/(q-1)\in[1,2]$. As in Step 1 of the proof of
  Theorem \ref{Section-6-long-time-Gaussian}, choose $\varepsilon>0$ with
  $|\kappa(h)-\gamma h|\le Ch^{2}$ and $\kappa(h)\ge\frac{\gamma}{2}h$ for $0\le h\le\varepsilon$;
  the mean value theorem gives
  \[
    |\tilde w(\xi,t)|\le Ct\,|\xi|^{4}e^{-\frac{\gamma}{2}t|\xi|^{2}}|\tilde u_0(\xi)|,
    \qquad |\xi|^{2}\le\varepsilon.
  \]
  On the complement, $\kappa(h)\ge ch^{\beta}$ for $h\ge\varepsilon$ gives
  $e^{-t\kappa(|\xi|^{2})}\le e^{-ct|\xi|^{2\beta}}$, and
  $e^{-\gamma t|\xi|^{2}}\le e^{-\gamma\varepsilon t/2}e^{-\gamma t|\xi|^{2}/2}$, whence
  $\|\tilde w\|_{L^{q'}(|\xi|^{2}>\varepsilon)}=O(e^{-ct})$ (also for $\tilde u_0\equiv1$).
  Hence, by the Gaussian moments \eqref{Section-6-gaussian-moments},
  \[
    \|\tilde w\|_{L^{q'}}
    \le Ct\,\|\tilde u_0\|_{\infty}\bigl\||\xi|^{4}e^{-\frac{\gamma}{2}t|\xi|^{2}}\bigr\|_{L^{q'}}
    +O(e^{-ct})
    =Ct\cdot t^{-\frac{n}{2q'}-2}
    =Ct^{-\frac{n}{2}\bigl(1-\frac1q\bigr)-1},
  \]
  and the Hausdorff--Young inequality $\|w\|_{L^{q}}\le(2\pi)^{-n/q'}\|\tilde w\|_{L^{q'}}$ gives
  \eqref{Section-6-gaussian-Lq} for $S_\kappa$; for $q=2$ this is simply Plancherel's identity.

  The analogous estimate holds for $R_\kappa(t)u_0-K_\gamma(\cdot,t)*u_0$: on the diffusive
  region, \eqref{Section-6-u-expansion} gives, as in Step 2 of the proof of Theorem
  \ref{Section-6-long-time-Gaussian},
  \[
    \bigl|\tilde u(\xi,t)-e^{-\gamma t|\xi|^{2}}\tilde u_0(\xi)\bigr|
    \le\Bigl[C\bigl(|\xi|^{2}+t|\xi|^{4}+|\xi|^{4}+t|\xi|^{6}+t^{2}|\xi|^{8}\bigr)
    e^{-\gamma t|\xi|^{2}}+Ce^{-at/2}\Bigr]|\tilde u_0(\xi)|,
  \]
  where the amplitude term $|\xi|^{2}e^{-\gamma t|\xi|^{2}}$ is not majorized pointwise by
  the $S_\kappa$-side bound $Ct|\xi|^{4}e^{-\gamma t|\xi|^{2}/2}$, but its $L^{q'}$ norm is
  $O\bigl(t^{-\frac{n}{2q'}-1}\bigr)$ by \eqref{Section-6-gaussian-moments} --- the same rate
  as on the $S_\kappa$ side --- while the remaining terms are faster; the exterior region is
  exponentially damped by \eqref{Section-6-u-exterior}, the factor $1+\log^{+}h$ being
  absorbed by the rapid decay of $\tilde u_0$.

  Finally, for the kernel difference (i.e., $\tilde u_0\equiv1$ on the $S_\kappa$ side) the
  same bounds hold verbatim, giving $\|k_t\|_{L^{r}}=O\bigl(t^{-\frac{n}{2}(1-\frac1r)-1}\bigr)$
  for every $r\in[2,\infty]$, and Young's convolution inequality yields \eqref{Section-6-gaussian-LpLq}.
\end{proof}

\begin{rem}[Optimality of the rate]\label{Section-6-rem-optimal}
  The rate $O\bigl(t^{-(n+2)/2}\bigr)$ in Theorem \ref{Section-6-long-time-Gaussian} is
  optimal: the first-order corrections are genuinely present. Indeed, for
  $|\xi|\le t^{-1/2}$ (so that $t|\xi|^{6}=O(t^{-2})$ and $\delta t|\xi|^{4}=O(t^{-1})$), the expansion
  \eqref{Section-6-kappa-expansion} gives
  \begin{equation}\label{Section-6-f-expansion}
    \tilde f(\xi,t)
    =\tilde u_0(\xi)\,e^{-\gamma t|\xi|^{2}}\exp\bigl(\delta t|\xi|^{4}+O(t|\xi|^{6})\bigr)
    =\tilde u_0(\xi)\,e^{-\gamma t|\xi|^{2}}\bigl(1+\delta t|\xi|^{4}+O\bigl(t|\xi|^{6}+t^{2}|\xi|^{8}\bigr)\bigr);
  \end{equation}
  evaluating the inverse Fourier transforms of \eqref{Section-6-f-expansion} and
  \eqref{Section-6-u-expansion} at $x=0$ (the remaining terms being $o\bigl(t^{-(n+2)/2}\bigr)$
  by the same Gaussian moment computation), one obtains, for every
  $u_0\in\mathcal{S}(\mathbb{R}^n)$ with $\int u_0\neq0$,
  \[
    \bigl(S_\kappa(t)u_0-K_\gamma(\cdot,t)*u_0\bigr)(0)
    =\delta\,\frac{n(n+2)}{4\gamma^{2}}(4\pi\gamma)^{-n/2}\Bigl(\int u_0\Bigr)t^{-(n+2)/2}
    +o\bigl(t^{-(n+2)/2}\bigr),
  \]
  \[
    \bigl(R_\kappa(t)u_0-K_\gamma(\cdot,t)*u_0\bigr)(0)
    =-\delta\,\frac{n^{2}}{4\gamma^{2}}(4\pi\gamma)^{-n/2}\Bigl(\int u_0\Bigr)t^{-(n+2)/2}
    +o\bigl(t^{-(n+2)/2}\bigr),
  \]
  where we used $\Delta K_\gamma(0,t)=-\frac{n}{2\gamma t}(4\pi\gamma t)^{-n/2}$ and
  $\Delta^{2}K_\gamma(0,t)=\frac{n(n+2)}{4\gamma^{2}t^{2}}(4\pi\gamma t)^{-n/2}$. Both leading
  constants are nonzero, since $\delta=\frac{\beta(1-\beta)}{2}a^{1-2/\beta}>0$ for
  $\beta<1$; hence neither estimate in Theorem \ref{Section-6-long-time-Gaussian} can be
  improved in general. For the mutual difference the same evaluation gives the nonzero
  constant $\frac{\delta n(n+1)}{2\gamma^{2}}(4\pi\gamma)^{-n/2}\int u_0$.
\end{rem}

\begin{rem}\label{Section-6-rem-chi}
  (1) Theorem \ref{Section-6-long-time-Gaussian} is stated for the exponent $\kappa$ because of
  the time-dual side; the space-dual estimate \eqref{Section-6-gaussian-S}, by contrast, does not
  use the special form of $\kappa$: its proof requires only the expansion
  \eqref{Section-6-kappa-expansion} of the exponent at $0$ and its monotonicity, and therefore
  applies verbatim to $S_\varphi$ for every Laplace exponent $\varphi$ analytic at the origin
  (with $\gamma=\varphi'(0)$ and $\delta=-\varphi''(0)/2$). The proof of
  \eqref{Section-6-gaussian-R}, on the other hand, relies on the explicit singularity structure
  of the inverse function $\lambda(s)=(s+a)^{1/\beta}-a^{1/\beta}$: the rightmost singularity of
  the symbol is the pole $s_*(h)=\kappa(-h)$, which requires $\kappa$ to continue analytically
  across the negative real axis near $0$. When the pole approaches the branch point as $h$
  varies (the critical region $h\to a^{1/\beta}$), a case-by-case analysis is required; for
  $\beta=1/2$ this is exactly the computation in Appendix B (the region $|\xi|\ge a/c$).
  Extensions of \eqref{Section-6-gaussian-R} to general Laplace exponents are left for future
  work.

  (2) Concerning the range of $\beta$ in Theorem \ref{Section-6-long-time-Gaussian}: the space-dual
  conclusion \eqref{Section-6-gaussian-S} holds for every $\beta\in(0,1)$, since $S_\kappa(t)$ is always
  well-defined; the time-dual conclusion \eqref{Section-6-gaussian-R} holds for $\beta\in(1/2,1)$ by Theorem
  \ref{existence-R(t)-theorem-1} and for $\beta=1/2$ by the telegraph theory of Section 3
  (Theorem \ref{Section-3-main-thm} and Case 2 of Section 3.3), so that the full
  dual-pair statement is valid for $\beta\in[1/2,1)$. The limiting diffusion coefficient
  $D=\beta a^{1-1/\beta}$ coincides with the constant $D$ in Theorem \ref{Section-6-nonrelativistic} when
  $c=1$. For $\beta=1/2$ one has $\gamma=\frac{1}{2a}$ and $\delta=\frac{1}{8a^{3}}$, so that
  $\frac{\delta}{\gamma}=\frac{D}{2a}$ and $\delta=\frac{D^{2}}{2a}$ with $D=\frac{1}{2a}$: Theorem
  \ref{Section-6-long-time-Gaussian} reproduces exactly the asymptotics of Section 3.3 (Case 2), and
  the two first-order corrections are explicit and opposite in phase: the space-dual side has the
  phase correction $+\delta t|\xi|^{4}=+\frac{D^{2}}{2a}|\xi|^{4}t$ (from
  \eqref{Section-6-f-expansion}), while the time-dual side, i.e., the telegraph equation of
  Appendix B, has the phase correction $-\frac{D^{2}}{2a}|\xi|^{4}t$ and the amplitude correction
  $1+\frac{D}{2a}|\xi|^{2}$. This cancellation is not special to $\beta=1/2$: for every
  $\beta\in[1/2,1)$ the phase corrections of the two dual solutions have opposite signs and cancel
  in the average, leaving a constant time shift of the Gaussian profile (Proposition
  \ref{Section-6-prop-average} below).

\end{rem}

Theorem \ref{Section-6-long-time-Gaussian} shows that the two dual families share the same
Gaussian long-time profile, and Remark \ref{Section-6-rem-chi} (2) indicates that their
first-order corrections have phase terms of opposite signs. The simplest symmetric combination
of the dual pair is the plain average, and it is natural to ask whether this cancellation
improves the Gaussian approximation. The next proposition answers this affirmatively: the
average is Gaussian to first order up to a constant time shift, with an error one order
smaller than each side's own.

\begin{prop}[First-order cancellation in the dual average]\label{Section-6-prop-average}
  Let $\kappa$, $\gamma$ and $\delta$ be as in Theorem \ref{Section-6-long-time-Gaussian} (in
  particular $\beta\in[1/2,1)$), and set $\tau:=\frac{\delta}{2\gamma^{2}}\ge0$. Then, for every
  $u_0\in\mathcal{S}(\mathbb{R}^n)$,
  \begin{equation}\label{Section-6-average}
    \Bigl\|\frac{S_\kappa(t)u_0+R_\kappa(t)u_0}{2}-K_\gamma(\cdot,t-\tau)*u_0\Bigr\|_{L^\infty(\mathbb{R}^n)}
    =O\bigl(t^{-(n+4)/2}\bigr),\qquad t\to\infty.
  \end{equation}
  Thus the first-order phase corrections of the two dual solutions have opposite signs and cancel
  in the average, and the only surviving first-order effect is the constant time shift $\tau$ of
  the Gaussian profile.
\end{prop}

\begin{proof}
  Assume first $\beta\in(1/2,1)$; the boundary case $\beta=1/2$ is treated at the end. Write
  $h:=|\xi|^{2}$ and split the Fourier side at a fixed cutoff $h_*\in(0,h_c/2]$, where $h_c/2$ is
  the diffusive cutoff of Step 2 of the proof of Theorem \ref{Section-6-long-time-Gaussian},
  diminished if necessary so that, uniformly on $[0,h_*]$,
  \[
    \kappa(h)=\gamma h-\delta h^{2}+O(h^{3}),\qquad
    \kappa(h)\ge\frac{\gamma}{2}h,\qquad
    -(\kappa(h)-\gamma h)=\delta h^{2}\bigl(1+O(h)\bigr)\ge0.
  \]

  \emph{Low region $h\le h_*$.} On the space-dual side, with
  $w:=-t(\kappa(h)-\gamma h)\ge0$ and $|w|\le Cth^{2}$, the inequality
  $|e^{w}-1-w|\le\frac{w^{2}}{2}e^{w}$ and $e^{-\gamma th+w}=e^{-t\kappa(h)}\le e^{-\gamma th/2}$
  (using $\kappa(h)\ge\frac{\gamma}{2}h$) give
  \[
    \tilde f(\xi,t)=\tilde u_0(\xi)e^{-\gamma th}\bigl(1+\delta th^{2}\bigr)+\tilde u_0(\xi)E_f(\xi,t),
    \qquad
    |E_f(\xi,t)|\le C\bigl(th^{3}+t^{2}h^{4}\bigr)e^{-\gamma th/2}.
  \]
  On the time-dual side, \eqref{Section-6-u-pole} and \eqref{Section-6-pole-factor-bound} give
  \[
    \tilde u(\xi,t)=\tilde u_0(\xi)e^{-\gamma th}\Bigl(1+\frac{\delta}{\gamma}h-\delta th^{2}\Bigr)
    +\tilde u_0(\xi)E_u(\xi,t),
    \qquad
    |E_u(\xi,t)|\le C\bigl(h^{2}+th^{3}+t^{2}h^{4}\bigr)e^{-\gamma th}+Ce^{-at/2},
  \]
  where the $h^{2}$ term comes from the amplitude
  $A_\beta(h/a^{1/\beta})=1+\frac{\delta}{\gamma}h+O(h^{2})$ and the remaining ones from
  \eqref{Section-6-pole-factor-bound} (the cross term $\delta th^{2}\cdot O(h^{2})$ is absorbed
  into $th^{3}$ on $h\le h_*$). Averaging cancels the phase terms $\pm\delta th^{2}$:
  \[
    \frac{\tilde f(\xi,t)+\tilde u(\xi,t)}{2}
    =\tilde u_0(\xi)e^{-\gamma th}\Bigl(1+\frac{\delta}{2\gamma}h\Bigr)+\tilde u_0(\xi)E(\xi,t),
    \qquad
    |E|\le C\bigl(h^{2}+th^{3}+t^{2}h^{4}\bigr)e^{-\gamma th/2}+Ce^{-at/2}.
  \]
  With $\tau=\frac{\delta}{2\gamma^{2}}$ one has $\gamma\tau=\frac{\delta}{2\gamma}$, and for
  $t\ge2\tau$,
  $\bigl|e^{\gamma\tau h}-1-\gamma\tau h\bigr|\le\frac{(\gamma\tau h)^{2}}{2}e^{\gamma\tau h}
  \le Ch^{2}e^{\gamma th/2}$, so
  \[
    \tilde u_0(\xi)e^{-\gamma th}\Bigl(1+\frac{\delta}{2\gamma}h\Bigr)
    =\tilde u_0(\xi)e^{-\gamma(t-\tau)h}+\tilde u_0(\xi)E_\tau(\xi,t),
    \qquad
    |E_\tau|\le Ch^{2}e^{-\gamma th/2}.
  \]
  The Gaussian moments \eqref{Section-6-gaussian-moments} with $k=2,3,4$ now give
  \[
    \int_{h\le h_*}\bigl(|E|+|E_\tau|\bigr)|\tilde u_0(\xi)|\,d\xi
    \le C\|\tilde u_0\|_{\infty}\Bigl(t^{-\frac{n+4}{2}}+t\cdot t^{-\frac{n+6}{2}}
    +t^{2}\cdot t^{-\frac{n+8}{2}}\Bigr)+Ce^{-at/2}\|\tilde u_0\|_{1}
    =O\bigl(t^{-(n+4)/2}\bigr).
  \]

  \emph{Exterior region $h>h_*$.} By the monotonicity of $\kappa$,
  $|\tilde f(\xi,t)|\le e^{-t\kappa(h_*)}|\tilde u_0(\xi)|$; by \eqref{Section-6-u-exterior},
  $|\tilde u(\xi,t)|\le Ce^{-c_0t}\bigl(1+\log^{+}h\bigr)|\tilde u_0(\xi)|$; and
  $e^{-\gamma(t-\tau)h}\le e^{-\gamma h_*t/2}$ for $t\ge2\tau$. All three contributions are
  $O(e^{-ct})$, the factor $1+\log^{+}h$ being absorbed by the rapid decay of $\tilde u_0$.
  Combining the two regions and applying $\|g\|_{L^\infty}\le(2\pi)^{-n}\|\tilde g\|_{L^1}$ to
  $g:=\frac{f(t)+u(t)}{2}-K_\gamma(\cdot,t-\tau)*u_0$ gives \eqref{Section-6-average}.

  \emph{The boundary case $\beta=1/2$.} The Dimensionless Multiplier Lemma was proved only for
  $\beta>1/2$, but then $R_\kappa(t)u_0$ is the telegraph solution with $c=1$ and
  $D=\gamma=\frac{1}{2a}$: the explicit multiplier \eqref{appendix-Fourier-TE} satisfies
  $r_+(\xi)=-Dh-\frac{D^{2}}{2a}h^{2}+O(h^{3})\le-Dh$ and
  $\frac{r_-}{r_--r_+}=1+\frac{D}{2a}h+O(h^{2})$, where $\frac{D}{2a}=\frac{\delta}{\gamma}$ and
  $\frac{D^{2}}{2a}=\delta$ (Appendix B), so the same low-region expansion of $\tilde u$ holds by
  the argument of \eqref{Section-6-pole-factor-bound}, while the exterior is covered by
  \eqref{appendix-TE-complement}; the space-side computation is unchanged. Hence
  \eqref{Section-6-average} holds for all $\beta\in[1/2,1)$.
\end{proof}

\begin{rem}\label{Section-6-rem-average}
  (1) The gain is quantitative and free of charge: each side approximates $K_\gamma(\cdot,t)$
  at the optimal rate $O(t^{-(n+2)/2})$ (Remark \ref{Section-6-rem-optimal}), while the plain
  average attains $O(t^{-(n+4)/2})$ around the slightly shifted profile $K_\gamma(\cdot,t-\tau)$.
  The mechanism is already visible at $x=0$: the first-order constants of the two sides in
  Remark \ref{Section-6-rem-optimal} average to
  $\frac{\delta n}{4\gamma^{2}}(4\pi\gamma)^{-n/2}\int u_0$, which is exactly the leading term
  $\frac{n\tau}{2}(4\pi\gamma)^{-n/2}\int u_0$ of the time-shifted kernel at $x=0$.

  (2) As an approximation scheme, the average converts two first-order approximations into a
  second-order one: whenever the time-nonlocal side $R_\kappa(t)$ is expensive to evaluate,
  $2K_\gamma(\cdot,t-\tau)*u_0-S_\kappa(t)u_0$ approximates $R_\kappa(t)u_0$ with error
  $O(t^{-(n+4)/2})$. To the best of our knowledge, this first-order cancellation between
  Bochner and reverse subordination has not been noted in the literature; higher-order
  symmetric combinations and extensions to more general exponents are left for future work.
\end{rem}

\subsection{The nonrelativistic limit of the relativistic stable pair}

We now return to the relativistic stable operator of Example \ref{general-example-3} and introduce the speed $c$ as
in Section 3. Consider the dual pair
\begin{equation}\label{Section-6-NLE-c}
  \frac{\partial f_c}{\partial t}(x,t)=\bigl[a-\bigl(a^{1/\beta}-c^{2}\Delta\bigr)^{\beta}\bigr]f_c(x,t),
\end{equation}
\begin{equation}\label{Section-6-TLE-c}
  \Phi_t^{(c)}u_c(x,t)=c^{2}\Delta u_c(x,t),
\end{equation}
where $a=a(c), \beta\in(0,1)$ and the time operator $\Phi_t^{(c)}$ is determined by the symbol relation
\[
  \lambda_c(s)=(s+a)^{1/\beta}-a^{1/\beta}.
\]

\begin{thm}\label{Section-6-nonrelativistic}
  Let $\beta\in[1/2,1)$ and $D>0$ be fixed, and let
  \begin{equation}\label{Section-6-scaling}
    a=a(c)=\Bigl(\frac{\beta c^{2}}{D}\Bigr)^{\frac{\beta}{1-\beta}},
    \qquad\text{equivalently}\qquad
    D=\beta c^{2}a^{1-1/\beta}.
  \end{equation}
  Then, as $c\to\infty$:

  (i) the Fourier symbol of \eqref{Section-6-NLE-c} satisfies, for each fixed $\xi\in\mathbb{R}^n$,
  \[
    \Psi_c(\xi):=a-\bigl(a^{1/\beta}+c^{2}|\xi|^{2}\bigr)^{\beta}
    =-D|\xi|^{2}+\frac{(1-\beta)D^{2}}{2\beta a}\,|\xi|^{4}+O\bigl(c^{6}a^{1-3/\beta}|\xi|^{6}\bigr),
  \]
  in particular $\Psi_c(\xi)=-D|\xi|^{2}\bigl(1+O\bigl(c^{-2\beta/(1-\beta)}\bigr)\bigr)$;

  (ii) the Laplace symbol of \eqref{Section-6-TLE-c} satisfies, for each fixed $s>0$,
  \[
    \lambda_c(s)=\frac{c^{2}}{D}\,s+\frac{(1-\beta)c^{2}}{2\beta D a}\,s^{2}
    +O\bigl(a^{1/\beta-3}s^{3}\bigr),
  \]
  in particular $\frac{D}{c^{2}}\lambda_c(s)\to s$, i.e., the rescaled time operator
  $\frac{D}{c^{2}}\Phi_t^{(c)}$ converges to $\partial_t$ at the symbol level;

  (iii) consequently, for every $u_0\in\mathcal{S}(\mathbb{R}^n)$, the solutions $f_c,u_c$ of
  the Cauchy problems for \eqref{Section-6-NLE-c} and \eqref{Section-6-TLE-c} with initial datum $u_0$
  satisfy, as $c\to\infty$,
  \[
    \bigl\|f_c(\cdot,t)-K_D(\cdot,t)*u_0\bigr\|_{L^\infty(\mathbb{R}^n)}=O\bigl(c^{-2\beta/(1-\beta)}\bigr)
    \quad\text{uniformly in }t\ge0,
  \]
  and
  \[
    \bigl\|u_c(\cdot,t)-K_D(\cdot,t)*u_0\bigr\|_{L^\infty(\mathbb{R}^n)}
    =O\bigl(c^{-2\beta/(1-\beta)}\bigr)
    \quad\text{uniformly in }t\ge0.
  \]
\end{thm}

\begin{proof}
  For $\beta=1/2$ the coupling \eqref{Section-6-scaling} is Kac's scaling $a=c^{2}/(2D)$, and the
  $u_c$-part of (iii) is contained in Section 3.3 (Case 1) and the corresponding computations of
  Appendix B; the remaining statements are proved below for all $\beta\in[1/2,1)$, so we assume
  $\beta\in(1/2,1)$ only in the proof of the $u_c$-part of (iii).

  (i) Factoring out $a^{1/\beta}$ and expanding the binomial,
  \[
    \Psi_c(\xi)
    =-a\Bigl[\beta\,\frac{c^{2}|\xi|^{2}}{a^{1/\beta}}
    +\frac{\beta(\beta-1)}{2}\,\frac{c^{4}|\xi|^{4}}{a^{2/\beta}}
    +O\Bigl(\frac{c^{6}|\xi|^{6}}{a^{3/\beta}}\Bigr)\Bigr].
  \]
  The first term equals $-D|\xi|^{2}$ by \eqref{Section-6-scaling}; using $c^{4}a^{2-2/\beta}=D^{2}/\beta^{2}$,
  the second term equals $\frac{(1-\beta)D^{2}}{2\beta a}|\xi|^{4}$, and
  $a^{-1}=O\bigl(c^{-2\beta/(1-\beta)}\bigr)$.

  (ii) Similarly,
  \[
    \lambda_c(s)=a^{1/\beta}\Bigl[\frac{s}{\beta a}
    +\frac{(1/\beta)(1/\beta-1)}{2}\,\frac{s^{2}}{a^{2}}
    +O\Bigl(\frac{s^{3}}{a^{3}}\Bigr)\Bigr]
    =\frac{c^{2}}{D}\,s+\frac{(1-\beta)c^{2}}{2\beta Da}\,s^{2}+O\bigl(a^{1/\beta-3}s^{3}\bigr),
  \]
  where we used $\frac{1}{\beta}a^{1/\beta-1}=\frac{c^{2}}{D}$, again by \eqref{Section-6-scaling}.

  (iii) \emph{Convergence of $f_c$ (with rate).} Write $w:=c^{2}|\xi|^{2}/a^{1/\beta}$, so that
  $\Psi_c(\xi)=-a\bigl[(1+w)^{\beta}-1\bigr]$. The Taylor formula with integral remainder gives, for
  $w\ge0$,
  \[
    (1+w)^{\beta}-1-\beta w
    =\beta(\beta-1)w^{2}\int_0^1(1-s)(1+sw)^{\beta-2}\,ds,
  \]
  whose modulus is at most $\frac{\beta(1-\beta)}{2}w^{2}$, since $(1+sw)^{\beta-2}\le1$. Because
  $aw=c^{2}a^{1-1/\beta}|\xi|^{2}=\frac{D}{\beta}|\xi|^{2}$ by \eqref{Section-6-scaling}, this yields
  the global bound
  \begin{equation}\label{Section-6-Psi-global}
    \bigl|\Psi_c(\xi)+D|\xi|^{2}\bigr|
    \le\frac{\beta(1-\beta)}{2}\,aw^{2}
    =\frac{\beta(1-\beta)}{2}\Bigl(\frac{\beta}{D}\Bigr)^{\frac{\beta-2}{1-\beta}}
    \frac{|\xi|^{4}}{c^{2\beta/(1-\beta)}}.
  \end{equation}
  Moreover, the quotient $\bigl((1+w)^{\beta}-1\bigr)/(\beta w)$ being decreasing, one has
  $(1+w)^{\beta}-1\ge w/2^{1/\beta}$ on $\{w\le2^{1/\beta}\}$ and $(1+w)^{\beta}-1\ge w^{\beta}/2$
  on its complement; hence, with $q:=|\xi|^{2}$,
  \begin{equation}\label{Section-6-fc-majorant}
    e^{t\Psi_c(\xi)}
    \le\exp\Bigl(-\frac{Dtq}{\beta\,2^{1/\beta}}\Bigr)+\exp\Bigl(-\frac{tq^{\beta}}{2}\Bigr)
    \qquad(c\ge1).
  \end{equation}
  Since $\Psi_c(\xi)\le0$, the mean value theorem gives
  $\bigl|e^{t\Psi_c}-e^{-Dtq}\bigr|\le t\bigl|\Psi_c+Dq\bigr|\bigl(e^{t\Psi_c}+e^{-Dtq}\bigr)$;
  combining \eqref{Section-6-Psi-global} with \eqref{Section-6-fc-majorant} and using
  $tq^{2}e^{-c_1Dtq}\le\frac{q}{c_1D}$ (from $xe^{-x}\le1$, with
  $c_1:=\bigl(\beta2^{1/\beta}\bigr)^{-1}$) and $tq^{2}e^{-tq^{\beta}/2}\le2q^{2-\beta}$ (from
  $xe^{-x/2}\le2$), one obtains, uniformly in $t\ge0$ and $c\ge1$,
  \[
    \bigl|e^{t\Psi_c(\xi)}-e^{-Dt|\xi|^{2}}\bigr|
    \le C_{\beta,D}\,a^{-1}\bigl(|\xi|^{2}+|\xi|^{4-2\beta}\bigr).
  \]
  Fourier inversion (via $\|g\|_{L^\infty}\le(2\pi)^{-n}\|\tilde g\|_{L^1}$) applied to
  $g=f_c(\cdot,t)-K_D(\cdot,t)*u_0$, whose transform is
  $\bigl(e^{t\Psi_c}-e^{-Dt|\xi|^{2}}\bigr)\tilde u_0$, therefore gives
  \[
    \bigl\|f_c(\cdot,t)-K_D(\cdot,t)*u_0\bigr\|_{L^\infty(\mathbb{R}^n)}
    \le\frac{C_{\beta,D}\,\bigl\|\bigl(|\xi|^{2}+|\xi|^{4-2\beta}\bigr)\tilde u_0\bigr\|_{L^1}}
             {(2\pi)^{n}\,c^{2\beta/(1-\beta)}},
    \qquad t\ge0.
  \]

  \emph{Convergence of $u_c$ (with rate).} Write $q:=|\xi|^{2}$ and
  $\kappa_c(h):=-a+(a^{1/\beta}+c^{2}h)^{\beta}$, so that $\lambda_c(\kappa_c(h))=c^{2}h$: the
  symbol
  \[
    V_c(\xi,s):=\frac{\lambda_c(s)/s}{\lambda_c(s)+c^{2}|\xi|^{2}}
  \]
  coincides with the symbol $V(\xi,s)$ of \eqref{Section-6-u-laplace} with $h=c^{2}q$;
  hence the scaling identity \eqref{Section-6-scaling-multiplier} gives
  \[
    \tilde u_c(\xi,t)=\tilde u_0(\xi)\,m_\beta\bigl(y_c,at\bigr),
    \qquad
    y_c:=\frac{c^{2}q}{a^{1/\beta}}=\frac{Dq}{\beta a},
  \]
  where the last equality is \eqref{Section-6-scaling}. The gain of the dimensionless formulation
  is that the cutoff of Lemma \ref{Section-6-DML} is $c$-independent: with
  $y_0:=\frac{1-2^{-1/\beta}}{2}$, the diffusive region $y_c\le y_0$ reads
  $q\le\varepsilon_0a$, $\varepsilon_0:=\frac{\beta(1-2^{-1/\beta})}{2D}$. The coupling
  \eqref{Section-6-scaling} pins the first-order coefficient of $\kappa_c$ at the origin to the
  prescribed value
  \[
    \gamma_c:=\kappa_c'(0)=\beta c^{2}a^{1-1/\beta}=D,
  \]
  while the second-order coefficient satisfies
  \[
    \delta_c:=-\frac{\kappa_c''(0)}{2}=\frac{\beta(1-\beta)}{2}c^{4}a^{1-2/\beta}
    =\frac{(1-\beta)D^{2}}{2\beta a}=O(a^{-1}),
  \]
  by the identity $c^{4}a^{2-2/\beta}=D^{2}/\beta^{2}$ from (i).

  \emph{Diffusive region $y_c\le y_0$ (i.e., $q\le\varepsilon_0a$).} By
  \eqref{Section-6-dml-diffusive},
  \[
    \tilde u_c(\xi,t)
    =\tilde u_0(\xi)\Bigl[A_\beta(y_c)\,e^{a\,r_\beta(y_c)t}+\rho_\beta(y_c,at)\Bigr],
    \qquad
    |\rho_\beta(y_c,at)|\le C\,y_c\,e^{-at/2}.
  \]
  Here $a\,r_\beta(y_c)=\kappa_c(-q)$, and the expansions \eqref{Section-6-dml-expansions},
  multiplied by $a$ and written in terms of $q$ via $y_c=\frac{Dq}{\beta a}$, read
  \[
    \kappa_c(-q)=-Dq-\delta_c q^{2}+O(a^{-2}q^{3}),
    \qquad
    A_\beta(y_c)=1+\frac{\delta_c}{D}q+O(a^{-2}q^{2}),
  \]
  uniformly in $c$: indeed $a\,y_c^{3}=\bigl(\frac{D}{\beta}\bigr)^{3}a^{-2}q^{3}$ and
  $y_c^{2}=\bigl(\frac{D}{\beta}\bigr)^{2}a^{-2}q^{2}$, while
  $\frac{1-\beta}{2}\,\frac{D}{\beta a}=\frac{\delta_c}{D}$. With
  $g_c(q):=\kappa_c(-q)+Dq\le0$ (by the concavity of $\kappa_c$: $g_c(0)=g_c'(0)=0$ and
  $g_c''<0$), the inequality $|e^{w}-1-w|\le w^{2}/2$ ($w\le0$) yields
  \[
    \Bigl|e^{\kappa_c(-q)t}-e^{-Dqt}\bigl(1-\delta_c tq^{2}\bigr)\Bigr|
    \le C\bigl(a^{-2}q^{3}t+a^{-2}q^{4}t^{2}\bigr)e^{-Dqt},
  \]
  and therefore
  \begin{equation}\label{Section-6-uc-expansion}
    \tilde u_c(\xi,t)
    =\tilde u_0(\xi)e^{-Dqt}\Bigl(1-\delta_c tq^{2}+\frac{\delta_c}{D}q
    +O\bigl(a^{-2}(q^{2}+q^{3}t+q^{4}t^{2})\bigr)\Bigr)
    +\tilde u_0(\xi)\mathcal{R}(\xi,t),
  \end{equation}
  with $|\mathcal{R}(\xi,t)|\le Cy_ce^{-at/2}\le Ca^{-1}q$, where $C$ is independent of $c$ and of
  $q\le\varepsilon_0a$.

  \emph{Exterior region $q>\varepsilon_0a$.} The crude bound \eqref{Section-6-dml-crude} gives
  the $c$-uniform multiplier estimate
  \begin{equation}\label{Section-6-uc-bound}
    |\tilde u_c(\xi,t)|
    \le\Bigl(1+\frac{CD}{\beta}\,tq\Bigr)|\tilde u_0(\xi)|,
    \qquad c\ge1,\ t>0.
  \end{equation}
  Hence, for $0\le t\le1$, the rapid decay of $\tilde u_0$ gives
  \[
    \int_{q>\varepsilon_0a}\bigl(|\tilde u_c(\xi,t)|+e^{-Dqt}|\tilde u_0(\xi)|\bigr)d\xi
    \le C\int_{q>\varepsilon_0a}(1+q)|\tilde u_0(\xi)|\,d\xi
    =O(a^{-N})
  \]
  for every $N$; for $t\ge1$ the exterior bound \eqref{Section-6-dml-exterior} gives the
  stronger decay
  $|\tilde u_c(\xi,t)|\le Ce^{-c_*at}\bigl(1+\log(1+y_c)\bigr)|\tilde u_0(\xi)|$ with
  $c_*:=-r_\beta(y_0)/2>0$, whose integral is $O\bigl(e^{-c_*a}\log a\bigr)=O(a^{-N})$.

  Combining the two regions with \eqref{Section-6-uc-expansion} and applying the inversion
  bound $\|g\|_{L^\infty}\le(2\pi)^{-n}\|\tilde g\|_{L^1}$ therefore gives
  \[
    \bigl\|u_c(\cdot,t)-K_D(\cdot,t)*u_0\bigr\|_{L^\infty(\mathbb{R}^n)}
    =O\bigl(c^{-2\beta/(1-\beta)}\bigr)
    \quad\text{uniformly in }t\ge0:
  \]
  on the diffusive region, $tqe^{-Dqt}\le\frac{C}{D}$ and $t^{2}q^{2}e^{-Dqt}\le\frac{C}{D^{2}}$
  (from $xe^{-x}\le C$) show that the amplitude correction contributes
  $\frac{\delta_c}{D}\|q\tilde u_0\|_1$, the phase correction contributes
  $C\frac{\delta_c}{D}\|q\tilde u_0\|_1$, the $O(a^{-2})$ remainder terms contribute
  $O(a^{-2})$, and $|\mathcal{R}(\xi,t)|\le Ca^{-1}q$ contributes $O(a^{-1})$, all uniformly in
  $t\ge0$.
\end{proof}

\begin{rem}
  (1) The coupling \eqref{Section-6-scaling} is not a technical choice but the unique scaling
  with $\beta c^{2}a^{1-1/\beta}\to D$: the $c$-dependent exponent
  $\kappa_c(h)=-a+(a^{1/\beta}+c^{2}h)^{\beta}$ has effective diffusion coefficient
  $\kappa_c'(0)=\beta c^{2}a^{1-1/\beta}$, and pinning this coefficient at the prescribed value
  $D$ is equivalent to \eqref{Section-6-scaling}. Any slower growth of $a$ makes the limiting
  diffusivity infinite, any faster growth makes it vanish; under the coupling, all higher-order
  (relativistic) corrections are automatically of order $a^{-1}=O(c^{-2\beta/(1-\beta)})$. For
  $\beta=1/2$ this reduces to Kac's scaling $D=c^{2}/(2a)$ \cite{Kac1974}, i.e., constant rest
  mass $m=\hbar/(2D)$.

  (2) For $\beta\neq1/2$ the forced growth $a\propto c^{2\beta/(1-\beta)}$ of the mass-type
  parameter (e.g., $a\propto c^{4}$ for $\beta=2/3$, in contrast with the classical relativistic
  scaling $a\propto c^{2}$) also has a physical reading: in natural units, as in the standard
  relativistic stable generator $(m^{2/\alpha}-\Delta)^{\alpha/2}-m$ with $\alpha=2\beta$
  \cite{Ryznar2002}, the two terms $(mc^{2})^{1/\beta}$ and $c^{2}p^{2}$ in a candidate energy
  relation $E(p)=\bigl((mc^{2})^{1/\beta}+c^{2}p^{2}\bigr)^{\beta}$ have different physical
  dimensions, so no energy relation involving only $m,c,\hbar$ exists. Thus, among the
  relativistic stable models \eqref{Section-6-NLE-c}, only $\beta=1/2$ admits a nonrelativistic
  limit with a constant mass parameter in the relation $a=mc^{2}/\hbar$.

  (3) The restriction $\beta\ge1/2$ in the theorem is needed only for the time-dual side: statement (i) and
  the $f_c$-part of (iii) hold for every $\beta\in(0,1)$, the proof using only the concavity of
  $(\cdot)^\beta$; for $\beta<1/2$ the time-dual equation has order $1/\beta>2$ and lies outside the
  framework of Section 4.
\end{rem}

\appendix

  \section{Analytic continuation derivation}

  In this section, we review the derivations of the two concrete relativistic diffusion equations \eqref{introduction-TE} and \eqref{introduction-NLE} via analytic continuation, as shown in \cite{Baeumer2010}. Further interpretations are also presented.

  Our starting point is the free relativistic Schr\"odinger equation \cite{Carmona1990JFA,Lieb1990BAMS,Lieb2001book}:
  \begin{equation}\label{free-relativistic-Schrodinger-equation}
    i\hbar\frac{\partial \Psi}{\partial t}(x,t)=\bigl[-mc^2+\sqrt{m^2c^4-\hbar^2c^2\Delta}\bigr]\Psi(x,t), \qquad x\in\mathbb{R}^n,
  \end{equation}
 where  the operator $\sqrt{m^2c^4-\hbar^2c^2\Delta}$ is defined  by the Fourier transform
  \[
    \mathcal{F}\bigl(\sqrt{m^2c^4-\hbar^2c^2\Delta}\Psi\bigr)(\xi)= \sqrt{m^2c^4+\hbar^2c^2\|\xi\|^2}\;\mathcal{F}(\Psi)(\xi).
  \]

  Consider the following three analytic continuations:

  (1) Let $\tau=it$, $f(x,\tau)=\Psi(x,t)=\Psi(x,\tau/i)$. Then
  \begin{align*}
  \hbar\frac{\partial f}{\partial \tau}(x,\tau)&=\hbar \left(\frac{1}{i} \frac{\partial \Psi}{\partial t}\left(x,\frac{\tau}{i}\right)\right)\\
  &=-i\hbar  \frac{\partial \Psi}{\partial t}\left(x,\frac{\tau}{i}\right)\\
  &=-\bigl[-mc^2+\sqrt{m^2c^4-\hbar^2c^2\Delta}\bigr]\Psi\left(x,\frac{\tau}{i}\right)\\
  &=-\bigl[-mc^2+\sqrt{m^2c^4-\hbar^2c^2\Delta}\bigr]f(x,\tau).
  \end{align*}
  We obtain
  \[
    \hbar\frac{\partial f}{\partial t}(x,t)=-\bigl[-mc^2+\sqrt{m^2c^4-\hbar^2c^2\Delta}\bigr]f(x,t),
  \]
  where we still use $t$ to denote the time variable.

  (2) By using the analytic continuation $m\rightarrow im$ in \eqref{free-relativistic-Schrodinger-equation}, we have formally
  \[
    i\hbar\dfrac{\partial \Psi}{\partial t}=[-imc^2+\sqrt{(im)^2c^4-\hbar^2c^2\Delta}]\Psi=i[-mc^2+\sqrt{m^2c^4+\hbar^2c^2\Delta}]\Psi.
  \]

  Then
  \begin{equation}\label{Section-4-TE-previous}
    \hbar\frac{\partial \Psi}{\partial t}=\bigl[-mc^2+\sqrt{m^2c^4+\hbar^2c^2\Delta}\bigr]\Psi.
  \end{equation}
  Here $\sqrt{m^2c^4+\hbar^2c^2\Delta}\Psi(x)$ is defined formally by the Fourier transform
  \[
    \mathcal{F}\bigl(\sqrt{m^2c^4+\hbar^2c^2\Delta}\Psi\bigr)(\xi)= \sqrt{m^2c^4-\hbar^2c^2\|\xi\|^2}\;\mathcal{F}(\Psi)(\xi).
  \]

  However, since the symbol of the operator is not positive definite,
  one may use the fractional-order definition of the wave operator to obtain a rigorous expression \cite{Enciso2017JFA}.
  Note also that interpreting the expression
  \[
    \bigl[-mc^2+\sqrt{m^2c^4+\hbar^2c^2\Delta}\bigr]\Psi
  \]
  physically is rather difficult \cite{Baeumer2010,Enciso2017JFA}.

  To bypass both the unclear physical interpretation and the inconvenient fractional-order definition of the wave operator,
  we perform a formal computation:
  \[
    \left( \hbar\dfrac{\partial }{\partial t}+ mc^2 \right)^2\Psi=(m^2c^4+\hbar^2c^2\Delta)\Psi.
  \]
  Then
  \[
    \hbar^2\dfrac{\partial^2 \Psi}{\partial t^2}+2mc^2\hbar\dfrac{\partial \Psi}{\partial t}=\hbar^2c^2\Delta\Psi.
  \]
  Finally, we obtain the telegraph equation \cite{Baeumer2010,Gaveau1984}
  \begin{equation}\label{Section-4-TE}
    \dfrac{\partial^2 u}{\partial t^2}+\frac{2mc^2}{\hbar}\dfrac{\partial u}{\partial t}=c^2\Delta u,
  \end{equation}
  where we have replaced the function symbol $\Psi$ by $u$.
  In contrast to \eqref{Section-4-TE-previous}, the telegraph equation \eqref{Section-4-TE}
 admits a clear physical interpretation \cite{Dunkel2009PRYREP,Gaveau1984,Joseph1989, Masoliver1996, Ozisik1994}.
Note also that squaring is not a reversible operation: while every solution of \eqref{Section-4-TE-previous} satisfies \eqref{Section-4-TE}, the converse is false; the telegraph equation admits spurious solutions that do not correspond to solutions of the original first-order equation.

  (3) Analytic continuation  $\hbar\to -i\hbar$  yields the same results as $m\to im$.

\section{Auxiliary calculations for the asymptotic behaviors}\label{appendix}

We provide the detailed calculations of the asymptotic properties of the telegraph equation as
described in Section 3.3; the corresponding limits of the nonlocal diffusion equation are the
boundary instances $\beta=1/2$ of the general theorems of Section 5 (see Section 3.3 for the
precise references). These calculations are elementary and already known.
Throughout this paper, the fundamental solution of an evolution equation is understood in the initial-value
sense, i.e., the solution operator applied to the Dirac distribution $\delta(x)$ (with vanishing initial
time-derivatives where applicable). For equations of first order in time, this coincides with the fundamental
solution in the sense of a point source $\delta(x)\delta(t)$.

Fix $D=\frac{\hbar}{2m}$, $a=\frac{mc^2}{\hbar}=\frac{c^2}{2D}$. Let
\[
K(x,t)=\frac{1}{(4\pi D t)^{n/2}}\exp\left(-\frac{|x|^2}{4D t}\right), \qquad x\in\mathbb{R}^n, t>0
\]
be the standard heat kernel with diffusion coefficient $D$.
Its Fourier transform is
\[
\tilde K(\xi,t)=e^{-D t|\xi|^2}.
\]

Both computations below follow the same scheme: the solution is written as a Fourier multiplier
applied to $\tilde u_0$; the multiplier is compared with the Gaussian $e^{-Dt|\xi|^{2}}$ on a
suitable (diffusive) frequency region and bounded crudely on its complement; and the two
contributions are assembled into an $L^{1}$-estimate of the Fourier-side difference, from which the
$x$-side estimate follows by the inversion bound
\begin{equation}\label{appendix-inversion}
  \|w\|_{L^\infty(\mathbb{R}^n)}
  \le(2\pi)^{-n}\|\tilde w\|_{L^1(\mathbb{R}^n)},
\end{equation}
valid whenever $\tilde w\in L^{1}(\mathbb{R}^n)$ (then $w$ has a continuous representative given by
the inversion integral $w(x)=(2\pi)^{-n}\int_{\mathbb{R}^n}e^{ix\cdot\xi}\tilde w(\xi)\,d\xi$; we
use the convention $\tilde w(\xi)=\int_{\mathbb{R}^n}e^{-ix\cdot\xi}w(x)\,dx$). It is applied with
$w=u_c(\cdot,t)-K(\cdot,t)*u_0$ etc., noting that
$\widetilde{K(\cdot,t)*u_0}(\xi)=e^{-Dt|\xi|^{2}}\tilde u_0(\xi)$.
This scheme is classical in the study of the damped wave equation and goes back to the multiplier
analysis of Matsumura \cite{Matsumura1976}.
The only integral formulas needed are the Gaussian moments
\begin{equation}\label{appendix-gauss-moment}
  \int_{\mathbb{R}^n}e^{-\rho t|\xi|^{2}}|\xi|^{2j}\,d\xi
  =\frac{\omega_{n-1}\,\Gamma\bigl(\frac{n+2j}{2}\bigr)}{2\,(\rho t)^{(n+2j)/2}}
  =O\bigl(t^{-(n+2j)/2}\bigr),
  \qquad \rho>0,\ j=0,1,2,\ t>0,
\end{equation}
where $\omega_{n-1}=2\pi^{n/2}/\Gamma(n/2)$ is the surface area of the unit sphere in
$\mathbb{R}^n$; \eqref{appendix-gauss-moment} follows from polar coordinates and the substitution
$s=\rho t r^{2}$.

  \begin{itemize}
    \item  Case 1: $c\rightarrow \infty$.

   \textbf{Telegraph equation.}
    The solution of the telegraph equation \eqref{Section-4-CTE} in Fourier space is
    $\tilde u_c(\xi,t)=m_c(\xi,t)\tilde u_0(\xi)$, where
    \[
      m_c(\xi,t)=\frac{r_-(\xi)e^{r_+(\xi)t}-r_+(\xi)e^{r_-(\xi)t}}{r_-(\xi)-r_+(\xi)},
    \]
    with $r_\pm(\xi)=-a\pm\sqrt{a^{2}-c^{2}|\xi|^{2}}$.
    Two properties of the multiplier $m_c$ are needed. First, for each $\xi$ the function
    $m(t):=m_c(\xi,t)$ solves $\frac{D}{c^{2}}\ddot m+\dot m+D|\xi|^{2}m=0$ with $m(0)=1$,
    $\dot m(0)=0$; multiplication by $\dot m$ shows that the energy
    $E(t):=\frac{D}{2c^{2}}\dot m(t)^{2}+\frac{D|\xi|^{2}}{2}m(t)^{2}$ is nonincreasing, hence
    $|m(t)|^{2}\le 2E(0)/(D|\xi|^{2})=1$, i.e.,
    \begin{equation}\label{appendix-mc-energy}
      |m_c(\xi,t)|\le1 \qquad\text{for all }\xi\in\mathbb{R}^n,\ c>0,\ t\ge0
    \end{equation}
    (for $\xi=0$ one has $m\equiv1$). Second, write $q:=|\xi|^{2}$ and $v:=c^{2}q/a^{2}$. On the
    region $\{q\le c^{2}/(16D^{2})\}$ one has $v\le1/4$; the binomial series of $\sqrt{1-v}$,
    whose coefficients after the first are all negative, gives
    \[
      \frac{v^{2}}{8}\le1-\sqrt{1-v}-\frac{v}{2}\le\frac{v^{2}}{8(1-v)}\le\frac{v^{2}}{6}
      \qquad\Bigl(0\le v\le\frac14\Bigr).
    \]
    With $\omega:=\sqrt{a^{2}-c^{2}q}=a\sqrt{1-v}\ge\frac{\sqrt3}{2}a$ and $Dq=\frac{av}{2}$ this
    yields
    \[
      r_+=-Dq-a\Bigl(1-\sqrt{1-v}-\frac{v}{2}\Bigr)\le-Dq,
      \qquad
      |r_++Dq|\le\frac{av^{2}}{6}=\frac{4D^{3}}{3}\,\frac{q^{2}}{c^{2}},
    \]
    \[
      \Bigl|\frac{r_-}{r_--r_+}\Bigr|=\frac{a+\omega}{2\omega}\le\frac{2}{\sqrt3},
      \qquad
      \Bigl|\frac{r_+}{r_--r_+}\Bigr|=\frac{a-\omega}{2\omega}=O\Bigl(\frac{q}{c^{2}}\Bigr).
    \]
    Decomposing
    \[
      m_c-e^{-Dqt}
      =\frac{r_-}{r_--r_+}\bigl(e^{r_+t}-e^{-Dqt}\bigr)
      +\frac{r_+}{r_--r_+}\bigl(e^{-Dqt}-e^{r_-t}\bigr),
    \]
    the first bracket satisfies, since $r_+\le-Dq$ and $xe^{-x}\le1/e$,
    \[
      \bigl|e^{r_+t}-e^{-Dqt}\bigr|
      \le t\,|r_++Dq|\,e^{-Dqt}
      \le\frac{4D^{3}}{3}\,\frac{q^{2}t}{c^{2}}\,e^{-Dqt}
      \le\frac{4D^{2}}{3e}\,\frac{q}{c^{2}},
    \]
    and the second satisfies $|e^{-Dqt}-e^{r_-t}|\le e^{-Dqt}\le1$, since $r_-\le-a\le-Dq$ on the
    region. Altogether, on $\{q\le c^{2}/(16D^{2})\}$ and uniformly in $t\ge0$,
    \begin{equation}\label{appendix-mc-expansion}
      \bigl|m_c(\xi,t)-e^{-Dt|\xi|^{2}}\bigr|
      \le\frac{C(D)\,|\xi|^{2}}{c^{2}}.
    \end{equation}
    On the complement $\{q>c^{2}/(16D^{2})\}$ we simply use $|m_c|+e^{-Dtq}\le2$ by
    \eqref{appendix-mc-energy}, and
    \[
      \int_{|\xi|>c/(4D)}2\,|\tilde u_0(\xi)|\,d\xi=O(c^{-N})\qquad\text{for every }N,
    \]
    since $\tilde u_0\in\mathcal{S}(\mathbb{R}^n)$. We now assemble the two regions. Since
    $\tilde u_c(\xi,t)-e^{-Dt|\xi|^{2}}\tilde u_0(\xi)
    =\bigl(m_c(\xi,t)-e^{-Dt|\xi|^{2}}\bigr)\tilde u_0(\xi)$, the region bound
    \eqref{appendix-mc-expansion} and the complement estimate give
    \[
      \bigl\|\tilde u_c(\cdot,t)-e^{-Dt|\cdot|^{2}}\tilde u_0\bigr\|_{L^1(\mathbb{R}^n)}
      \le\frac{C(D)}{c^{2}}\int_{\mathbb{R}^n}|\xi|^{2}\,|\tilde u_0(\xi)|\,d\xi+O(c^{-N}),
      \qquad t\ge0.
    \]
    Applying \eqref{appendix-inversion} to $w=u_c(\cdot,t)-K(\cdot,t)*u_0$ then yields, for all
    sufficiently large $c$ and all $t\ge0$,
    \[
      \bigl\|u_c(\cdot,t)-K(\cdot,t)*u_0\bigr\|_{L^\infty(\mathbb{R}^n)}
      \le\frac{C}{c^{2}},
      \qquad
      C=(2\pi)^{-n}\Bigl[C(D)\,\bigl\||\xi|^{2}\tilde u_0\bigr\|_{L^1}+1\Bigr],
    \]
    where the summand $1$ accounts for the $O(c^{-N})$ term once $c^{-N}\le c^{-2}$.

    \item  Case 2: $t\rightarrow \infty$.

\textbf{Telegraph equation.} The solution of \eqref{Section-4-CTE} in Fourier space is
$\tilde u(\xi,t)=m_c(\xi,t)\tilde u_0(\xi)$ with the multiplier $m_c$ of Case 1 ($c$ now being
fixed), i.e.,
\begin{equation}\label{appendix-Fourier-TE}
  \tilde{u}(\xi,t)=
  \frac{r_-(\xi)e^{r_+(\xi)t}-r_+(\xi)e^{r_-(\xi)t}}{r_-(\xi)-r_+(\xi)}\,
  \tilde{u}_0(\xi),
\end{equation}
where $r_\pm(\xi)=-a\pm\sqrt{a^2-c^2|\xi|^2}$ and $a=mc^2/\hbar$.
For $q:=|\xi|^{2}\le\varepsilon$ (with $\varepsilon>0$ sufficiently small and
$\varepsilon\le a^{2}/c^{2}$), write $\omega(\xi)=\sqrt{a^2-c^2|\xi|^2}$ and expand
\[
  \omega(\xi)=a-\frac{c^2}{2a}|\xi|^2-\frac{c^4}{8a^3}|\xi|^4+O(|\xi|^6),
\]
to obtain
\[
  r_+(\xi)=-D|\xi|^2-\frac{D^2}{2a}|\xi|^4+O(|\xi|^6),\qquad
  r_-(\xi)=-2a+D|\xi|^2+O(|\xi|^4).
\]
To substitute into \eqref{appendix-Fourier-TE}, note that
$r_--r_+=-2\omega$ and
$a/\omega=(1-c^{2}q/a^{2})^{-1/2}=1+\frac{c^{2}q}{2a^{2}}+O(q^{2})$, hence
\[
  \frac{r_-}{r_--r_+}=\frac{a+\omega}{2\omega}=\frac12\Bigl(\frac{a}{\omega}+1\Bigr)
  =1+\frac{c^{2}q}{4a^{2}}+O(q^{2})=1+\frac{D}{2a}q+O(q^{2}),
\]
\[
  \frac{r_+}{r_--r_+}=\frac{a-\omega}{2\omega}=\frac12\Bigl(\frac{a}{\omega}-1\Bigr)
  =\frac{c^{2}q}{4a^{2}}+O(q^{2})=\frac{D}{2a}q+O(q^{2}),
\]
since $\frac{c^{2}}{4a^{2}}=\frac{D}{2a}$. For the first term of the multiplier, since
$r_+=-Dq-\frac{D^{2}}{2a}q^{2}+O(q^{3})$, the estimate $|e^{-w}-1|\le w$ for
$w=\bigl[\frac{D^{2}}{2a}q^{2}+O(q^{3})\bigr]t\ge0$ gives
\[
  \frac{r_-}{r_--r_+}e^{r_+t}
  =\bigl(1+\tfrac{D}{2a}q+O(q^{2})\bigr)e^{-Dqt}\bigl(1+O(q^{2}t)\bigr)
  =e^{-Dqt}\bigl(1+O(q+q^{2}t)\bigr).
\]
For the second term we do not expand the exponential: since
$r_-(\xi)=-a-\omega(\xi)\le-a$ whenever the roots are real,
\[
  -\frac{r_+}{r_--r_+}e^{r_-t}
  =-\bigl(\tfrac{D}{2a}q+O(q^{2})\bigr)e^{r_-t}
  =O\bigl(q\,e^{-at}\bigr),
\]
uniformly in $(\xi,t)$. Combining the two terms gives
\begin{equation}\label{appendix-TE-main}
  \tilde{u}(\xi,t)=
  \tilde{u}_0(\xi)e^{-D|\xi|^2 t}
  \Bigl[1+O\bigl(|\xi|^2+|\xi|^4 t\bigr)\Bigr]
  +O\bigl(|\xi|^2 e^{-at}\bigr)\tilde{u}_0(\xi), \qquad |\xi|^{2}\le\varepsilon,
\end{equation}
where the $O$-constants are uniform in $(\xi,t)$ for $|\xi|^{2}\le\varepsilon$ and $t>0$.
For the complement $q>\varepsilon$ we need a bound on the multiplier $m_c$ which is uniform in
$\xi$. Note that the naive bound $Ce^{-at}$ fails at the critical circle $|\xi|=a/c$, where the
characteristic root is double and $m_c=(1+at)e^{-at}$. The correct uniform bound is
\begin{equation}\label{appendix-secular}
  |m_c(\xi,t)|\le e^{\operatorname{Re}r_+(\xi)\,t}\bigl(1+at\bigr), \qquad \xi\in\mathbb{R}^n,\ t>0.
\end{equation}
Indeed, for $|\xi|<a/c$ the roots are real and
$m_c=e^{r_+t}-r_+\frac{e^{r_+t}-e^{r_-t}}{r_+-r_-}$, where the divided difference satisfies
$\bigl|\frac{e^{r_+t}-e^{r_-t}}{r_+-r_-}\bigr|\le t\,e^{r_+t}$; for $|\xi|>a/c$ the roots are
$r_\pm=-a\pm i\varpi(\xi)$, $\varpi(\xi)=\sqrt{c^2|\xi|^2-a^2}$, and
$m_c=e^{-at}\bigl[\cos(\varpi t)+\frac{a}{\varpi}\sin(\varpi t)\bigr]$ with
$\bigl|\frac{\sin(\varpi t)}{\varpi}\bigr|\le t$; both cases give \eqref{appendix-secular}. Now
$r_+(\xi)\le r_+(\sqrt\varepsilon)=:-\delta_0<0$ for $\varepsilon\le q\le a^{2}/c^{2}$, and
$\operatorname{Re}r_\pm=-a$ for $q>a^{2}/c^{2}$, so \eqref{appendix-secular} yields
\begin{equation}\label{appendix-TE-complement}
  |m_c(\xi,t)|\le(1+at)e^{-\min\{\delta_0,a\}t},\qquad q>\varepsilon.
\end{equation}

We now assemble the two regions. On $\{q\le\varepsilon\}$, \eqref{appendix-TE-main} gives
\[
  \bigl|\tilde u(\xi,t)-e^{-D|\xi|^{2}t}\tilde u_0(\xi)\bigr|
  \le C\bigl(|\xi|^{2}+t|\xi|^{4}\bigr)e^{-D|\xi|^{2}t}|\tilde u_0(\xi)|
  +C|\xi|^{2}e^{-at}|\tilde u_0(\xi)|,
\]
and integrating with $|\tilde u_0(\xi)|\le\|\tilde u_0\|_{L^\infty}\le\|u_0\|_{L^1}$, the Gaussian
moments \eqref{appendix-gauss-moment} with $\rho=D$ and $j=1,2$ give
\[
  \int_{|\xi|^{2}\le\varepsilon}\bigl|\tilde u-e^{-D|\xi|^{2}t}\tilde u_0\bigr|\,d\xi
  \le C\|u_0\|_{L^1}\biggl[
  \frac{\omega_{n-1}\Gamma\bigl(\frac{n+2}{2}\bigr)}{2(Dt)^{(n+2)/2}}
  +\frac{\omega_{n-1}\Gamma\bigl(\frac{n+4}{2}\bigr)}{2D^{(n+4)/2}\,t^{(n+2)/2}}
  +\frac{\omega_{n-1}\varepsilon^{(n+2)/2}}{n+2}\,e^{-at}\biggr]
  =O\bigl(t^{-(n+2)/2}\bigr).
\]
On $\{q>\varepsilon\}$, \eqref{appendix-TE-complement} and
$e^{-D|\xi|^{2}t}\le e^{-D\varepsilon t}$ give
\[
  \int_{|\xi|^{2}>\varepsilon}\bigl|\tilde u-e^{-D|\xi|^{2}t}\tilde u_0\bigr|\,d\xi
  \le\bigl[(1+at)e^{-\min\{\delta_0,a\}t}+e^{-D\varepsilon t}\bigr]\|\tilde u_0\|_{L^1}
  =O\bigl(t\,e^{-\min\{\delta_0,a\}t}\bigr),
\]
which is negligible compared with $t^{-(n+2)/2}$. Adding the two contributions and applying
\eqref{appendix-inversion} to $w=u(\cdot,t)-K(\cdot,t)*u_0$ gives, for
$u_0\in\mathcal{S}(\mathbb{R}^n)$,
\[
  \|u(\cdot,t)-K(\cdot,t)*u_0\|_{L^\infty(\mathbb{R}^n)}
  \le(2\pi)^{-n}\bigl\|\tilde u(\cdot,t)-e^{-Dt|\cdot|^{2}}\tilde u_0\bigr\|_{L^1(\mathbb{R}^n)}
  \le C t^{-(n+2)/2},\qquad t\to\infty.
\]
Thus the telegraph solution $u(x,t)$ converges to the classical heat profile
$K(\cdot,t)*u_0$, and the difference decays as fast as $O(t^{-(n+2)/2})$, which is
faster than the leading order $O(t^{-n/2})$ of the heat kernel as $t\to\infty$.

  \end{itemize}

\vspace{1cm}
{\bf Data Availability} No datasets were generated or analyzed during the current study.

\vspace{0.5cm}


{\bf Conflict of interest} The author declares no conflict of interest.

\end{document}